\RequirePackage{fix-cm}
\documentclass[smallextended]{svjour3}       
\smartqed  
\usepackage{graphicx}
\usepackage[utf8]{inputenc}
\usepackage[english]{babel}
\usepackage{bbm}

\usepackage[nottoc]{tocbibind}
\usepackage{mathrsfs,amsfonts,amssymb,amsmath}

\usepackage{amsthm}
\usepackage{enumerate}
\usepackage{graphicx,cite}
\usepackage{romannum}
\usepackage{hyperref}
\usepackage{authblk}
\usepackage{graphicx}
\usepackage{pst-node}
\usepackage{tikz-cd} 
\newcommand\sbullet[1][.5]{\mathbin{\vcenter{\hbox{\scalebox{#1}{$\bullet$}}}}}
\allowdisplaybreaks

\hypersetup{colorlinks=true, linkcolor=blue, citecolor=red}

\DeclareMathOperator{\diam}{diam}
\DeclareMathOperator{\Leb}{Leb}

\DeclareMathOperator{\interior}{int}

\DeclareMathOperator{\Emb}{Emb}
\DeclareMathOperator{\dist}{Dist}

\newcommand{\N}{\mathbb{N}}

\numberwithin{equation}{section}

\newtheorem{assumption}{Assumption}

\usepackage[square,numbers]{natbib}
\begin{document}
\pagenumbering{arabic}
\setlength{\belowdisplayskip}{0pt}

\title{Back to Boundaries in Billiards}


\author{Leonid A. Bunimovich  \and
        Yaofeng Su
}


\institute{Leonid A. Bunimovich \at
              School of Mathematics, Georgia Institute of Technology,  Atlanta, USA\\
             \email{leonid.bunimovich@math.gatech.edu}           
           \and
           Corresponding author: Yaofeng Su  \at
              School of Mathematics, Georgia Institute of Technology,  Atlanta, USA\\
              \email{yaofeng.su@math.gatech.edu} 
}

\date{Received: date / Accepted: date}

\maketitle

\begin{abstract}
We prove Poisson limit laws for open billiard systems with holes in the boundary of billiard tables. Traditionally some abstract holes in the phase space of a billiard were studied. Holes in the boundary are of an intrinsic interest for billiard systems, especially for applications. Sinai billiards with or without a finite horizon, diamond billiards, and semi-dispersing billiards are considered. However, the emphasis is on focusing billiards with slow decay of correlations, where various new technical difficulties arise.

\end{abstract}

\tableofcontents

\section{Introduction}\ \par

The studies of Poisson approximations for recurrences to small subsets in the phase spaces of chaotic dynamical systems are developed now into a large active area. Another view at this type of problems is a subject of the theory of open dynamical systems, where some positive measure subset of the phase space is named a hole, and the process of escape through the hole is studied. 

In a general setup, one picks a small measure subset (a hole) in the phase space of a hyperbolic (chaotic) ergodic dynamical system and attempts to prove that in the limit, when the measure of the hole approaches zero, the corresponding process of recurrences to the hole converges to a Poisson process. 

This area received an essential boost after L-S.Young papers \cite{Young,Young2}, where a new general framework was introduced for analysis of statistical properties of hyperbolic dynamical systems. This approach employs a representation of the phase space of a dynamical system as a tower (later called a Young tower, a Gibbs-Markov-Young tower, etc), which allows to study dynamics by analysing recurrences to the base of the tower.  
Particularly, in the papers \cite{peneijm,peneetds,Su} the holes, which are the balls, shrinking to a point in the phase space of billiard systems, were studied. The paper \cite{vaientinullset} deals with the holes (shrinking to a curve) in the phase space of Sinai billiards with finite horizons.  The paper \cite{penebacktoball} studied holes within the Sinai's billiard tables, and the work \cite{peneijm} considered the holes near corners of a diamond-shaped billiard table. Such holes correspond to strip-shaped holes in the phase  space  which  are  shrinking to broken line segments. Also, the paper \cite{vaientinullset} should be mentioned, which is dealing with various holes shrinking to null sets in the phase space. Some of the systems, which were considered in this paper \cite{vaientinullset}, can not be modelled by Young towers, but they have milder singularities than the ones in the billiard systems.

Here we present a new development of this area, which, particularly, allows to prove Poisson approximations for various billiard systems with arbitrarily slow decays of correlations. Moreover, we also consider holes located on the boundaries of billiard tables. (Such holes are of special interest for applications, where the holes in the boundary are really made by switching off a particular field generated, e.g., by scanning lasers, and measuring escapes of particles through such holes in the boundary \cite{friedman2001observation, milner2001optical}).  Our main result can be informally described as presented below (a formal description can be found in Theorem \ref{thm}).\begin{proof}[Theorem] For a large class of hyperbolic billiards, the processes of hitting and escaping through a hole in the boundary of a billiard table (generated by billiard maps)  are asymptotically Poissonian.\phantom\qedhere 
\end{proof}The holes within the boundaries are natural to consider for the process of escape in billiard systems. In the phase space such holes tend to straight segments, when a hole within the boundary shrinks to a point. Moreover, exactly such holes are studied in real systems, most notably in  physics experiments \cite{nockel1997ray, nockel1996directional, richter2001wave, friedman2001observation, milner2001optical}. It should be also mentioned that real experiments in physics dealing with billiards usually consider simply connected (i.e., without inner ``holes") billiard tables. In such experiments the particles (which could be considered as noninteracting between themselves) escape through a hole within the boundary. Then a special (measuring) experimental device counts a number of escaping per unit time particles. It is especially important for quantum chaos experiments because the counting devices must be located outside of  billiard tables in order to measure the dynamics of the system, rather than interactions between the system and the measuring device. However, in numerical, rather than real, experiments one can certainly consider a hole of any type.  For instance, Sinai billiards and the Lorentz gas have not simply connected billiard tables. We are not familiar with physics experiments where open billiards of this type were studied. One can make a standard formal assumption of the mathematical theory of open systems that when the particle gets into a hole then it ``disappears". (In real experiments  the particles after escape do not disappear instantly, but continue to propagate outside the billiard table, and interact with other particles, fields, etc). However, for the sake of generality, we do not assume in the present paper that the billiard tables are simply connected.

In comparison to previous papers we obtain here new results, some of which are a kind of unexpected.
\begin{enumerate}
    \item The technique used in \cite{penebacktoball} works  for Sinai billiards only with finite horizons. Our Theorem \ref{thm} is applicable to a larger class of billiard systems. The technique used here is also new for general open dynamical systems.
    \item The approach employed in the papers \cite{penebacktoball, peneijm} requires to verify the so-called short return conditions specifically for each billiard system, while our main Theorem \ref{thm} assumes only some natural general, and easy to verify, conditions.
    \item Unlike \cite{vaientinullset}, Theorem \ref{thm} shows that, surprisingly, the validity of the Poisson approximation for billiard systems does not depend on fast correlations decay, i.e., it holds for any rate of decay of correlations.
    \item The papers \cite{peneetds, Su} require that the contracting (resp. expanding) rates along stable (resp. unstable) manifolds must be sufficiently large. Our results show that this condition can be weakened.
 \end{enumerate}

The structure of the paper is the following. The section \ref{section2} presents some notations, definitions,  the formulation of the main Theorem \ref{thm}, and ideas of the proofs. The sections \ref{section3}-\ref{section8} contain a proof of this Theorem \ref{thm}. We start by giving a general result on a Poisson approximation for general point processes. Then we simplify it, step by step, from the section \ref{section3} to the section \ref{section8}. The section \ref{sectionapp} deals with applications to various billiard systems, especially to slowly mixing billiards, which are the main focus in this paper.


\section{Definitions, Notations, Main Results, and Ideas of the Proofs}\label{section2}
\subsection{Definitions, notations and main results}
We start by introducing some notations
\begin{enumerate}
 \item $C_z$ denotes a constant depending on $z$.
    \item The notation $``a_n \precsim_z b_n"$  ($``a_n=O_{z}(b_n)"$) means that there is a constant $C_z \ge 1$ such that (s.t.) $ a_n \le C_z  b_n$ for all $n \ge 1$, whereas the notation $``a_n \precsim b_n"$ (or $``a_n=O(b_n)"$) means that there is a constant $C \ge 1$ such that $ a_n \le C  b_n$ for all $n \ge 1$. Next, $``a_n \approx_z b_n"$ and $a_n=C_z^{\pm 1}b_n$ mean that there is a constant $C_z \ge 1$ such that  $ C_z^{-1}  b_n \le a_n \le C_z b_n$ for all $n \ge 1$. Further, the notations $``a_n=C^{\pm1} b_n"$ and $``a_n \approx b_n"$ mean that there is a constant $C \ge 1$ such that $ C^{-1}  b_n \le a_n \le C b_n$ for all $n \ge 1$. Finally, $``a_n =o(b_n)"$ means that  $\lim_{n \to \infty}|a_n/b_n|=0$. 
      \item The notation $\mathbb{P}$ refers to a probability distribution on the probability space, where a random variable lives, and $\mathbb{E}$ denotes the expectation of the random variable.
    \item $\mu_A, \Leb_A$ denote measures on a set $A$, unless it is specifically mentioned.
    \item $\mathcal{T}(\mathcal{A})$ denotes a tangent bundle of a (sub)manifold $\mathcal{A}$.
    \item $\mathbb{N}=\{1,2,3,\cdots\}$, $\mathbb{N}_0=\{0,1,2,3,\cdots\}$. 
\end{enumerate}

\begin{definition}[Billiard tables, billiard maps and phase spaces]\ \par
We consider a billiard in a two-dimensional region $Q\subseteq \mathbb{R}^2$ (called a billiard table) with a piece-wise smooth (of class $C^3$) boundary $\partial Q$. Each smooth piece has a uniformly bounded curvature. The boundary $\partial Q$ is equipped with a field of inward unit normal vectors $n(q), q \in \partial Q$. 

A billiard is a dynamical system generated by the motion of a point
particle with the unit velocity inside the region $Q$ being reflected from its
boundary according to the law ``the angle of incidence equals the angle of
reflection". It means that upon reflection the tangent component of the
velocity remains the same, while the normal component changes its sign
according to the rule $v_{+}=v_{-}-2\langle n(q), v_{-}\rangle n(q)$, where $v_{+}$ (resp. $v_{-}$) is the velocity of the particle immediately after (resp. before) reflection.

The phase space of a billiard is the restriction of the unit tangent
bundle of $\mathbb{R}^2$ to $Q$. We will use the standard notation for phase points $x=(q, v)$, where $q$ is the point of the configuration space $Q$ and $v$ is the unit velocity vector. The billiard preserves the Liouville measure $d\nu:=dq dv$ where $dq$ and $dv$ are Lebesgue measures on $Q$ and on the unit
one-dimensional sphere. The corresponding flow will be denoted by $\{S^t\}$. It is customary for billiard-type systems to study instead of $\{S^t\}$ a
dynamical system with discrete time, which is called a billiard map $f$. Denote \[\mathcal{M}=\{x=(q, v), q \in \partial Q, \langle v, n(q)\rangle >0\}.\]  For $x=(q, v) \in \mathcal{M}$, let $\tau (x)$ be
the first positive moment of reflection from the boundary of the billiard
orbit determined by $x$. Then the billiard map is defined by $f(x)=(q', v')=S^{\tau}
x$, so that $q'$ is the point of
the next reflection and $v'$ is the outcoming velocity vector at that point. We call $\mathcal{M}$ a phase space of the billiard map $f$.

Due to the regularities of $\partial Q$,  the set of singular points of the boundary $\partial Q$ is of measure zero, the angle $\phi$ of the velocity vector $v$ varies from $-\pi/2$ to $\pi/2$ at any regular point $q \in \partial Q$, hence $\mathcal{M}=\partial Q \times [-\pi/2,\pi/2]$ almost surely. In what follows we always identify $\mathcal{M}$ as $\partial Q \times [-\pi/2,\pi/2]$ and denote the phase point $x\in \mathcal{M}$ by $(q, \phi) \in \partial Q \times [-\pi/2,\pi/2]$ throughout the paper.

The phase space  $\mathcal{M}=\partial Q \times [-\pi/2,\pi/2]$ is endowed with a natural Riemannian metric $d_{\mathcal{M}}$ and Riemannian volume $\Leb_{\mathcal{M}}$. The billiard map $f$ preserves \[d\mu_{\mathcal{M}}:=(2\Leb_{\partial Q} \partial Q)^{-1}\cos \phi d\phi dq=(2\Leb_{\partial Q} \partial Q)^{-1}\cos \phi \Leb_{\mathcal{M}},\]where $dq$ is the one-dimensional Lebesgue measure on the boundary $\partial Q$ and
$d\phi$ is the one-dimensional Lebesgue (uniform) measure on $[-\pi/2,\pi/2]$. 
\end{definition}

\begin{definition}[An induced system]\label{inducesystem}\ \par
Suppose that there is a fixed subset $X\subseteq \mathcal{M}$ with $\Leb_{\mathcal{M}}(X)>0$. The first return time to $X$ is $R: X\to \mathbb{N}$. We assume that $X$ can be partitioned into countably many connected pieces
\begin{equation}\label{parition}
    X=\bigcup_{i\ge 1} X_i \mod{0},
\end{equation} so that $R$ is constant on each $X_i$ and  \[\Leb_{\mathcal{M}} (\partial X_i)=0,\quad \interior{X_i} \bigcap \interior{X_j}=\emptyset \text{ for } i \neq j.\]

The first return time $R$ induces a first return map  $f^R: X \to X$ and a new dynamical system $(X, f^R)$.

\end{definition}
\begin{definition}[Singularities and (un)stable manifolds]\label{unstablemfd} \par
 Denote by $\mathbb{S}\subseteq X$ the singularity set for $f^R$. For billiard systems, $\mathbb{S}$ has zero Lebesgue measure, and $\mathbb{S}^c \subseteq X$ consists of countably many open connected components. Unstable (resp. stable) manifolds are the connected components of $(\bigcup_{i \ge 0}(f^R)^{i}\mathbb{S})^c$ (resp. $(\bigcup_{i \ge 0}(f^R)^{-i}\mathbb{S})^c$). A closed and connected part of the unstable (resp. stable) manifold will be called an unstable (resp. stable ) disk. We denote each unstable (resp. stable) manifold/disk by $\gamma^u$ (resp. $\gamma^s$), and its tangent vectors by $v^u$ (resp. $v^s$).
\end{definition}

\begin{remark}\label{singular}
The singularity set $\mathbb{S}$ consists of the points in $X $ which are not ``well-behaved". It includes the discontinuities and the points where the map $f^R$ is not differentiable. It may also include other points in $X$ with some ``bad" properties.
\end{remark}
\begin{definition}[Chernov-Markarian-Zhang (CMZ) structures]\label{cmz}\ \par
We say that an induced system $(X, f^R)$ in Definition \ref{inducesystem} is a CMZ structure of the billiard system $(\mathcal{M}, f)$ if there are constants $C>0$ and $\beta \in (0,1)$ such that the following conditions hold
\begin{enumerate}
    \item Hyperbolicity. For any $n\in \mathbb{N}$, $v^u$ and $v^s$,\begin{align*}
        |D(f^R)^n v^u|\ge C\beta^{-n} |v^u|, \quad  |D(f^R)^{n} v^s|\le C\beta^n |v^s|,
    \end{align*}where $|\cdot|$ is the Riemannian metric induced from $\mathcal{M}$ to (un)stable manifolds.
    \item SRB measures and u-SRB measures. $(X, f^R,\mu_X) $ is a K-system where $\mu_X:=\frac{\mu_{\mathcal{M}}|_X}{\mu_{\mathcal{M}}(X)}$. A corresponding measurable K-partition consists of smooth pieces of stable manifolds. Moreover, conditional distributions on $\gamma^u$ (say $\mu_{\gamma^u}$) are absolutely continuous w.r.t. Lebesgue measure $\Leb_{\gamma^u}$ on $\gamma^u$.
    \item Distortion bounds. Let $d_{\gamma^u}(\cdot, \cdot)$ be the distance measured along $\gamma^u$. By $\det D^uf^R$ we denote the Jacobian of $Df^R$ along the unstable manifolds. Then, if $x, y\in X$ belong to a $\gamma^u$, such that $(f^R)^n$ is smooth on $\gamma^u$, the following relation holds
    \begin{align*}
        \log \frac{\det D^u(f^R)^n(x)}{\det D^u(f^R)^n(y)}\le \psi\left[d_{\gamma^u}\Big((f^R)^nx,(f^R)^ny\Big)\right],
    \end{align*}where $\psi(\cdot)$ is some function, which does not depend on $\gamma^u$, and $\lim_{s\to 0^+}\psi(s)=0$.
    \item Bounded curvatures. The curvatures of all $\gamma^u$ are uniformly
bounded by $C$.
\item Absolute continuity. Consider a holonomy map $h: \gamma^u_1 \to \gamma^u_2$,  which maps a point $x \in \gamma^u_1$ to the point $h(x)\in \gamma^u_2$, such that both $x$ and $h(x)$ belong to the same $\gamma^s$. We assume that the holonomy map satisfies the following relation \begin{align*}
    \frac{\det D^u(f^R)^n(x)}{\det D^u(f^R)^n\big(h(x)\big)}=C^{\pm 1} \text{ for all }n \ge 1 \text{ and }x\in \gamma^u_1,
\end{align*}
\item Growth lemmas. There exist $N \in \mathbb{N}$, sufficiently small $\delta_0>0$ and
constants  $ \kappa, \sigma > 0$ which satisfy the following condition. For any
sufficiently small $\delta>0$ and for any disk on a smooth unstable manifold  $\gamma^u$ with $\diam \gamma^u \le \delta_0$, denote by $U_{\delta} \subseteq \gamma^u$ a $\delta$-neighborhood of the subset $\gamma^u \bigcap \bigcup_{0\le i \le N}(f^R)^{-i}\mathbb{S}$ within the set $\gamma^u$. Then there exists an open subset $V_{\delta} \subseteq \gamma^u \setminus U_{\delta}$, such that $\Leb_{\gamma^u}\big(\gamma^u \setminus (U_{\delta}\bigcup V_{\delta})\big)=0$, and for any $\epsilon >0$ \begin{gather*}
    \Leb_{\gamma^u}(r_{V_{\delta},N}<\epsilon)\le 2\epsilon \beta+\epsilon C \delta_0^{-1} \Leb_{\gamma^u} (\gamma^u),\\
    \Leb_{\gamma^u}(r_{U_{\delta},0}< \epsilon)\le C \delta^{-\kappa} \epsilon,\\
    \Leb_{\gamma^u}(U_{\delta})\le C \delta^{\sigma},
\end{gather*}where $r_{U_{\delta},0}(x):=d_{\gamma^u}(x, \partial U_{\delta})$, $r_{V_{\delta}, N}(x):=d_{(f^R)^N\gamma^u}\big((f^R)^Nx, \partial (f^R)^NV_{\delta}(x)\big)$, and $V_{\delta}(x)$ is a connected component of $V_{\delta}$, which contains $x$.

\item Finiteness: $\int R d\mu_X< \infty$.
\item Mixing: $\gcd \{R\}=1$.

\end{enumerate}

\end{definition}

\begin{remark}
 In the growth lemmas a positive integer $N$ is usually chosen as a sufficiently large number. From the paper \cite{gurevich}, the conditions that $\gcd{R}=1$ and K-mixing of $f^R$ guarantee that $f$ is also K-mixing.
\end{remark}

Consider now the first return tower \begin{align*}
    \Delta:=\{(x,n)\in X \times \{0,1,2,\cdots\}: n < R(x)\}.
\end{align*}

A dynamics $F: \Delta \to \Delta$ is defined as $F(x, n)=(x, n+1)$ if $n+1\le R(x)-1$ and $F(x,n)=(f^Rx, 0)$ if $n=R(x)-1$. The projection $\pi:\Delta \to \mathcal{M}$ is defined by $\pi(x,n):=f^n(x)$ as $\pi \circ F=f \circ \pi$. Finally we introduce projections $\pi_{X}: \Delta \to X$ and $\pi_{\mathbb{N}}: \Delta \to \mathbb{N}_0$, so that for any $(x,n) \in \Delta$ \begin{equation}\label{projections}
    \pi_{X}(x,n)=x,\quad \pi_{\mathbb{N}}(x,n)=n.
\end{equation}

Extend now $\mu_X$ from $X$ to $\Delta$ as \begin{align*}
    \mu_{\Delta}:=(\int R d\mu_X)^{-1}\sum_{j}(F^j)_{*}\big(\mu_X|_{\{R>j\}}\big),
\end{align*} which reproduces the invariant probability measure on $\mathcal{M}$ \begin{align*}
\mu_{\mathcal{M}}=(\pi)_{*}\mu_{\Delta}.  \end{align*}

We identify $\Delta_0:=(X \times \{0\})\bigcap \Delta$ with $X$, $\mu_{\Delta_0}$ with $\mu_X$ and $F^R$ with $f^R$. Therefore $\pi: X \to X$ is the identity map.

Note that $\pi:\Delta \to \mathcal{M}$ is bijective. Thus $(\Delta,F)$ is identical to $(\mathcal{M},f)$, and $(X, f^R)=(\Delta_0, F^R)$ is a CMZ structure of $(\Delta, F)$.

\begin{remark}\ \par
\begin{enumerate}
    \item If $R=1$, then $X=\mathcal{M}$, meaning that $(\mathcal{M}, f, \mu_{\mathcal{M}})$ has a CMZ structure.
    \item It follows from \cite{Chernovjsp, CZnon} that $(X, f^R, \mu_X)$ can be modelled by a hyperbolic Young tower \cite{Young}.
    \item It follows from \cite{CZnon,CZcmp} that the mixing rates for the dynamical system $(\mathcal{M}, f, \mu_{\mathcal{M}})$ are determined by the decay rate of $\mu_X(R>n)$. However, we do not use this fact in the present paper.
    \item $(\mathcal{M}, f, \mu_{\mathcal{M}})$ is a K-system (and therefore mixing) because of the condition that $f^R$ is K-mixing and $\gcd \{R\}=1$.
\end{enumerate}
\end{remark}
\begin{definition}[Holes and dynamical point processes]\label{dynamicalpointprocess} \par
Throughout the paper, a hole within the boundary $\partial Q$ is an open disk $B_{r}(q)$ with radius $r$ and the center at a regular point $q \in \partial Q$ of the boundary of a billiard table. 

We define now a dynamical point process $\mathcal{N}^{r,q}$ on $\mathbb{R}^{+}\bigcup\{0\}$. For any measurable $A\subseteq \mathbb{R}^{+}\bigcup\{0\}$, and any $x \in \mathcal{M}$, \begin{align*}
    \mathcal{N}^{r,q}(x)(A):&=\#\{i\ge 0: f^i(x) \in B_r(q) \times [-\pi/2,\pi/2], \quad i \cdot \mu_{\mathcal{M}}\big(B_r(q) \times [-\pi/2,\pi/2]\big) \in A\}\\
    &=\sum_{i \cdot \mu_{\mathcal{M}}(B_r(q)\times [-\pi/2,\pi/2])\in A}\mathbbm{1}_{B_r(q)\times [-\pi/2,\pi/2]}\circ f^i(x).
\end{align*}

We will usually drop the symbol $x$ and write 
 \begin{align*}
    \mathcal{N}^{r,q}(A)=\sum_{i \cdot \mu_{\mathcal{M}}(B_r(q)\times [-\pi/2,\pi/2])\in A}\mathbbm{1}_{B_r(q)\times [-\pi/2,\pi/2]}\circ f^i.
\end{align*}

By using $\mu_{\mathcal{M}}:=(\pi)_{*}\mu_{\Delta}$, we get 
 \begin{align*}
    \mathcal{N}^{r,q}(A)=\sum_{i \cdot \mu_{\Delta}(A_r)\in A}\mathbbm{1}_{A_r}\circ F^i,
\end{align*}where $A_{r}:=\pi^{-1}(B_r(q)\times [-\pi/2,\pi/2])$.
\end{definition}
\begin{remark}
Following the theory of point processes (see e.g. page 226 of \cite{kallenberg}) we have that 
\[\mathcal{N}^{r,q}(x)=\sum_{i\ge 0: f^i(x) \in B_r(q) \times [-\pi/2,\pi/2]} \delta_{i\cdot \mu_{\mathcal{M}}(B_r(q)\times [-\pi/2, \pi/2]) },\] where $\delta$ is a Dirac measure. Hence, $\mathcal{N}^{r,q}$ is a random counting measure, e.g., $\mathcal{N}^{r,q}(x)[0,1]$ counts the number of $i \in [0, \frac{1}{\mu_{\mathcal{M}}\{B_r(q)\times [-\pi/2, \pi/2]\}}]$, such that $f^i(x)$ lies in $B_r(q)\times [-\pi/2, \pi/2]$.
\end{remark}

\begin{definition}[Poisson point processes]\ \par
We say that $\mathcal{P}$ is a Poisson point process on $\mathbb{R}^+\bigcup\{0\}$ if
\begin{enumerate}
    \item $\mathcal{P}$ is a random counting measure on $\mathbb{R}^+\bigcup\{0\}$.
    \item $\mathcal{P}(A)$ is a Poisson-distributed random variable for any Borel set $A \subseteq \mathbb{R}^+\bigcup\{0\}$.
    \item If $A_1, A_2, \cdots, A_n \subseteq \mathbb{R}^+\bigcup\{0\}$ are pairwise disjoint, then  $\mathcal{P}(A_1), \cdots, \mathcal{P}(A_n)$ are independent.
    \item $\mathbb{E}\mathcal{P}(A)=\Leb_{\mathbb{R}^{+}\bigcup \{0\}}(A)$ for any Borel set $A\subseteq \mathbb{R}^+\bigcup\{0\}$.
\end{enumerate}

\end{definition}

\begin{definition}[Poisson approximations]\label{pdef} \par
We say that $\mathcal{N}^{r,q} \to_d \mathcal{P}$ if  $\mathcal{N}^{r,q}f \to_d \mathcal{P}f$ for any $f\in C^{+}_c(\mathbb{R}^+\bigcup\{0\})$, i.e.,
\begin{align*}
    \lim_{r\to 0}\int \exp{(-t\cdot  \mathcal{N}^{r,q}f)}d\mu_{\mathcal{M}}=\int \exp{(-t\cdot \mathcal{P}f)}d\mathbb{P} \text{ for all }t>0,
\end{align*} where $C^{+}_c(\mathbb{R}^+\bigcup \{0\})$ is the space consisting of positive continuous functions with a compact support, defined on $[0,\infty)$.

This is equivalent to the relation \begin{align*}
    (\mathcal{N}^{r,q}I_1,\mathcal{N}^{r,q}I_2, \cdots, \mathcal{N}^{r,q}I_k) \to_d (\mathcal{P}I_1, \mathcal{P}I_2, \cdots, \mathcal{P}I_k) 
\end{align*}for any $k \in \mathbb{N}$ and any bounded intervals $I_1, \cdots I_k\subseteq{R}^+\bigcup\{0\}$. (See, e.g., Theorem 16.16 in \cite{kallenberg}).

Thus, the limit distribution of $\mathcal{N}^{r,q}$ is Poisson, when the disk $B_r(q)$ shrinks to a regular point on the boundary $\partial Q$ of a billiard table.

\end{definition}

\begin{definition}[Sections and quasi-sections]\ \par
Recall that $\pi_X: \Delta \to X $ is defined as $\pi_X(x,n)=x$. We say that $B_r(q) \times [-\pi/2, \pi/2]\subseteq \mathcal{M}$ is a section if $\pi_X: \pi^{-1} \big(B_r(q) \times [-\pi/2, \pi/2]\big) \to X$ is injective for any sufficiently small $r>0$. Further, $B_r(q) \times [-\pi/2, \pi/2] \subseteq \mathcal{M}$ is a quasi-section if for any sufficiently small $r>0$ there is a measurable set $S_r \subseteq B_r(q) \times [-\pi/2, \pi/2]$, such that  $\mu_{\mathcal{M}}[\big(B_r(q) \times [-\pi/2, \pi/2]\big)\setminus S_r]=o\big(\mu_{\mathcal{M}}\big[B_r(q) \times [-\pi/2, \pi/2]\big]\big)$,  and $\pi_X: \pi^{-1} S_r \to X$ is injective. In this case, we also refer to $S_r$ as a section in $B_r(q) \times [-\pi/2, \pi/2]$.  We will explicitly write what it is for a given example throughout the present paper.
\end{definition}
\begin{remark}\label{remarkonsection}
In the applications to two-dimensional billiards, $B_r(q)\times [-\pi/2,\pi/2]$ is a strip in $\mathcal{M}$, and $ B_r(q)\times [-\pi/2,\pi/2]\setminus S_r$ is usually a union of finitely many rectangles, whose measures are of order $O(r^2)$, see section \ref{sectionapp}. To avoid unnecessary complications, we always assume that $ B_r(q)\times [-\pi/2,\pi/2]\setminus S_r$ has a regular shape, e.g., as a union of finitely many rectangles.
\end{remark}

\begin{assumption}[\textbf{Geometric assumptions}]\label{assumption}\ \par
\begin{enumerate}
    \item For a.e. $q \in \partial Q$ the set $B_r(q)\times [-\pi/2, \pi/2]$ is a quasi-section.
    \item $\bigcup_{i\ge 1}\partial X_i\subseteq \mathbb{S}$ (see the definition of $X_i$ in (\ref{parition})).
    \item There are constants $C>0$ and $\alpha \in (0,1]$, such that for any $\gamma^k$, $k=u$ or $s$ (the condition $\bigcup_{i\ge 1}\partial X_i\subseteq \mathbb{S}$ implies that $\gamma^k \subseteq X_i$ for some $i\ge 1$),  and for any $x,y \in \gamma^k$, \begin{align*}
        d_{f^j\gamma^k}(f^jx, f^jy)\le C d_{\gamma^k}(x,y)^\alpha \text{ for all }j\in[0,R(x)).
    \end{align*}
    \item There exist two cones $C^u,C^s \subseteq \mathcal{T}(\mathcal{M})$, such that \[\dim(\interior{C^u}\bigcap \interior{C^s})<1, \quad (Df)C^u \subseteq C^u, \quad (Df)^{-1}C^s \subseteq C^s,\]
    and for all $n\ge 1$ and $\Leb_{\partial Q}$-a.e. $q \in \partial Q$
    \[\dim \big(\mathcal{T}(\{q\}\times [-\pi/2, \pi/2])\bigcap \interior{C^u}\big)<1, \quad \dim \big(\mathcal{T}(\{q\}\times [-\pi/2, \pi/2])\bigcap \interior{C^s}\big)<1,\]
    \[(Df)^n\mathcal{T}(\{q\}\times [-\pi/2, \pi/2]) \subseteq{\interior{C^u}}, \quad (Df)^{-n}\mathcal{T}(\{q\}\times [-\pi/2, \pi/2]) \subseteq{\interior{C^s}},\]where $\interior{C^u}$ (resp. $\interior{C^s}$) is the interior of $C^u$ (resp. $C^s$).
\end{enumerate}
\begin{remark}\ \par
\begin{enumerate}
    \item We will present later an easy-to-implement scheme to verify the existence of $C^u, C^s$ for billiard systems. The cones $C^u, C^s$ should be transversal (but not necessarily uniformly transversal). This condition on $C^u, C^s$ in Assumption \ref{assumption} is called an  aperiodic condition, because it (almost surely) rules out the periodic orbits (see Lemma \ref{aperiodic}).
    \item The H\"older condition is natural, and it is traditionally used for hyperbolic systems, and, particularly, for billiards.
\end{enumerate}
 
\end{remark}
\end{assumption}
\begin{theorem}[Poisson limit laws]\label{thm}\ \par
Suppose that dynamical system $(\mathcal{M}, f)$ has a CMZ structure $(X, f^R)$ (see Definitions \ref{cmz} and \ref{inducesystem}), and the  Assumption \ref{assumption} holds. Then, when $r\to 0$, we have $\mathcal{N}^{r,q} \to_d \mathcal{P}$ (see Definition \ref{pdef}) for $\Leb_{\partial Q}$-a.s. $q\in\partial Q$.
\end{theorem}

\begin{corollary}[First hitting]\ \par
Under the same conditions as in Theorem \ref{thm} consider the moment of time when the first hitting (passage) of the hole occurs, i.e., $\tau_{r,q}(x):=\inf\left\{n \ge 1: f^n(x) \in B_r(q)\times [-\pi/2, \pi/2]\right\}$ for any $x \in \mathcal{M}$. Then for any $t>0$ and almost every $q \in \partial Q$, the following relation holds for the first hitting probability 
\begin{equation*}
    \lim_{r\to 0}\mu_{\mathcal{M}}\Big\{\tau_{r,q}>t\big/\mu\Big(B_r(q)\times [-\pi/2, \pi/2]\Big)\Big\}=e^{-t}.
\end{equation*}
\end{corollary}
\begin{proof}
Clearly \[\mu_{\mathcal{M}}\left\{\tau_{r,q}>t/\mu\big(B_r(q)\times [-\pi/2,\pi/2]\big)\right\}=\mu_{\mathcal{M}}\left\{\mathcal{N}^{r,q}\left[0,t/\mu\big(B_r(q)\times [-\pi/2, \pi/2]\big)\right]=0\right\},\] Apply now Theorem \ref{thm}, and the corollary holds.
\end{proof}

\subsection{Related work and comparison with our result}
\begin{enumerate}
    \item Unlike \cite{vaientinullset, penebacktoball}, Theorem \ref{thm} claims that validity of Poisson limit laws does not depend on the rate of correlations decay.  In other words, decay of correlations can be arbitrarily slow as long as the first return time $R \in L^1$, i.e., $R$ must be just integrable.
    \item The technique used in \cite{penebacktoball}  works only for Sinai billiards with finite horizons, while our approach works for arbitrarily slow mixing billiards, and particularly for Sinai billiards with infinite horizons.
    \item The papers \cite{peneetds, peneijm} established the (spatio-temporal) Poisson limit law under several conditions including that contracting (resp. expanding) rates $\alpha$ along stable (resp. unstable) manifolds and $\lim_{r \to 0}\frac{\log \mu_{\mathcal{M}}\{B_r(q)\times [-\pi/2, \pi/2]\}}{\log r}$ are sufficiently large, i.e., $\alpha \cdot \lim_{r \to 0}\frac{\log \mu_{\mathcal{M}}\{B_r(q)\times [-\pi/2, \pi/2]\}}{\log r}>1$. These conditions fail for billiards with focusing components of the boundary (e.g. if $\alpha=1$ for stadium-type billiards). 
    \item One of the main difficulties in proving Poisson limit laws, related to short returns, was outlined in \cite{peneetds}. The papers \cite{vaientinullset, peneijm, penebacktoball,  peneetds} handled it by inspecting the original dynamics $f$, which leads to a requirement of a fast mixing rate. Our approach via an inducing method allows to restrict a hole in the phase space to a good set, and it works for systems with arbitrarily slow mixing.
    
    Another challenge for proving Poisson limit laws, called a corona, comes from the condition $\alpha \cdot \lim_{r \to 0}\frac{\log \mu_{\mathcal{M}}\{B_r(q)\times [-\pi/2, \pi/2]\}}{\log r}>1$ (see \cite{peneetds}). The failure of $\alpha \cdot \lim_{r \to 0}\frac{\log \mu_{\mathcal{M}}\{B_r(q)\times [-\pi/2, \pi/2]\}}{\log r}>1$ in this paper causes essential difficulties in proving Poisson limit laws. Our techniques, which combine the inducing method with an approximation method, allow to overcome this challenge. 
    \item The papers \cite{ peneijm} studied various holes in $\mathcal{M}$ for Sinai billiards with bounded horizons and diamond billiards. Although in the present paper we consider a special type of holes, i.e., the ones in $B_r(q)\times [-\pi/2, \pi/2]$, our technique can be adapted for more general holes. The holes we consider here are the most natural for billiard systems and their applications. A consideration of a general type holes will make the paper much longer and even more technical.
    \item The method used in \cite{peneijm} requires,  besides the existence of CMZ structure, some special properties of Sinai billiards with finite horizons and diamond billiards to hold, (e.g. see page 657 of \cite{peneijm}). Our method only uses the assumption of the existence of CMZ structure, and can be applied to a large class of billiards. 
    \item Unlike \cite{Su}, our result does not provide convergence rates to Poisson limit laws. We believe though that it is possible to get convergence rates by placing more restrictive conditions on the first return time $R$ and singularities $\mathbb{S}$. We expect to deal with convergence rates in another paper.
\end{enumerate}

\subsection{An informal description of the scheme of proof of Theorem \ref{thm}}
\begin{enumerate}
\item The lemmas in section \ref{section3} reduce the point process $\mathcal{N}^{r,q}$ to a new point process generated by a section $S_r$ in $B_r(q) \times [-\pi/2, \pi/2]$.
 \item In section \ref{section4} Lemma \ref{lifting}, we lift this new point process to $\Delta$, which becomes yet another new point process $\mathcal{N}_s^r$.  Then we just need to prove a Poisson approximation for $\mathcal{N}_s^r$ defined on $\Delta$. 
 \item In section \ref{section4} we truncate the tower $\Delta$ as  $\Delta_m$ with height $m$. Since $\lim_{m \to \infty}\Delta_m=\Delta$, we can restrict $\mathcal{N}_s^r$ to $\Delta_m$, in order to investigate the Poisson approximation.
 \item For each $m \ge 1$, the truncated tower $\Delta_m$ can ``induce" a ``better" hyperbolic system and a ``better" point process $\mathcal{N}_i^r$. Lemma \ref{inducing} shows that, to study a  Poisson approximation for $\mathcal{N}_s^r$, restricted to $\Delta_m$, one just needs to prove a Poisson approximation for $\mathcal{N}_i^r$ on $\Delta_m$.
 \item Lemma \ref{approximation} shows that if, for any large $m \ge 1$, a Poisson approximation holds for $\mathcal{N}_i^r$, which is defined on $\Delta_m$, then the Poisson approximation for $\mathcal{N}^{r,q}$ in Theorem \ref{thm} holds. Since we study a specific large $m\ge 1$, a proof of Poisson approximations for $\mathcal{N}_i^r$  on $\Delta_m$ does not require uniformity in $m\ge 1$.
\item $\Delta_m$ can be modeled by a non-mixing hyperbolic Young tower as shown in Section \ref{section5}. 
\item The sections \ref{section6}, \ref{sectionshortreturn} and \ref{section8} apply the non-mixing hyperbolic Young tower of $\Delta_m$ for proving a Poisson approximation for $\mathcal{N}_i^r$ for each large $m\ge 1$.

 In section \ref{section6}, we obtain two conditions (\ref{correlationinpoissonlaw}) and (\ref{shortreturn}), which will be verified to prove Poisson approximations for $\mathcal{N}_i^r$ on $\Delta_m$.

 The sections \ref{sectionshortreturn} and \ref{section8} deal with  (\ref{correlationinpoissonlaw}) and (\ref{shortreturn}) respectively for  $B_r(q)\times [-\pi/2, \pi/2]$ in the phase space. The idea is to compare a measure related to $B_r(q)\times [-\pi/2, \pi/2]$ with a measure of a family of hyperbolic product sets which have non-empty intersections with $B_r(q)\times [-\pi/2, \pi/2]$. 
 The estimates in this procedure do not require uniformity in $m$. The Assumption \ref{assumption} plays a crucial role in these estimates.
\end{enumerate}

\section{Poisson limit laws: from quasi-sections to sections}\label{section3}
To prove Poisson limit laws (see Definition \ref{pdef}), we will need one result from \cite{kallenberg}. To simplify notations, we denote $\Leb_{\mathbb{R}^+\bigcup \{0\}}$ by $\Leb$ throughout this section.
\begin{lemma}[See Proposition 16.17 of \cite{kallenberg}]\label{kallen}\ \par
For any $q \in \partial Q$ convergence $\mathcal{N}^{r,q} \to_d \mathcal{P}$ holds if
 \begin{enumerate}
     \item For any compact set $K\subseteq [0,\infty)$, $\limsup_{r\to 0}\int \mathcal{N}^{r,q}(K)d\mu_{\mathcal{M}}\le \Leb(K)$.
     \item For any disjoint bounded intervals $J_1, J_2, \cdots, J_n\subseteq [0,\infty)$, \begin{align}\label{survival}
         \lim_{r \to 0} \mu_{\mathcal{M}}\Big(\mathcal{N}^{r,q}(\bigcup_{i \le n}J_i)=0\Big)=\mathbb{P}\Big(\mathcal{P}(\bigcup_{i \le n}J_i)=0\Big). 
     \end{align}
 \end{enumerate}
\end{lemma}
Now we will verify these conditions.
\begin{lemma}\label{poissoncondition1}
For any compact set $K\subseteq [0,\infty)$, $\limsup_{r\to 0}\int \mathcal{N}^{r,q}(K)d\mu_{\mathcal{M}}\le \Leb(K)$.
\end{lemma}
\begin{proof}
Assume that $K\subseteq [0,T)$ for sufficiently large $T>0$. Then $[0,T)\setminus K=\bigcup_{i\ge 1} (a_i,b_i)$ for some disjoint open intervals $(a_i,b_i)$. Therefore, $K=\bigcap_{N\ge 1}\bigcap_{i \le N}[0,T)\bigcap (a_i,b_i)^c$. Let $K_N:=\bigcap_{i \le N}[0,T)\bigcap (a_i,b_i)^c$, which is a disjoint union of finitely many intervals $J_1, J_2, \cdots, J_{i_N}$. Hence\begin{align*}
    \int \mathcal{N}^{r,q}(J_i)d\mu_{\mathcal{M}}&=\int \sum_{j \cdot \mu_{\mathcal{M}}\big(B_r(q)\times [-\pi/2, \pi/2]\big)\in J_i}\mathbbm{1}_{B_r(q)\times [-\pi/2,\pi/2]}\circ f^jd\mu_{\mathcal{M}}\\
    &\le \big[1+\Leb(J_i)/\mu_{\mathcal{M}}\big(B_r(q)\times [-\pi/2,\pi/2]\big)\big]\mu_{\mathcal{M}}\big(B_r(q)\times [-\pi/2,\pi/2]\big).
\end{align*} 

Therefore
\begin{align*}
    \int \mathcal{N}^{r,q}(K)d\mu_{\mathcal{M}}&\le \int \mathcal{N}^{r,q}(K_N)d\mu_{\mathcal{M}}=\int \sum_{j \cdot \mu_{\mathcal{M}}(B_r(q)\times [-\pi/2, \pi/2])\in K_N}\mathbbm{1}_{B_r(q)\times [-\pi/2,\pi/2]}\circ f^jd\mu_{\mathcal{M}}\\
    &\le i_N\cdot \mu_{\mathcal{M}}\big(B_r(q)\times [-\pi/2,\pi/2]\big)+\Leb(K_N).
\end{align*}

Thus $\limsup_{r \to 0}\int \mathcal{N}^{r,q}(K)d\mu_{\mathcal{M}} \le \Leb(K_N)$. We conclude the proof by letting $N\to \infty$.
\end{proof}

Now we study the relation (\ref{survival}). Suppose that $B_r(q)\times [-\pi/2, \pi/2]$ is a quasi-section for some small $r>0$. Hence there is a section $S_r\subseteq B_r(q)\times [-\pi/2, \pi/2]$, such that $\mu_{\mathcal{M}}\big(B_r(q)\times [-\pi/2, \pi/2]\setminus S_r\big)=o\big(\mu_{\mathcal{M}}(B_r(q)\times [-\pi/2, \pi/2])\big)$ for any sufficiently small $r>0$. Define \begin{align}\label{sectionpointprocess}
    \mathcal{N}_s^{r,q}(A):=\sum_{i \cdot \mu_{\mathcal{M}}(S_r)\in A}\mathbbm{1}_{S_r}\circ f^i,
\end{align} where $A$ is a measurable set in $[0,\infty)$. Then we can further modify the relation (\ref{survival}) as follows. 

\begin{lemma}[From quasi-sections to sections]\label{nonsectiontosection}\ \par For any disjoint bounded intervals $J_1, J_2, \cdots, J_n\subseteq [0,\infty)$, \begin{align*}
         \lim_{r \to 0} \mu_{\mathcal{M}}\Big(\mathcal{N}_s^{r,q}(\bigcup_{i \le n}J_i)=0\Big)=\lim_{r \to 0} \mu_{\mathcal{M}}\Big(\mathcal{N}^{r,q}(\bigcup_{i \le n}J_i)=0\Big). 
     \end{align*}
\end{lemma}
\begin{proof} Denote $J'_i:=\{x:x\cdot  \mu_{\mathcal{M}}\big(B_r(q)\times [-\pi/2, \pi/2]\big)\in J_i\}$, $J''_i:=\{x:x\cdot \mu_{\mathcal{M}}(S_r)\in J_i\}$. Now we estimate the difference between  $\mu_{\mathcal{M}}\big(\mathcal{N}_s^{r,q}(\bigcup_{i \le n}J_i)=0\big)$ and $\mu_{\mathcal{M}}\big(\mathcal{N}^{r,q}(\bigcup_{i \le n}J_i)=0\big)$. \begin{align}
    \Big|\mu_{\mathcal{M}}&\Big(\mathcal{N}_s^{r,q}(\bigcup_{i \le n}J_i)=0\Big)- \mu_{\mathcal{M}}\Big(\mathcal{N}^{r,q}(\bigcup_{i \le n}J_i)=0\Big)\Big|\nonumber\\
    &= \Big|\mu_{\mathcal{M}}\Big(\bigcap_{j\in \bigcup J''_i}f^{j}\notin S_r\Big)- \mu_{\mathcal{M}}\Big(\bigcap_{j\in \bigcup J'_i}f^{j}\notin B_r(q)\times [-\pi/2, \pi/2]\Big)\Big|\nonumber\\
    &\le \Big|\mu_{\mathcal{M}}\Big(\bigcap_{j\in \bigcup J''_i}f^{j}\notin S_r\Big)- \mu_{\mathcal{M}}\Big(\bigcap_{j\in \bigcup J'_i}f^{j}\notin S_r\Big)\Big|\nonumber\\
    &\quad +\Big|\mu_{\mathcal{M}}\Big(\bigcap_{j\in \bigcup J'_i}f^{j}\notin S_r\Big)- \mu_{\mathcal{M}}\Big(\bigcap_{j\in \bigcup J'_i}f^{j}\notin B_r(q)\times [-\pi/2, \pi/2]\Big)\Big|.\nonumber
    \end{align}
    
    Using $\big(\bigcap_{j\in \bigcup J'_i}f^{j}\notin S_r\big)\setminus \big(\bigcap_{j\in \bigcup J'_i}f^{j}\notin B_r(q)\times [-\pi/2, \pi/2]\big) =\big(\bigcap_{j\in \bigcup J'_i}f^{j}\notin S_r\big) \bigcap \big(\bigcup_{j\in \bigcup J'_i}f^{j}\in B_r(q)\times [-\pi/2, \pi/2]\big)\subseteq \bigcup_{j\in \bigcup J'_i}\{f^{j}\in B_r(q)\times [-\pi/2, \pi/2]\setminus S_r\}$ we can continue the estimate above as
    \begin{align}
    &\le \Big|\mu_{\mathcal{M}}\Big(\bigcap_{j\in \bigcup J''_i}f^{j}\notin S_r\Big)- \mu_{\mathcal{M}}\Big(\bigcap_{j\in \bigcup (J'_i\bigcup J''_i)}f^{j}\notin S_r\Big)\Big|\nonumber\\
    &\quad +\Big|\mu_{\mathcal{M}}\Big(\bigcap_{j\in \bigcup J'_i}f^{j}\notin S_r\Big)- \mu_{\mathcal{M}}\Big(\bigcap_{j\in \bigcup (J'_i\bigcup J''_i)}f^{j}\notin S_r\Big)\Big|\nonumber\\
    &\quad +\mu_{\mathcal{M}}\Big(\bigcap_{j\in \bigcup J'_i}f^{j}\notin S_r \bigcap \bigcup_{j\in \bigcup J'_i}f^{j}\in B_r(q)\times [-\pi/2, \pi/2]\Big)\nonumber\\
    &\le \mu_{\mathcal{M}}\Big(\bigcap_{j\in \bigcup J''_i}f^{j}\notin S_r\bigcap \bigcup_{j\in \bigcup (J'_i\bigcup J''_i)}f^{j}\in S_r\Big)\nonumber\\
    &\quad +\mu_{\mathcal{M}}\Big(\bigcap_{j\in \bigcup J'_i}f^{j}\notin S_r\bigcap \bigcup_{j\in \bigcup (J'_i\bigcup J''_i)}f^{j}\in S_r\Big)\nonumber\\
    & \quad +\mu_{\mathcal{M}}\Big(\bigcup_{j\in \bigcup J'_i}f^{j}\in B_r(q)\times [-\pi/2, \pi/2]\setminus S_r\Big)\label{1}.
\end{align}

 There is a small $r_J>0$ depending on $J_1,J_2, \cdots J_n$ such that, for any $r\in (0,r_J)$ and all $i\le n$, the interval $J'_i$ is the only possible one that intersect $J''_i$ and $\Leb J'_i\ge 1$, $\Leb J''_i \ge 1$. Therefore\begin{align*}
     &\Big\{\bigcap_{j\in \bigcup J''_i}f^{j}\notin S_r\bigcap \bigcup_{j\in \bigcup (J'_i\bigcup J''_i)}f^{j}\in S_r\Big\} \subseteq \Big\{\bigcup_{j\in \bigcup (J'_i\setminus J''_i)}f^{j}\in S_r\Big\},\\
     &\Big\{\bigcap_{j\in \bigcup J'_i}f^{j}\notin S_r\bigcap \bigcup_{j\in \bigcup (J''_i\bigcup J'_i)}f^{j}\in S_r\Big\} \subseteq \Big\{\bigcup_{j\in \bigcup (J''_i\setminus J'_i)}f^{j}\in S_r\Big\},
 \end{align*}and $J_i'\setminus J_i'', J_i''\setminus J_i'$ contain at most $\Leb (J_i'\setminus J_i'')+1, \Leb(J_i''\setminus J_i')+1$ positive integers, respectively. 
 
 For any $r\in(0,r_J)$ by making use of $f_{*}\mu_{\mathcal{M}}=\mu_{\mathcal{M}}$ and the fact that $S_r$ is a section in $B_r(q)\times [-\pi/2, \pi/2]$, we can continue the estimate (\ref{1}) as 
\begin{align*}
    &\le \left[\Leb(\bigcup J_i'\setminus J_i'')+\Leb(\bigcup J_i''\setminus J_i')+2n \right]\mu_{\mathcal{M}}(S_r)+\Leb(\bigcup J'_i)\mu_{\mathcal{M}}\Big( B_r(q)\times [-\pi/2, \pi/2]\setminus S_r\Big)\\
    &\le \left[2n\max_i J_i \left|\frac{1}{\mu_{\mathcal{M}}(S_r)}-\frac{1}{\mu_{\mathcal{M}}(B_r(q)\times [-\pi/2, \pi/2])}\right|+2n\right]\mu_{\mathcal{M}}(S_r)\\
    &\quad +n \max_i J_i \frac{\mu_{\mathcal{M}}\Big( B_r(q)\times [-\pi/2, \pi/2]\setminus S_r\Big)}{\mu_{\mathcal{M}}\Big( B_r(q)\times [-\pi/2, \pi/2]\Big)},
\end{align*}where $\max_i J_i$ is the maximal positive number in $\bigcup_iJ_i$. By letting $r\to 0$ we conclude a proof of this lemma.
\end{proof}

It follows from Lemmas \ref{kallen}, \ref{poissoncondition1}, and \ref{nonsectiontosection} that, in order to prove Poisson approximations for quasi-sections $B_r(q)\times [-\pi/2, \pi/2]$ (see Theorem \ref{thm}), we just need to prove the following for sections $S_r \subseteq B_r(q)\times [-\pi/2, \pi/2]$, i.e., \begin{align}\label{poissonsection}
    \lim_{r \to 0} \mu_{\mathcal{M}}\Big(\mathcal{N}_s^{r,q}(\bigcup_{i \le n}J_i)=0\Big)=\mathbb{P}\Big(\mathcal{P}(\bigcup_{i \le n}J_i)=0\Big),
\end{align}where $J_1, \cdots, J_n\subseteq [0,\infty)$ are disjoint bounded intervals. The rest of the paper deals with a proof of the relation (\ref{poissonsection}).
\section{Inducing and approximations}\label{section4}

To prove (\ref{poissonsection}), we will give a sufficient condition in Lemma \ref{approximation} for (\ref{poissonsection}), which is deduced from the technical Lemmas \ref{lifting} and \ref{inducing} in this section.

Before proofs, we consider some notions and definitions used in this section for a CMZ structure (see Definition \ref{cmz}).
\begin{definition}[Truncated towers]\ \par
For each $m\ge 0$ let a truncated sub-tower of $\Delta$ be \begin{gather*}
    \Delta_m:=\Delta \bigcap (X \times \{0,1,2,\cdots, m\}).
\end{gather*}

Define now projections $\pi_{X}: \Delta_m \to X$ and $\pi_{\mathbb{N}}: \Delta_m \to \mathbb{N}_0$ (we use here the same notations as in (\ref{projections})) so that for any $(x,n) \in \Delta_m$ \begin{gather*}
    \pi_{X}(x,n)=x, \quad \pi_{\mathbb{N}}(x,n)=n.
\end{gather*}

The first return time $R_m: \Delta_m \to \mathbb{N}$ is  \begin{align*}
    R_m(x)=\inf\{n \ge 1: F^n(x)\in \Delta_m\} \text{ for any }x \in \Delta_m.
\end{align*} 

Explicitly, for any $(x,l) \in \Delta_m$,
\begin{eqnarray*}
 R_m(x,l)=
\begin{cases}
1,    & l< \min\{R(x)-1,m\} \\
\max\{R(x)-m,1\},  & l=\min\{R(x)-1,m\} \\
\end{cases}.
\end{eqnarray*} 

Thus we also have the first return map $F^{R_m}: \Delta_m \to \Delta_m$. Define the $i$-th return times $R_m^i$ recursively as \begin{align*}
   R_m^1:=R_m, \quad R_m^i:=R_m^{i-1}+R_m \circ F^{R_m} \text{ for any }i \ge 2.
\end{align*}

A probability distribution on $\Delta_m$ is defined by
\begin{align*}
    \mu_{\Delta_m}:=(\int \min\{R, m+1\} d\mu_X)^{-1}\sum_{j\le m}(F^j)_{*}\big(\mu_X|_{\{R>j\}}\big).
\end{align*}

Note that the relation $\int R_m d\mu_{\Delta_m}=\mu_{\Delta}(\Delta_m)^{-1}$ follows from the Kac's lemma (see \cite{kac}).

Now, a map $\pi_{\Delta_m}: \Delta \to \Delta_m$ is defined so, that for any $(x,n) \in \Delta$,\begin{eqnarray*}\label{pmmap}
\pi_{\Delta_m} (x,n) =
\begin{cases}
(x,n),      & n\le m\\
(x,m),  & n>m \\
\end{cases},
\end{eqnarray*} i.e., $\pi_{\Delta_m}$ pulls every element of $\Delta \setminus \Delta_m$ back to the roof of $\Delta_m$ and keeps the elements in $\Delta_m$ unchanged. 

Observe also that, if $R(x)>m$ for $x \in X$, then $F^{R^{m+1}_m}(x)=F^R(x)=f^R(x)$.
\end{definition}

\begin{remark}\label{sectionisstillsection}
Since $S_r$ is a section, then $\pi_X: \pi^{-1}S_r \to X$ is injective, and thus $\pi_{X}: \pi_{\Delta_m}\pi^{-1}S_r \to X$ is also injective.
\end{remark}

Now we introduce some point processes for the section $S_r$, which is contained in the quasi-section $B_r(q)\times [-\pi/2, \pi/2] \subseteq \mathcal{M}$.
\begin{definition}
For each $m \ge 1$, define point processes on $(\Delta_m, \mu_{\Delta_m})$ so that for any measurable set $A\subseteq [0,\infty)$,  \begin{gather*}
    \mathcal{N}_s^{r}(A):=\sum_{i\cdot \mu_{\Delta}(\pi^{-1}S_r)\in A}\mathbbm{1}_{\pi^{-1}S_r}\circ F^i,\\
      \mathcal{N}_{i}^{r}(A):=\sum_{k\cdot \mu_{\Delta_m}(\pi_{\Delta_m}\pi^{-1}S_r)\in A}\mathbbm{1}_{\pi_{\Delta_m}\pi^{-1}S_r}\circ (F^{R_m})^k.
\end{gather*}
\end{definition}

Note that $\mathcal{N}_s^r$ can be viewed as a point process on $(\Delta, \mu_{\Delta})$. Using $\mathcal{N}_s^r=_d\mathcal{N}_s^{r,q}$, where the last one was defined in (\ref{sectionpointprocess}), we get the following lemma.

\begin{lemma}[Lifting]\label{lifting}\ \par
Suppose that for any disjoint bounded intervals $J_1,\cdots, J_n\subseteq [0,\infty)$,\begin{align*}
    \lim_{r \to 0} \mu_{\Delta}\Big(\mathcal{N}_s^{r}(\bigcup_{k \le n}J_k)=0\Big)=\mathbb{P}\Big(\mathcal{P}(\bigcup_{k \le n}J_k)=0\Big). 
\end{align*}

Then for any disjoint bounded intervals $J_1,\cdots, J_n\subseteq [0,\infty)$,\begin{gather*}
    \lim_{r \to 0} \mu_{\mathcal{M}}\Big(\mathcal{N}_s^{r,q}(\bigcup_{i \le n}J_i)=0\Big)=\mathbb{P}\Big(\mathcal{P}(\bigcup_{i \le n}J_i)=0\Big).
\end{gather*}
\end{lemma}

For each $m\ge 1$, we can now study a limit law for $\mathcal{N}_i^r$, which is defined on $\Delta_m$.
\begin{lemma}[Inducing]\label{inducing}\ \par
For $m\ge 1$, suppose that for any disjoint bounded intervals $J_1,\cdots, J_n\subseteq [0,\infty)$,\begin{align*}
    \lim_{r \to 0} \mu_{\Delta_m}\Big(\mathcal{N}_i^{r}(\bigcup_{k \le n}J_k)=0\Big)=\mathbb{P}\Big(\mathcal{P}(\bigcup_{k \le n}J_k)=0\Big).
\end{align*}

Then for any disjoint bounded intervals $J_1,\cdots, J_n\subseteq [0,\infty)$,\begin{gather*}
    \lim_{r \to 0} \mu_{\Delta_m}\Big(\mathcal{N}_s^{r}(\bigcup_{k \le n}J_k)=0\Big)=\mathbb{P}\Big(\mathcal{P}(\bigcup_{k \le n}J_k)=0\Big).
\end{gather*}
\end{lemma}
\begin{proof} We divide the proof into several steps.

As the first step, we introduce hitting times and their properties.  For any $x\in \Delta_m$ define \begin{gather*}
    \tau^1_s(x):=\inf\{n\ge 1: F^n(x)\in \pi^{-1}S_r\},\\ \tau^1_p(x):=\inf\{n\ge 1: F^n(x)\in \pi_{\Delta_m}\pi^{-1}S_r\},\\
   \tau^1_i(x):=\inf\{n\ge 1: (F^{R_m})^n(x)\in \pi_{\Delta_m}\pi^{-1}S_r\}.
\end{gather*} 

The corresponding $j$-th ($j\ge 2$) return times $\tau_s^j, \tau_p^j, \tau_i^j$ are defined inductively as follows\begin{gather*}
    \tau^j_s(x):=\inf\{n> \tau^{j-1}_s(x): F^n(x)\in \pi^{-1}S_r\},\\ \tau^j_p(x):=\inf\{n> \tau^{j-1}_p(x): F^n(x)\in \pi_{\Delta_m}\pi^{-1}S_r\},\\
    \tau^j_i(x):=\inf\{n> \tau^{j-1}_i(x): (F^{R_m})^n(x)\in \pi_{\Delta_m}\pi^{-1}S_r\}.
\end{gather*}

So for any $j\ge 1$ we have \begin{gather*}
    \tau_s:=\tau_s^1<\tau_s^2<\cdots<\tau_s^j<\tau_s^{j+1}<\cdots \to \infty,\\
    \tau_p:=\tau_p^1<\tau_p^2<\cdots<\tau_p^j<\tau_p^{j+1}<\cdots \to \infty,\\
    \tau_i:=\tau_i^1<\tau_i^2<\cdots<\tau_i^j<\tau_i^{j+1}<\cdots \to \infty.
\end{gather*}
 
By making use of the fact that $S_r$ is a section and Remark \ref{sectionisstillsection},  we get that for any $j\ge 1$,  
 \begin{gather*}
     \tau^j_p=R_m+R_m\circ F^{R_m}+ \cdots + R_m \circ (F^{R_m})^{\tau^j_i-1},\\
     |\tau^j_p-\tau^j_s|\le R_m \circ (F^{R_m})^{\tau^j_i}.\\
     \end{gather*}
 
 From the Birkhoff's ergodic theorem we obtain \begin{gather*}
     \lim_{n \to \infty}\frac{\sum_{i=0}^{n-1}R_m\circ (F^{R_m})^i}{n}=\int R_m d\mu_{\Delta_m}=\mu_{\Delta}(\Delta_m)^{-1} \quad  \mu_{\Delta_m}\text{-a.s.}\\
     \lim_{n \to \infty}\frac{R_m \circ (F^{R_m})^n}{n}=0 \quad  \mu_{\Delta_m}\text{-a.s.}
 \end{gather*}
 
 Therefore, for any sufficiently small $\epsilon' \in (0, \min\{\frac{1}{2\mu_{\Delta}(X)}, \min_{j}\diam J_j \})$ and a.e. $x\in \Delta_m$, there is $N_{\epsilon', x}>0$, such that for any $n >N_{\epsilon',x}$,\begin{gather*}
     \Big|\frac{\sum_{i=0}^{ n-1}R_m\circ (F^{R_m})^i}{n}-\mu_{\Delta}(\Delta_m)^{-1}\Big|\le \epsilon', \quad \Big|\frac{R_m \circ (F^{R_m})^n}{n}\Big|\le \epsilon'.
 \end{gather*}
 
 Let $G_{n, \epsilon'}:=\{x\in \Delta_m: n \ge N_{\epsilon',x}\}$. It is obvious that $G_{n,\epsilon'}\subseteq G_{n+1, \epsilon'}$ for any $n \ge 1$. Hence $\lim_{n\to \infty}\mu_{\Delta_m}(G_{n,\epsilon'})=1$. Therefore, there is $N_{\epsilon'}>0$, such that for any $n\ge N_{\epsilon'}$, $\mu_{\Delta_m}(G_{n,\epsilon'}^c)\le \epsilon'$.
 
 Let $x\in G_{N_{\epsilon'}, \epsilon'}$.  We have that $x \in G_{n, \epsilon'}$ for any $n \ge N_{\epsilon'}$ and \begin{gather*}
     \Big|\frac{\sum_{i=0}^{ n-1}R_m\circ (F^{R_m})^i}{n}-\mu_{\Delta}(\Delta_m)^{-1}\Big|\le \epsilon', \quad \Big|\frac{R_m \circ (F^{R_m})^n}{n}\Big|\le \epsilon'.
 \end{gather*}
 
 In particular, if $\tau_i \ge N_{\epsilon'}$, (which implies $\tau_i^j\ge N_{\epsilon'}$ for all $j\ge 1$), then for all $j \ge 1$, \begin{gather*}
     \frac{\tau_p^j}{\tau_i^j}\in [\mu_{\Delta}(\Delta_m)^{-1}-\epsilon', \mu_{\Delta}(\Delta_m)^{-1}+\epsilon'],\\
     \Big|\frac{\tau_s^j}{\tau_i^j}-     \frac{\tau_p^j}{\tau_i^j}\Big| \le \frac{R_m \circ (F^{R_m})^{\tau_i^j}}{\tau_i^j} \le \epsilon'.
 \end{gather*} 
 
 Therefore, for any $j\ge 1$ \begin{gather*}
     \frac{\tau_s^j}{\tau_i^j}\in [\mu_{\Delta}(\Delta_m)^{-1}-2\epsilon', \mu_{\Delta}(\Delta_m)^{-1}+2\epsilon'].\end{gather*}

Now let $H_{\epsilon'}:=\{x\in \Delta_m: \tau_i(x)> N_{\epsilon'}\}$. Then\begin{align*}
    \mu_{\Delta_m}(H_{\epsilon'}^c)=\mu_{\Delta_m}\{x\in \Delta_m: \tau_i(x)\le  N_{\epsilon'}\}&\le N_{\epsilon'}\mu_{\Delta_m}(\pi_{\Delta_m}\pi^{-1}S_r)\\
    &=N_{\epsilon'} \frac{\int R d\mu_X}{\int \min\{R, m+1\} d\mu_X}\mu_{\Delta}(\pi^{-1}S_r)\\
    &=N_{\epsilon'} \frac{\int R d\mu_X}{\int \min\{R, m+1\} d\mu_X}\mu_{\mathcal{M}}(S_r),
\end{align*} where the first equality holds because $S_r$ is a section.

Therefore, there is $r_{\epsilon'}>0$, such that for any $r\in (0,r_{\epsilon'})$,  \begin{gather}\label{step11}
    \mu_{\Delta_m}(H_{\epsilon'}^c)\le \epsilon',\quad  \mu_{\Delta_m}(H_{\epsilon'}^c\bigcup G_{N_{\epsilon'}, \epsilon'}^c)\le 2\epsilon'.
\end{gather}

Besides, for any $x\in H_{\epsilon'}\bigcap G_{N_{\epsilon'}, \epsilon'}$, and any $j\ge 1$, \begin{gather}\label{step12}
     \frac{\tau_s^j(x)}{\tau_i^j(x)}\in [\mu_{\Delta}(\Delta_m)^{-1}-2\epsilon', \mu_{\Delta}(\Delta_m)^{-1}+2\epsilon'].\end{gather}

As the second step, we connect $\mathcal{N}_s^r, \mathcal{N}_i^r$ with the hitting times $\tau_s^j, \tau_i^j$.

Let $aJ_k:=\{aj:j \in J_k\}$ for any $a>0$, and $J_k':= \mu_{\Delta}(\pi^{-1}S_r)^{-1} J_k$, \begin{align}
    J_k'':&=(\mu_{\Delta}(\Delta_m)^{-1}-2\epsilon')^{-1}J_k'\bigcap (\mu_{\Delta}(\Delta_m)^{-1}+2\epsilon')^{-1}J_k' \label{step21}\\
    &=(\mu_{\Delta}(\Delta_m)^{-1}-2\epsilon')^{-1}\mu_{\Delta}(\pi^{-1}S_r)^{-1} J_k \bigcap (\mu_{\Delta}(\Delta_m)^{-1}+2\epsilon')^{-1}\mu_{\Delta}(\pi^{-1}S_r)^{-1} J_k \nonumber\\
    &=(\mu_{\Delta}(\Delta_m)^{-1}-2\epsilon')^{-1}\mu_{\Delta}(\pi_{\Delta_m}\pi^{-1}S_r)^{-1} J_k\bigcap (\mu_{\Delta}(\Delta_m)^{-1}+2\epsilon')^{-1}\mu_{\Delta}(\pi_{\Delta_m}\pi^{-1}S_r)^{-1} J_k\nonumber\\
     &=(1-2\mu_{\Delta}(\Delta_m)\epsilon')^{-1}\mu_{\Delta_m}(\pi_{\Delta_m}\pi^{-1}S_r)^{-1} J_k\bigcap (1+2\mu_{\Delta}(\Delta_m)\epsilon')^{-1}\mu_{\Delta_m}(\pi_{\Delta_m}\pi^{-1}S_r)^{-1} J_k, \nonumber
\end{align}where the third equality holds because $S_r$ is a section. Next, let
\begin{align*}
    J_k''':=(1-2\mu_{\Delta}(\Delta_m)\epsilon')^{-1}J_k\bigcap (1+2\mu_{\Delta}(\Delta_m)\epsilon')^{-1} J_k,
\end{align*} which implies that $\mu_{\Delta_m}(\pi_{\Delta_m}\pi^{-1}S_r)^{-1}J_k'''=J_k''.$
Then we have \begin{gather}
     \{x \in \Delta_m  \setminus \pi^{-1}S_r: \mathcal{N}_s^{r}(\bigcup_{k \le n}J_k)(x)=0 \}=\{x\in \Delta_m \setminus \pi^{-1}S_r: \tau^j_s(x) \notin \bigcup_{k \le n}J'_k \text{ for all }j \ge 1\},\label{step22}\\
     \{x \in \Delta_m  \setminus \pi_{\Delta_m}\pi^{-1}S_r: \mathcal{N}_i^{r}(\bigcup_{k \le n}J_k''')(x)=0 \}=\{x\in \Delta_m \setminus \pi_{\Delta_m}\pi^{-1}S_r: \tau^j_i(x) \notin \bigcup_{k \le n}J''_k \text{ for all }j \ge 1\}. \label{step23}
 \end{gather}

As the third step, we estimate $\mathcal{N}_{s}^r$ via $\mathcal{N}_{i}^r$ by applying (\ref{step11}) and (\ref{step22}),
\begin{align*}
    \mu_{\Delta_m}&\{x \in \Delta_m: \mathcal{N}_s^{r}(\bigcup_{k \le n}J_k)(x)=0 \}\\
    &\le \mu_{\Delta_m}\{x\in \Delta_m \setminus \pi^{-1}S_r: \tau^j_s(x) \notin \bigcup_{k \le n}J'_k \text{ for all }j \ge 1\}+\mu_{\Delta_m}(\pi^{-1}S_r)\\
    &\le \mu_{\Delta_m}\{x\in  H_{\epsilon'}\bigcap G_{N_{\epsilon'}, \epsilon'}: \tau^j_s(x) \notin \bigcup_{k \le n}J'_k \text{ for all }j \ge 1\}+2 \epsilon' +\mu_{\Delta_m}( \pi_{\Delta_m}\pi^{-1}S_r).
    \end{align*}
    
    By using (\ref{step12}) and (\ref{step21}), the expression above can be estimated as
    \begin{align*}
    &\le \mu_{\Delta_m}\{x\in  H_{\epsilon'}\bigcap G_{N_{\epsilon'}, \epsilon'}: \tau^j_i(x) \notin \bigcup_{k \le n}J''_k \text{ for all }j \ge 1\}+2 \epsilon'+\mu_{\Delta_m}( \pi_{\Delta_m}\pi^{-1}S_r)\\
    &\le \mu_{\Delta_m}\{x\in \Delta_m: \tau^j_i(x) \notin \bigcup_{k \le n}J_k'' \text{ for all }j \ge 1\} +2 \epsilon'+ \mu_{\Delta_m}( \pi_{\Delta_m}\pi^{-1}S_r)\\
    &\le \mu_{\Delta_m}\{x\in \Delta_m\setminus \pi_{\Delta_m}\pi^{-1}S_r: \tau^j_i(x) \notin \bigcup_{k \le n}J_k''\text{ for all }j \ge 1\} +2 \epsilon'+2\mu_{\Delta_m}( \pi_{\Delta_m}\pi^{-1}S_r).
    \end{align*}
    
    Now, according to (\ref{step23}), the inequality above can be written as
    \begin{align*}
    &= \mu_{\Delta_m}\{x\in \Delta_m\setminus \pi_{\Delta_m}\pi^{-1}S_r: \mathcal{N}_i^{r}(\bigcup_{k \le n}J_k''')(x)=0 \}+2 \epsilon'+2\mu_{\Delta_m}( \pi_{\Delta_m}\pi^{-1}S_r)\\
    &\le \mu_{\Delta_m}\{x \in \Delta_m: \mathcal{N}_i^{r}(\bigcup_{k \le n}J_k''')(x)=0 \}+2 \epsilon'+\frac{2\int R d\mu_X}{\int \min\{R, m+1\} d\mu_X}\mu_{\mathcal{M}}(S_r).
\end{align*}

From the conditions of this lemma we have\begin{gather*}
    \limsup_{r\to 0}\mu_{\Delta_m}\{x \in \Delta_m: \mathcal{N}_s^{r}(\bigcup_{k \le n}J_k)(x)=0 \}\le e^{-\Leb_{\mathbb{R}^+\bigcup \{0\}} (\bigcup_{k \le n}J_k''')}+2 \epsilon'.
\end{gather*}

Let $\epsilon'\to 0$. Then $\Leb_{\mathbb{R}^+\bigcup \{0\}}(\bigcup_{k \le n} J_k''') \to \Leb_{\mathbb{R}^+\bigcup \{0\}}(\bigcup_{k \le n} J_k)$ and \begin{gather*}
    \limsup_{r\to 0}\mu_{\Delta_m}\{x \in \Delta_m: \mathcal{N}_s^{r}(\bigcup_{k \le n}J_k)(x)=0 \}\le e^{-\Leb_{\mathbb{R}^+\bigcup \{0\}} (\bigcup_{k \le n}J_k)}.
\end{gather*}

As the fourth step, we connect $\mathcal{N}_s^r, \mathcal{N}_i^r$ with the hitting times $\tau_s^j, \tau_i^j$ in a different way. 

Let $\hat{J_k}:=(1-2\mu_{\Delta}(\Delta_m)\epsilon')^{-1}J_k\bigcup (1+2\mu_{\Delta}(\Delta_m)\epsilon')^{-1} J_k$, $\hat{J_k'}:=\mu_{\Delta_m}(\pi_{\Delta_m}\pi^{-1}S_r)^{-1} \hat{J_k}$, \begin{align}
    \hat{J_k''}:&=(\mu_{\Delta}(\Delta_m)^{-1}-2\epsilon')\hat{J_k'} \bigcap (\mu_{\Delta}(\Delta_m)^{-1}+2\epsilon')\hat{J_k'} \label{step41}\\
    &=(\mu_{\Delta}(\Delta_m)^{-1}-2\epsilon')\mu_{\Delta_m}(\pi_{\Delta_m}\pi^{-1}S_r)^{-1} \hat{J_k} \bigcap (\mu_{\Delta}(\Delta_m)^{-1}+2\epsilon')\mu_{\Delta_m}(\pi_{\Delta_m}\pi^{-1}S_r)^{-1} \hat{J_k} \nonumber \\
     &=(1-2\mu_{\Delta}(\Delta_m)\epsilon')\mu_{\Delta}(\pi_{\Delta_m}\pi^{-1}S_r)^{-1} \hat{J_k}\bigcap (1+2\mu_{\Delta}(\Delta_m)\epsilon')\mu_{\Delta}(\pi_{\Delta_m}\pi^{-1}S_r)^{-1} \hat{J_k}\nonumber\\
     &=(1-2\mu_{\Delta}(\Delta_m)\epsilon')\mu_{\Delta}(\pi^{-1}S_r)^{-1} \hat{J_k}\bigcap (1+2\mu_{\Delta}(\Delta_m)\epsilon')\mu_{\Delta}(\pi^{-1}S_r)^{-1} \hat{J_k},\nonumber,
\end{align}where the fourth equality holds because $S_r$ is a section. Clearly, \begin{align*}
    J_k=(1-2\mu_{\Delta}(\Delta_m)\epsilon')\hat{J_k}\bigcap (1+2\mu_{\Delta}(\Delta_m)\epsilon') \hat{J_k}, \quad \mu_{\Delta}(\pi^{-1}S_r)^{-1}J_k=\hat{J_k''}.
\end{align*} Then we have \begin{gather}
     \{x \in \Delta_m \setminus \pi_{\Delta_m}\pi^{-1}S_r: \mathcal{N}_i^{r}(\bigcup_{k \le n}\hat{J_k})(x)=0 \}=\{x\in \Delta_m \setminus \pi_{\Delta_m}\pi^{-1}S_r: \tau^j_i(x) \notin \bigcup_{k \le n}\hat{J_k'}\text{ for all }j \ge 1\},\label{step42}\\
     \{x \in \Delta_m \setminus \pi^{-1}S_r: \mathcal{N}_s^{r}(\bigcup_{k \le n}J_k)(x)=0 \}=\{x\in \Delta_m \setminus \pi^{-1}S_r: \tau^j_s(x) \notin \bigcup_{k \le n}\hat{J_k''} \text{ for all }j \ge 1\}. \label{step43}
 \end{gather}

As the fifth step, we estimate $\mathcal{N}_{i}^r$ via $\mathcal{N}_{s}^r$ from (\ref{step11}) and (\ref{step42})\begin{align*}
    \mu_{\Delta_m}&\{x \in \Delta_m: \mathcal{N}_i^{r}(\bigcup_{k \le n}\hat{J_k})(x)=0 \}\\
    &\le \mu_{\Delta_m}\{x\in \Delta_m \setminus \pi_{\Delta_m}\pi^{-1}S_r: \tau^j_i(x) \notin \bigcup_{k \le n}\hat{J_k'} \text{ for all }j \ge 1\}+\mu_{\Delta_m}(\pi_{\Delta_m}\pi^{-1}S_r)\\
    &\le \mu_{\Delta_m}\{x\in  H_{\epsilon'}\bigcap G_{N_{\epsilon'}, \epsilon'}: \tau^j_i(x) \notin \bigcup_{k \le n}\hat{J_k'}\text{ for all }j \ge 1\}+2 \epsilon'+\mu_{\Delta_m}(\pi_{\Delta_m}\pi^{-1}S_r).
    \end{align*}
    
    By using (\ref{step12}) and (\ref{step41}) the inequality above can be written as
    \begin{align*}
    &\le \mu_{\Delta_m}\{x\in  H_{\epsilon'}\bigcap G_{N_{\epsilon'}, \epsilon'}: \tau^j_s(x) \notin \bigcup_{k \le n}\hat{J_k''} \text{ for all }j \ge 1\}+2 \epsilon'+\mu_{\Delta_m}(\pi_{\Delta_m}\pi^{-1}S_r)\\
    &\le \mu_{\Delta_m}\{x\in \Delta_m: \tau^j_s(x) \notin \bigcup_{k \le n}\hat{J_k''}  \text{ for all }j \ge 1\}+2 \epsilon'+\mu_{\Delta_m}(\pi_{\Delta_m}\pi^{-1}S_r)\\
    &\le \mu_{\Delta_m}\{x\in \Delta_m\setminus \pi^{-1}S_r: \tau^j_s(x) \notin \bigcup_{k \le n}\hat{J_k''} \text{ for all }j \ge 1\}+2 \epsilon'+2\mu_{\Delta_m}(\pi_{\Delta_m}\pi^{-1}S_r).
    \end{align*}
    
    Now, from  (\ref{step43}) the inequality above can be estimated as
    \begin{align*}
    &= \mu_{\Delta_m}\{x\in \Delta_m\setminus \pi^{-1}S_r:\mathcal{N}_s^{r}(\bigcup_{k \le n}J_k)(x)=0 \}+2 \epsilon'+2\mu_{\Delta_m}(\pi_{\Delta_m}\pi^{-1}S_r)\\
    &\le\mu_{\Delta_m}\{x \in \Delta_m: \mathcal{N}_s^{r}(\bigcup_{k \le n}J_k)(x)=0 \}+2 \epsilon'+\frac{2\int R d\mu_X}{\int \min\{R, m+1\} d\mu_X}\mu_{\mathcal{M}}(S_r).
\end{align*}

From the conditions of this lemma we have\begin{gather*}
    \liminf_{r\to 0}\mu_{\Delta_m}\{x \in \Delta_m: \mathcal{N}_s^{r}(\bigcup_{k \le n}J_k)(x)=0 \}\ge e^{-\Leb_{\mathbb{R}^+\bigcup \{0\}} (\bigcup_{k \le n}\hat{J_k})}-2 \epsilon'.
\end{gather*}

Let $\epsilon'\to 0$. Then $\Leb_{\mathbb{R}^+\bigcup \{0\}}(\bigcup_{k \le n} \hat{J}_k) \to \Leb_{\mathbb{R}^+\bigcup \{0\}}(\bigcup_{k \le n} J_k)$ and \begin{gather*}
    \liminf_{r\to 0}\mu_{\Delta_m}\{x \in \Delta_m: \mathcal{N}_s^{r}(\bigcup_{k \le n}J_k)(x)=0 \}\ge e^{-\Leb_{\mathbb{R}^+\bigcup \{0\}} (\bigcup_{k \le n}J_k)}.
\end{gather*}

The previous steps lead to the following conclusion
\begin{gather*}
    \lim_{r\to 0}\mu_{\Delta_m}\{x \in \Delta_m: \mathcal{N}_s^{r}(\bigcup_{k \le n}J_k)(x)=0 \}= e^{-\Leb_{\mathbb{R}^+\bigcup \{0\}} (\bigcup_{k \le n}J_k)}=\mathbb{P}\Big(\mathcal{P}(\bigcup_{k \le n}J_k)=0\Big)
\end{gather*} for any disjoint bounded intervals $J_1, \cdots, J_n\subseteq [0,\infty)$. Hence,  this lemma is proved.
\end{proof}

\begin{remark} In \cite{vaientietdsinducing} a similar result was obtained for return statistics to balls. We are dealing here with a more general hitting statistics for sections.
\end{remark} 

Together with Lemmas \ref{lifting} and \ref{inducing}, the last lemma (an approximation technique) in this section provides a sufficient condition for (\ref{poissonsection}).

\begin{lemma}[Approximations]\label{approximation} \par
Suppose that for any sufficiently large $m \ge 1$, and any disjoint bounded intervals $J_1,\cdots, J_n\subseteq [0,\infty)$,\begin{align}\label{inducingforanym}
    \lim_{r \to 0} \mu_{\Delta_m}\Big(\mathcal{N}_i^{r}(\bigcup_{k \le n}J_k)=0\Big)=\mathbb{P}\Big(\mathcal{P}(\bigcup_{k \le n}J_k)=0\Big).
\end{align} Then for any disjoint bounded intervals $J_1,\cdots, J_n\subseteq [0,\infty)$ we have\begin{align*}
    \lim_{r \to 0} \mu_{\Delta}\Big(\mathcal{N}_s^{r}(\bigcup_{k \le n}J_k)=0\Big)=\mathbb{P}\Big(\mathcal{P}(\bigcup_{k \le n}J_k)=0\Big).
\end{align*} 

By Lemma \ref{lifting} we also get the required relation (\ref{poissonsection}), \begin{align*}
    \lim_{r \to 0} \mu_{\mathcal{M}}\Big(\mathcal{N}_s^{r,q}(\bigcup_{k \le n}J_k)=0\Big)=\mathbb{P}\Big(\mathcal{P}(\bigcup_{k \le n}J_k)=0\Big),
\end{align*} provided that (\ref{inducingforanym}) holds for each large $m \ge 1$. Then this concludes a proof of Poisson limit laws in Theorem \ref{thm}.
\end{lemma}
\begin{proof} It follows from Lemma \ref{inducing} that for any $m\gg 1$ \begin{align*}
    \lim_{r \to 0} \mu_{\Delta_m}\Big(\mathcal{N}_s^{r}(\bigcup_{k \le n}J_k)=0\Big)=\mathbb{P}\Big(\mathcal{P}(\bigcup_{k \le n}J_k)=0\Big).
\end{align*} Therefore,
\begin{align*}
    \lim_{r\to 0}\mu_{\Delta}\Big(\mathcal{N}_s^{r}(\bigcup_{k \le n}J_k)=0\Big)&= \mu_{\Delta}(\Delta_m)\lim_{r\to 0}\mu_{\Delta_m}\Big(\mathcal{N}_s^{r}(\bigcup_{k \le n}J_k)=0\Big)+O\Big(\mu_{\Delta}(\Delta_m^c)\Big)\\
    &= \mu_{\Delta}(\Delta_m)\mathbb{P}\Big(\mathcal{P}(\bigcup_{k \le n}J_k)=0\Big)+O\Big(\mu_{\Delta}(\Delta_m^c)\Big).
\end{align*} Let $m \to \infty$, then $\lim_{r\to 0}\mu_{\Delta}\Big(\mathcal{N}_s^{r}(\bigcup_{k \le n}J_k)=0\Big)=\mathbb{P}\Big(\mathcal{P}(\bigcup_{k \le n}J_k)=0\Big)$.
\end{proof}
 
\begin{remark}\label{remarkbefore5}
Thanks to Lemmas \ref{approximation}, \ref{nonsectiontosection}, \ref{poissoncondition1} and \ref{kallen}, in order to prove that $\mathcal{N}^{r,q} \to_d \mathcal{P}$ in Theorem \ref{thm}, we just need to verify  (\ref{inducingforanym})  for the dynamical system $(\Delta_m, F^{R_m}, \mu_{\Delta_m})$ for each large $m\ge 1$. Indeed, $(\Delta_m, F^{R_m}, \mu_{\Delta_m})$ is a  hyperbolic dynamical system with exponential decay of correlations and  arbitrarily large contracting (resp. expanding) rates $\alpha$ along stable (resp. unstable) manifolds. It allows us to skip verification of the condition $\alpha \cdot \lim_{r \to 0}\frac{\log \mu_{\mathcal{M}}\{B_r(q)\times [-\pi/2, \pi/2]\}}{\log r}>1$ (see \cite{peneetds, peneijm}), which fails for many slowly mixing billiard systems.
\end{remark}

\section{Thicker hyperbolic and expanding Young towers }\label{section5}
In order to prove (\ref{inducingforanym}), we will model dynamical systems $(\Delta_m, F^{R_m}, \mu_{\Delta_m})$ by hyperbolic (although non-mixing) Young towers (see \cite{Young,Young2}), but with exponential contracting (resp. expanding) rates along stable (resp. unstable) manifolds. Some properties of non-mixing hyperbolic Young towers are introduced in this section.

\subsection{Hyperbolic Young towers for $(X, f^R, \mu_X)$} In this subsection we will consider a dynamical system $(X, f^R, \mu_X)$. To simplify notations denote $f^R$ or $F^R$ by $g$. According to \cite{Chernovjsp,CZnon,Markarianetds} and Definition \ref{cmz} (CMZ structures), there exists a compact set $\Lambda \subseteq  \bigcap_{0\le i \le N}g^{i}\mathbb{S}^c \bigcap \bigcap_{0\le i \le N}g^{-i}\mathbb{S}^c \subseteq X$ with a hyperbolic product structure (here $N$ is defined in the Definition \ref{cmz}). Besides, there are families of $C^1$-stable disks (i.e., closed connected pieces of stable manifolds) $\Gamma^s:=\{\gamma^s\}$, and families of $C^1$-unstable disks (i.e., closed connected pieces of unstable manifolds) $\Gamma^u:=\{\gamma^u\}$, such that the following conditions hold.

\begin{enumerate}

    \item $\Lambda=\left(\bigcup \gamma^s\right) \bigcap \left(\bigcup \gamma^u\right)$,
    
    \item $\dim \gamma^s+\dim \gamma^u=\dim \mathcal{M}$,
    
    \item each $\gamma^s$ intersects every $\gamma^u$ 
    at exactly one point,
    
    \item stable and unstable manifolds are transversal, and the angles between them are uniformly bounded away from 0,
   \item $\Gamma^u:=\{\gamma^u\}$ is a continuous family, i.e., there is a  compact set $K^s$, a unit disk $O^u$ in some Euclidean space and a map $\phi^u: K^s \times O^u \to \mathcal{M}$, such that 
\begin{enumerate}
    \item $\gamma^u=\phi^u\left(\{x\}\times O^u\right)$ is an unstable manifold,
    \item $\phi^u$ maps $K^s \times O^u$ homeomorphically onto its image,
    \item $x \to \phi^u|_{\{x\} \times O^u}$ defines a continuous map from $K^s$ to $\Emb^1\left(O^u, \mathcal{M}\right)$, where $\Emb^1\left(O^u, \mathcal{M}\right)$ is the space of $C^1$-embeddings of $O^u$ into $\mathcal{M}$.
\end{enumerate}
$\Gamma^s:=\{\gamma^s\}$ is also a continuous family in the same sense.
\item Lebesgue detectability:  there exists $\gamma \in \Gamma^u$ such that $\Leb_{\gamma}\left(\Lambda \bigcap \gamma\right) > 0$, where $\Leb_{\gamma}$ is the Lebesgue measure on $\gamma$.
\item Markov property: there exist pairwise disjoint $s$-subsets $\Lambda_1,\Lambda_2, \cdots \subseteq \Lambda$, such that , each $\Lambda_i=\left(\bigcup_{\gamma^s \in \Gamma^s_i} \gamma^s\right) \bigcap \left(\bigcup_{\gamma^u \in \Gamma^u} \gamma^u\right)$ for some $\Gamma^s_i\subseteq \Gamma^s$, and \begin{enumerate}
        \item $\Leb_{\gamma}\left(\Lambda \setminus(\bigcup_{i \ge 1}\Lambda_i)\right)=0$ on each $\gamma \in \Gamma^u$,
        \item there is a return time function $R^e: \Lambda \to \mathbb{N}$ and a return map $g^{R^e}: \Lambda \to \Lambda$, such that for each $i \ge 1$ \[R^e\big|_{\Lambda_i}:=R^{e}_{i}\equiv 0 \pmod N,\quad 
g^{R^{e}}\big|_{\Lambda_i}:=g^{R^{e}_{i}}\big|_{\Lambda_i},\]$g^{R^{e}_{i}} (\Lambda_i)$ is a $u$-subset (i.e., each $g^{R^{e}_{i}}(\Lambda_i)=\left(\bigcup_{\gamma^s \in \Gamma^s} \gamma^s\right) \bigcap \left(\bigcup_{\gamma^u \in \Gamma^u_i} \gamma^u\right)$ for some $\Gamma^u_i\subseteq \Gamma^u$), and for all $x \in \Lambda_i$, \begin{align*}
g^{R^{e}_{i}}\Big(\gamma^s(x)\Big) \subseteq \gamma^s\Big(g^{R^{e}_{i}}(x)\Big), \quad g^{R^{e}_{i}}\Big(\gamma^u(x)\Big) \supseteq \gamma^u\Big(g^{R^{e}_{i}}(x)\Big),
        \end{align*}where $\gamma^u(y)$ (resp. $\gamma^s(y)$) is an element of $\Gamma^u$ (resp. $\Gamma^s$)
which contains $y\in \Lambda$, and $N$ is defined in  \ref{cmz}.
    \end{enumerate}

Moreover, there exist constants $C \ge 1$ and $0 < \beta < 1$ such that the following conditions hold.

\item Exponential contraction of stable disks: for any $\gamma^s \in \Gamma^s, x,y \in \gamma^s,  n \ge 1$,
\[ d\Big(g^n(x), g^n(y)\Big) \le C\beta^n.\]

\item Backward exponential contraction of unstable disks: for any $ \gamma^u \in \Gamma^u, x,y \in \gamma^u,  n\ge 1$,
\[ d\Big(g^{-n}(x), g^{-n}(y)\Big) \le C\beta^n.\]
\item Bounded distortion: for any $ \gamma \in \Gamma^u \text{ and } x, y \in \gamma \bigcap \Lambda_i$ for some $\Lambda_i$, 
\[\log \frac{\det (D^ug^{R^e})(x)}{\det (D^ug^{R^e})(y)} \le C  \beta^{s\left(g^{R^e}(x), g^{R^e}(y)\right)},\] where $s(x, y)$ is the separation time, i.e., for any $x,y \in \Lambda$,
\[s(x,y):=\min \{n \ge 0: (g^{R^e})^n(x) \text{ and } (g^{R^e})^n(y) \text{ belong to different sets } \Lambda_i\}.\]
\item Regularity of the stable foliations: for each $\gamma, \gamma'\in \Gamma^u$, define $\Theta_{\gamma, \gamma'}: \gamma' \bigcap \Lambda \to \gamma \bigcap \Lambda$ by
\[\Theta_{\gamma, \gamma'}(x)=\gamma^s(x)\bigcap \gamma.\]

Then the following properties hold
\begin{enumerate}
    \item $\Theta_{\gamma, \gamma'}$ is absolutely continuous, and for any $ x \in \gamma \bigcap \Lambda$
    \[\frac{d\left(\Theta_{\gamma, \gamma'}\right)_{*}\Leb_{\gamma'}}{d\Leb_{\gamma}}(x)=\prod_{n \ge 0} \frac{\det (D^ug)\left(g^n\left(x\right)\right)}{\det (D^ug)\big(g^n(\Theta_{\gamma, \gamma'}^{-1}x)\big)}= C^{\pm 1},\]
    \item for any $x,y \in \gamma \bigcap \Lambda$, 
    \[\log \frac{\frac{d\left(\Theta_{\gamma, \gamma'}\right)_{*}\Leb_{\gamma'}}{d\Leb_{\gamma}}(x)}{\frac{d\left(\Theta_{\gamma, \gamma'}\right)_{*}\Leb_{\gamma'}}{d\Leb_{\gamma}}(y)} \le C  \beta^{s(x,y)}.\]
\end{enumerate}
\item Decay rate of the return times $R^e$: for any $\gamma\in \Gamma^u$ \begin{align}\label{taildecay}
       \Leb_{\gamma}\left(R^e>n\right) \le C\beta^n.
   \end{align}
\end{enumerate}
\begin{remark}
Note that $\gcd{\{R^{e}\}}\ge N$. Indeed, it follows from the fact that $(f^R)^N$ satisfies the growth lemmas for a CMZ structure.
\end{remark}

Now we can construct a hyperbolic Young tower $\Delta_e$ with dynamics $F_e : \Delta_e \to \Delta_e$, where
\begin{align*}
    &\Delta_e:=\{(x,l) \in \Lambda \times \mathbb{N}_0 : 0 \le l < R^e(x) \},\\
    &F_e(x,l):=\begin{cases}
 (x,l+1),      &l < R^e(x)-1\\
\left(g^{R^e}(x),0\right),  & l=R^e(x)-1\\
\end{cases},
\end{align*}
and  define a projection $\pi_{e}: \Delta_{e} \to X $ such that 
\[\pi_{e}(x,l):=g^l(x).\]

The equivalence relation $\sim$ on $\Lambda$ is then
\[x \sim y \text{ if and only if } x,y \in \gamma^s \text{ for some } \gamma^s \in \Gamma^s.\]

Another equivalence relation $\sim$ on $\Delta_e$ is
\[(x,n) \sim (y,m) \text{ if and only if }  x,y \in \gamma^s \text{ for some } \gamma^s \in \Gamma^s, n=m.\]

By making use of these equivalence relations we can define a quotient tower  $\widetilde{\Delta_{e}}:=\Delta_{e}/\sim$, which has a quotient product structure $\widetilde{\Lambda}:=\Lambda/\sim$ with canonical projections $\widetilde{\pi_{\Delta_{e}}}:\Delta_{e} \to \widetilde{\Delta_{e}}$ and $\widetilde{\pi_{\Lambda}}:\Lambda \to \widetilde{\Lambda}$. We identify $\Lambda, \widetilde{\Lambda}$ with $\Delta_e\bigcap (\Lambda \times \{0\}),\widetilde{\Delta_e}\bigcap (\widetilde{\Lambda} \times \{0\})$, respectively. The quotient maps $\widetilde{F_e}: \widetilde{\Delta_{e}} \to \widetilde{\Delta_{e}}$, $\widetilde{g^{R^{e}}}: \widetilde{\Lambda} \to \widetilde{\Lambda}$, a quotient return time $\widetilde{R^e}:\widetilde{\Delta_e} \to \mathbb{N}$, and a quotient separation time $\widetilde{s}$ on $\widetilde{\Lambda}\times \widetilde{\Lambda}$ are defined via the following relations

\[\widetilde{\pi_{\Delta_e}} \circ F_e=\widetilde{F_e} \circ \widetilde{\pi_{\Delta_e}},\quad R^e=\widetilde{R^e} \circ \widetilde{\pi_{\Lambda}}, \quad \widetilde{g^{R^{e}}}=\widetilde{F_e}^{\widetilde{R^e}}, \quad s=\widetilde{s}\circ (\widetilde{\pi_{\Lambda}}, \widetilde{\pi_{\Lambda}}).\]

It follows from \cite{Young} that there exists a measure $m$ on $\widetilde{\Lambda}$ such that for any $x,y \in \widetilde{\Lambda}$ with $\widetilde{s}(x,y)\ge 1$, \begin{gather}\label{distortionofquotientinducedmap}
    \log \frac{\det D \widetilde{F_e}^{\widetilde{R^e}}(x)}{\det D \widetilde{F_e}^{\widetilde{R^e}}(y)} \le C \beta^{\widetilde{s}(\widetilde{F_e}^{\widetilde{R^e}}x,\widetilde{F_e}^{\widetilde{R^e}}y)},
\end{gather} where $\det D \widetilde{F_e}^{\widetilde{R^e}}$ is the Radon-Nikodym derivative of $\widetilde{F_e}^{\widetilde{R^e}}$ with respect to the measure $m$.

The set $(\widetilde{\Delta_e}, \widetilde{F_e})$, together with (\ref{distortionofquotientinducedmap}), is called an expanding quotient Young tower.

\begin{lemma}[See \cite{Young, Young2}]\label{resultsofyoungpaper}\par There are constants $C \ge 1$ and $0 < \beta < 1$, such that the following holds. \begin{enumerate}
    \item  There exists a probability distribution $\mu_{\widetilde{\Lambda}}$ on  $\widetilde{\Lambda}$,  which is constructed only from the measure $m$ and $\widetilde{F_e}^{\widetilde{R^e}}$, such that \begin{gather}\label{referencemeasureandsrb}
   (\widetilde{F_e}^{\widetilde{R^e}})_{*}\mu_{\widetilde{\Lambda}}=\mu_{\widetilde{\Lambda}}, \quad \frac{d\mu_{\widetilde{\Lambda}}}{dm}=C^{\pm 1}.
\end{gather}

A probability distribution on $\widetilde{\Delta_e}$, defined by \[\mu_{\widetilde{\Delta_e}}:=(\int \widetilde{R^e} d\mu_{\widetilde{\Lambda}})^{-1}\sum_{j}(\widetilde{F_e}^j)_{*}\big(\mu_{\widetilde{\Lambda}}|_{\{\widetilde{R^e}>j\}}\big),\] is an invariant measure, i.e., 
\begin{gather*}
     \widetilde{F_e}_{*}\mu_{\widetilde{\Delta_e}}=\mu_{\widetilde{\Delta_e}}.
\end{gather*}

\item Further, there exists a probability measure $\mu_{\Lambda}$ on $\Lambda$ such that \begin{gather}\label{propertyofsrb}
    \left(\widetilde{\pi_{\Lambda}}\right)_{*}\mu_{\Lambda}=\mu_{\widetilde{\Lambda}}, \quad (g^{R^e})_{*}\mu_{\Lambda}=\mu_{\Lambda}, \quad (\mu_{\Lambda})_{\gamma^u} \ll \Leb_{\gamma^u}, \quad \frac{d(\mu_{\Lambda})_{\gamma^u}}{d\Leb_{\gamma^u}}= C^{\pm 1},
    \end{gather} where $(\mu_{\Lambda})_{\gamma^u}$ is the conditional probability of $\mu_{\Lambda}$ on $\gamma^u \in \Gamma^u$. A probability measure on $\Delta_{e}$ defined by\[\mu_{\Delta_e}:=(\int R^e d\mu_{\Lambda})^{-1}\sum_{j}(F_e^j)_{*}\big(\mu_{\Lambda}|_{\{R^e>j\}}\big)\]
has the following properties \begin{gather*}
    \left(\widetilde{\pi_{\Delta_e}}\right)_{*}\mu_{\Delta_e}=\mu_{\widetilde{\Delta_e}},   \quad {\pi_e}_{*}\mu_{\Delta_e}=\mu_X, \quad {F_e}_{*}\mu_{\Delta_e}=\mu_{\Delta_e}.
\end{gather*} 

\item  Suppose that $\gcd{\{R^{e}\}}=N_e\ge N$. Now we define new towers \begin{gather*}
    \Delta_e':=\{(x,lN_e) \in \Lambda \times \mathbb{N}_0: 0 \le l < R^e(x)/N_e\},\\
    \widetilde{\Delta_e'}:=\{(x,lN_e) \in \widetilde{\Lambda} \times \mathbb{N}_0: 0 \le l < \widetilde{R^e}(x)/N_e\},
\end{gather*} which are  sub-towers of $\Delta_e$ and $\widetilde{\Delta_e}$, respectively. Then the maps $F_e^{N_e}:\Delta_e' \to \Delta_e'$ and $\widetilde{F_e}^{N_e}:\widetilde{\Delta_e'} \to \widetilde{\Delta_e'}$ preserve  probability measures \begin{gather*}
    \mu_{\Delta_e'}:=(\int R^e/N_e d\mu_{\Lambda})^{-1}\sum_{j}(F_e^{jN_e})_{*}\big(\mu_{\Lambda}|_{\{R^e/N_e>j\}}\big),\\
    \mu_{\widetilde{\Delta_e'}}:=(\int \widetilde{R^e}/N_e d\mu_{\widetilde{\Lambda}})^{-1}\sum_{j}(\widetilde{F_e}^{jN_e})_{*}\big(\mu_{\widetilde{\Lambda}}|_{\{\widetilde{R^e}/N_e>j\}}\big),
\end{gather*}respectively. Further, ${\pi_e}_{*}\mu_{\Delta_e'}$ is exactly $\mu_X$, since $(X, g, \mu_X)$ is mixing (see Definition \ref{cmz}).

\item A family of partitions $(\mathcal{Q}_k)_{k \ge 0}$ of $\Delta_{e}'$, $(\widetilde{\mathcal{Q}_k})_{k \ge 0}$ of $\widetilde{\Delta_{e}'}$, defined as \begin{gather*}
    \mathcal{Q}_0:=\{\Lambda_i \times \{lN_e\}, i \ge 1, 0\le l <R^{e}_{i}/N_e\}, \quad \mathcal{Q}_k:= \bigvee_{0 \le i \le k} (F_e^{N_e})^{-i} \mathcal{Q}_0,\\
    \widetilde{\mathcal{Q}_0}:=\{\widetilde{\Lambda_i} \times \{lN_e\}, i \ge 1, 0\le l < R^{e}_{i}/N_e\}, \quad \widetilde{\mathcal{Q}_k}:= \bigvee_{0 \le i \le k} (\widetilde{F_e}^{N_e})^{-i} \widetilde{\mathcal{Q}_0},
\end{gather*} satisfies the relations  \begin{gather*}
    \diam \left(\pi_e \circ F_e^{N_ek}(Q)\right) \le C\beta^{kN_e}, \quad \widetilde{\pi_{\Delta_e}}Q \in \widetilde{\mathcal{Q}_{2k}} \text{ for any }Q \in \mathcal{Q}_{2k}.
\end{gather*}

\item Finally, for any $n >2k\ge 2$,any $(\widetilde{Q_i})_{i \ge 1} \subseteq \widetilde{\mathcal{Q}_k}$, and any bounded function $\widetilde{h}: \widetilde{\Delta_e'} \to \mathbb{R}$, we have the following estimate of decay of correlations
\begin{equation}\label{decorrelationexpanding}
    \Big|\int \mathbbm{1}_{\bigcup_{i \ge 1} \widetilde{Q_i}}  \widetilde{h} \circ (\widetilde{F_e}^{N_e})^{n} d\mu_{\widetilde{\Delta_e'}}-\mu_{\widetilde{\Delta_e'}}\Big(\bigcup_{i \ge 1} \widetilde{Q_i}\Big) \int \widetilde{h} d\mu_{\widetilde{\Delta_e'}}\Big| \le C\beta^{n-2k} \mu_{\widetilde{\Delta_e'}}\Big(\bigcup_{i \ge 1} \widetilde{Q_i}\Big)||\widetilde{h}||_{\infty}.
\end{equation}

From (\ref{decorrelationexpanding}) and $\left(\widetilde{\pi_{\Delta_e}}\right)_{*}\mu_{\Delta_e}=\mu_{\widetilde{\Delta_e}}$, immediately follows an estimate of a rate of decay of correlations. Namely, for any $n>2k\ge 2$, any $(Q_i)_{i \ge 1} \subseteq \mathcal{Q}_k$, and any $\sigma(\bigcup_{k\ge 0}\mathcal{Q}_k)$-measurable function $h: \Delta_e' \to \mathbb{R}$,
\begin{equation}\label{decorrelationhyperbolic}
    \Big|\int \mathbbm{1}_{\bigcup_{i \ge 1} Q_i}  h \circ (F_e^{N_e})^{n} d\mu_{\Delta_e'}-\mu_{\Delta_e'}\Big(\bigcup_{i \ge 1} Q_i\Big) \int h d\mu_{\Delta_e'}\Big| \le C\beta^{n-2k} \mu_{\Delta_e'}\Big(\bigcup_{i \ge 1} Q_i\Big)||h||_{\infty}.
\end{equation}
\end{enumerate}
\end{lemma}

To make the exposition clearer all towers discussed in this subsection, and relations between them, are summarized in the diagrams below.

\[\begin{tikzcd}[arrows={-Stealth}]
\widetilde{\Delta_{e}'} \arrow{r}{\widetilde{F_{e}}^{N_{e}}} &\widetilde{\Delta_{e}'}&  \arrow{l}{\widetilde{\pi_{\Delta_{e}}}} \Delta'_e \arrow{r}{(F_e)^{N_e}} \arrow{d}{\pi_e }& \Delta'_e \arrow{r}{\text{inclusion}} \arrow{d}{\pi_e } &\Delta_e \arrow{r}{F_e} \arrow{d}{\pi_e }& \Delta_e \arrow{d}{\pi_e } & \Delta\arrow{r}{F} \arrow{d}{\pi }  & \Delta \arrow{d}{\pi }  \\%
&&X \arrow{r}{(f^R)^{N_e}} & X & X \arrow{r}{f^R} &X \arrow{r}{\text{inclusion}} & \mathcal{M}\arrow{r}{f} & \mathcal{M}
\end{tikzcd}
\]

\[\begin{tikzcd}[arrows={-Stealth}]
 &&  \Delta'_e \arrow{r}{(F_e)^{N_e}} \arrow{d}{\pi_e }& \Delta'_e \arrow{r}{\text{inclusion}} \arrow{d}{\pi_e } &\Delta_e \arrow{r}{F_e} \arrow{d}{\pi_e }& \Delta_e \arrow{d}{\pi_e } & \Delta\arrow{r}{F} \arrow{d}{\pi }  & \Delta \arrow{d}{\pi }  \\%
\Lambda\arrow{r}{F_{e}^{R^{e}}}&\Lambda\arrow{r}{\text{inclusion}}&X \arrow{r}{(f^R)^{N_e}} & X & X \arrow{r}{f^R} &X \arrow{r}{\text{inclusion}} & \mathcal{M}\arrow{r}{f} & \mathcal{M}
\end{tikzcd}
\]

\begin{remark} It is not difficult to prove that for any unstable manifold/disk $\gamma^u$ we have $(\mu_X)_{\gamma^u} \ll \Leb_{\gamma^u}$, where $(\mu_X)_{\gamma^u}$ is the conditional measure of $\mu_X$ on an unstable manifold/disk $\gamma^u$. This $\mu_X$ is a Sinai-Ruelle-Bowen measure (\textbf{SRB measure}).
\end{remark}

Thanks to all preparations above, we can construct now thicker hyperbolic and expanding Young towers.

\subsection{Thicker hyperbolic Young towers for $(\Delta_m, F^{R_m}, \mu_{\Delta_m})$} 
\begin{lemma}\label{mixingmlarge}
$(\Delta_m, F^{R_m}, \mu_{\Delta_m})$ is K-mixing for sufficiently large $m\ge 1$. (Observe, that it is not true for a small $m$). 
\end{lemma}
\begin{proof}
By Definition \ref{cmz}, $\gcd \{R\}=1$. Suppose that $\gcd\{i_1, i_2, \cdots, i_k\}=1$. Then we have $\gcd\{R\}=\gcd\{R|_{\{R=i_1\}}, R|_{\{R=i_2\}}, \cdots, R|_{\{R=i_k\}}\}=1$. Choose now $m \ge \max\{i_1,i_2, \cdots, i_k\}$. Then $\gcd\{R_m\}$ $=\gcd\{R\}=1$. Besides, $F^{R_m^{\min \{R,m+1\}}}|_X=g:X \to X$ is K-mixing. Therefore, $(\Delta_m, F^{R_m}, \mu_{\Delta_m})$ is K-mixing, and thus mixing.
\end{proof}

From now on we will consider only large enough $m$. Define a new tower (called a thicker hyperbolic Young tower) as \begin{gather*}
    \Delta_{e,m}=\bigcup_{i \ge 1} \bigcup_{0\le j \le R^{e}_{i}-1} \Lambda_i \times \{j\} \times \{0,1,2,\cdots, m_{i,j}\},
\end{gather*}where $m_{i,j}=\min\{m,R\big(g^j(x)\big)-1\}$ for $0\le j \le R^{e}_{i}-1$ and any $x\in \Lambda_i$. $\Delta_{e,m}$ can be visualized as ``inserting" extra layers into $\Delta_e$. Hence $\Delta_{e.m}$ is thicker than $\Delta_e$.
\begin{lemma}
$m_{i,j}$ is well-defined.
\end{lemma}
\begin{proof}
To prove this lemma we just need to show that $R\big(g^j(x)\big)$ does not depend on any $x\in \Lambda_i$. It follows from $j \le R^{e}_{i}-1$ that $g^{R^{e}_{i}-j}$ is smooth on $g^j(\Lambda_i)$. Therefore $g^j(\Lambda_i)\subseteq \mathbb{S}^c$. By Assumption \ref{assumption}, $\bigcup_{i\ge 1}\partial X_i\subseteq \mathbb{S}$. Thus $g^j(\Lambda_i) \subsetneq X_k$ for some $k\ge 1$. Since $R$ is constant on $X_k$, then $R(y)=R|_{X_k}$ for any $y \in g^j(\Lambda_i)$, and $m_{i,j}$ is well-defined.
\end{proof}

We identify $\Lambda_i \times \{j\}\times \{0\}$ with $\Lambda_i \times \{j\}$, and $\Lambda \times \{0\}\times \{0\}$ with $\Lambda$. Then $\Delta_e$ is a sub-tower of $\Delta_{e,m}$. Define now a map $F_{e,m}: \Delta_{e,m}\to \Delta_{e,m}$, so that for any $x\in \Lambda_i$ and for some $i \ge 1$
\begin{gather*}
    F_{e,m}(x,j,k):=\begin{cases}
 (x,j,k+1),      &j \le R^{e}_{i}-1, k < m_{i,j} \\
 (x,j+1,0),  &j < R^{e}_{i}-1, k= m_{i,j} \\
\left(g^{R^{e}_{i}}(x),0,0\right),  & j=R^{e}_{i}-1, k= m_{i,j}\\
\end{cases}.
\end{gather*}

The set $\Lambda$, as the base of $\Delta_{e,m}$, has a return time $R_{e,m}|_{\Lambda_i}:=\sum_{j < R^{e}_{i}}(1+m_{i,j})$, such that $F_{e,m}^{R_{e,m}}=g^{R^e}: \Lambda \to \Lambda$ is the induced map for the tower $(\Delta_{e,m}, F_{e,m})$.

Define a probability measure $\mu_{\Delta_{e,m}}$ on $\Delta_{e,m}$ as \begin{align*}
    \mu_{\Delta_{e,m}}:=(\int R_{e,m} d\mu_{\Lambda})^{-1}\sum_{j}(F_{e,m}^j)_{*}\big(\mu_{\Lambda}|_{\{R_{e,m}>j\}}\big).
\end{align*} A projection $\pi_{e,m}: \Delta_{e,m} \to \Delta_m$ is defined by\begin{align}\label{projectiondeltam}
    \pi_{e,m}(x,j,k):=F^k \circ g^j(x).
\end{align}
\begin{lemma}\label{returntailofthickertower}
For any $\gamma \in \Gamma^u$ we have $R^e \le R_{e,m} \le 2mR^e$, and $\Leb_{\gamma}\left(R_{e,m}>n\right) \le C\beta^{n/(2m)}$.
\end{lemma}
\begin{proof}
The first inequality is obvious, since $m_{i,j}\le m$ for any $i,j$. By (\ref{taildecay}), $\Leb_{\gamma}\left(R_{e,m}>n\right)\le\Leb_{\gamma}\left(R^e>n/2m\right)\le C\beta^{n/(2m)} $.
\end{proof}
\begin{lemma}\label{semiconjugacyforthickeryoungtower}
$(\pi_{e,m})_*\mu_{\Delta_{e,m}}=\mu_{\Delta_m}$ and $\pi_{e,m}$ is a semi-conjugacy, i.e., $\pi_{e,m} \circ F_{e,m}= F^{R_m} \circ \pi_{e,m}$.   
\end{lemma}
\begin{proof}
At first we prove a semi-conjugacy. For any $(x,j,k)\in \Delta_{e,m}$, where $x\in \Lambda_i$ for some $i \ge 1$, suppose that $k < m_{i,j}$, \begin{gather*}
\pi_{e,m} \circ F_{e,m}(x,j,k)=\pi_{e,m} (x,j,k+1)=F^{k+1}\circ g^j(x),\\
F^{R_m}\circ \pi_{e,m}(x,j,k)=F^{R_m}\circ F^{k}\circ g^j(x)=F^{k+1}\circ g^j(x),
\end{gather*}where the last equality holds because $\pi_{\mathbb{N}}\big(F^k \circ g^i(x)\big)=k<m_{i,j}$.

Suppose that $k=m_{i,j}$, $j\le R^e_i-1$,\begin{gather*}
\pi_{e,m} \circ F_{e,m}(x,j,k)=\pi_{e,m} (x,j+1,0)= g^{j+1}(x),\\
F^{R_m}\circ \pi_{e,m}(x,j,k)=F^{R_m}\circ F^{k}\circ g^j(x)=g\circ g^j(x)=g^{j+1}(x),
\end{gather*}where $F^{R_m}\circ F^{k}=g$ follows from the fact that $F^k \circ g^j(x)$ is already on the roof of $\Delta_m$. Therefore $\pi_{e,m}$ is a semi-conjugacy.

Next we prove that $(\pi_{e,m})_*\mu_{\Delta_{e,m}}=\mu_{\Delta_m}$. Denote $\sum_{j}(F_e^j)_{*}(\mu_{\Lambda}|_{R^e>j})$ by $\nu$. Then for any $A\times \{k\}\subseteq \Delta_m$, where $A\subseteq \{R=i\}$ for some $i\ge 1$, we have \begin{align*}
\mu_{\Delta_m}(A\times \{k\})&=(\int \min\{R, m+1\} d\mu_X)^{-1}\mu_{X}(A)\\
&=(\int \min\{R, m+1\} d\mu_X)^{-1}\mu_{\Delta_e}(\pi^{-1}_eA)\\
&=(\int \min\{R, m+1\} d\mu_X)^{-1}\mu_{\Delta_e}\{(x,n)\in \Lambda \times \{0,1,2,\cdots\}:g^n(x)\in A\}\\
&=(\int \min\{R, m+1\} d\mu_X)^{-1}(\int R^ed\mu_{\Lambda})^{-1}\nu( \pi_e^{-1}A).
\end{align*}

On the other hand,
\begin{align*}
    (\pi_{e,m})_{*}\mu_{\Delta_{e,m}}(A\times\{k\})&=\mu_{\Delta_{e,m}}\{\pi^{-1}_{e,m}(A\times \{k\})\}\\
    &=\mu_{\Delta_{e,m}}\{(x,j,i)\in \Delta_{e,m}: F^ig^j(x)\in A\times \{k\}\}\\
    &=\mu_{\Delta_{e,m}}\{(x,j,k)\in \Delta_{e,m}: g^j(x)\in A\}=(\int R_{e,m}d\mu_{\Lambda})^{-1}\nu( \pi_e^{-1}A).
\end{align*}

Since $m_{i,j}=\min\{m,R\big(g^j(x)\big)-1\}$ for $0\le j \le R^{e}_{i}-1, x \in \Lambda_i$, we have \begin{align*}
    \int \min\{R, m+1\} d\mu_{X}&=(\int R^ed\mu_{\Lambda})^{-1}\int \min\{R, m+1\}\circ \pi_e d\nu\\
    &=(\int R^ed\mu_{\Lambda})^{-1}\sum_{i \ge 1}\sum_{j < R^{e}_{i}}\min\{R, m+1\}|_{g^j\Lambda_i} \mu_{\Lambda}(\Lambda_i)\\
    &=(\int R^ed\mu_{\Lambda})^{-1}\sum_{i \ge 1}\sum_{j < R^{e}_{i}}(1+m_{i,j}) \mu_{\Lambda}(\Lambda_i)\\
    &=(\int R^ed\mu_{\Lambda})^{-1}\int R_{e,m}d\mu_{\Lambda}.
\end{align*}

Then \begin{equation*}
    (\int \min\{R,m+1\} d\mu_X)^{-1}(\int R^ed\mu_{\Lambda})^{-1}=(\int R_{e,m}d\mu_{\Lambda})^{-1}.
\end{equation*} 

Therefore, $(\pi_{e,m})_{*}\mu_{\Delta_{e,m}}(A\times \{k\})=\mu_{\Delta_m}(A\times \{k\})$ for any measurable set $A\times \{k\}\subseteq \Delta_m$. Since such set generates the $\sigma$-algebra of $\Delta_m$, then $(\pi_{e,m})_{*}\mu_{\Delta_{e,m}}=\mu_{\Delta_m}$.
\end{proof}

\subsection{Thicker expanding quotient Young tower for $(\Delta_m, F^{R_m}, \mu_{\Delta_m})$}

Introduce an equivalence relation $\sim$ on $\Delta_{e,m}$ as
\[(x,j,k) \sim (y,i,l) \text{ if and only if }  x,y \in \gamma^s \text{ for some } \gamma^s \in \Gamma^s, j=i, k=l.\]

Using this equivalence relation, define a quotient tower  $\widetilde{\Delta_{e,m}}:=\Delta_{e,m}/\sim$,  with canonical projections $\widetilde{\pi_{\Delta_{e,m}}}:\Delta_{e,m} \to \widetilde{\Delta_{e,m}}$. We identify $\widetilde{\Lambda}$ with $\widetilde{\Delta_{e,m}}\bigcap (\widetilde{\Lambda} \times \{0\} \times \{0\})$. A quotient map $\widetilde{F_{e,m}}: \widetilde{\Delta_{e,m}} \to \widetilde{\Delta_{e,m}}$, and a quotient return time $\widetilde{R_{e,m}}:\widetilde{\Lambda} \to \mathbb{N}$ are defined via the following relations\begin{gather}\label{thinkeryoungtowerreturntime}
    \widetilde{\pi_{\Delta_{e,m}}} \circ F_{e,m}=\widetilde{F_{e,m}} \circ \widetilde{\pi_{\Delta_{e,m}}},\quad R_{e,m}=\widetilde{R_{e,m}} \circ \widetilde{\pi_{\Lambda}}.
\end{gather} They satisfy to $\widetilde{g^{R^{e}}}=\widetilde{F_{e,m}}^{\widetilde{R_{e,m}}}$, which is easy to prove from the construction of $\Delta_{e,m}$ and $F_{e,m}$. 

Define a probability measure $\mu_{\widetilde{\Delta_{e,m}}}$ on $\widetilde{\Delta_{e,m}}$ as \begin{align*}
    \mu_{\widetilde{\Delta_{e,m}}}:=(\int \widetilde{R_{e,m}} d\mu_{\widetilde{\Lambda}})^{-1}\sum_{j}(\widetilde{F_{e,m}}^j)_{*}\big(\mu_{\widetilde{\Lambda}}|_{\{\widetilde{R_{e,m}}>j\}}\big).
\end{align*} 

Since $\widetilde{g^{R^{e}}}=\widetilde{F_{e}}^{\widetilde{R^{e}}}$, it is easy to see that $(\widetilde{\pi_{\Delta_{e,m}}})_*\mu_{\Delta_{e,m}}=\mu_{\widetilde{\Delta_{e,m}}}$. By (\ref{distortionofquotientinducedmap}) we obtain the expression for distortions that for any $x,y \in \widetilde{\Lambda}$  with $\widetilde{s}(x,y)\ge 1$,,
\begin{gather}\label{distortionofthickerquotientinducedmap}
    \log \frac{\det D \widetilde{F_{e,m}}^{\widetilde{R_{e,m}}}(x)}{\det D \widetilde{F_{e,m}}^{\widetilde{R_{e,m}}}(y)} \le C \beta^{\widetilde{s}(\widetilde{F_{e,m}}^{\widetilde{R_{e,m}}}x,\widetilde{F_{e,m}}^{\widetilde{R_{e,m}}}y)},
\end{gather}where $\det D \widetilde{F_{e,m}}^{\widetilde{R_{e,m}}}$ is the Radon-Nikodym derivative of $\widetilde{F_{e,m}}^{\widetilde{R_{e,m}}}$ with respect to the measure $m$ on $\widetilde{\Lambda}$.

Therefore $(\widetilde{\Delta_{e,m}}, \widetilde{F_{e,m}})$, together with (\ref{distortionofthickerquotientinducedmap}), is the thicker expanding quotient Young tower for $(\Delta_m, F^{R_m}, \mu_{\Delta_m})$.

\subsection{Decay of correlations}
Suppose that $\gcd{\{R_{e,m}\}}=N_{e,m}$. Now we can define new towers \begin{gather*}
    \Delta_{e,m}':=\{(x,j,k) \in \Delta_{e,m}: x \in \Lambda_i, k+\sum_{0 \le l< j}(1+m_{i,l}) \equiv 0 \pmod{N_{e,m}}\},\\
    \widetilde{\Delta_{e,m}'}:=\{(x,j,k) \in \widetilde{\Delta_{e,m}}: x \in \Lambda_i, k+\sum_{0 \le l<j}(1+m_{i,l}) \equiv 0 \pmod{N_{e,m}}\},
\end{gather*}which are sub-towers of $\Delta_{e,m}$ and $\widetilde{\Delta_{e,m}}$, respectively. Then the dynamics $F_{e,m}^{N_{e,m}}:\Delta_{e,m}' \to \Delta_{e,m}'$ and $\widetilde{F_{e,m}}^{N_{e,m}}:\widetilde{\Delta_{e,m}'} \to \widetilde{\Delta_{e,m}'}$ preserve  probability measures \begin{gather*}
    \mu_{\Delta_{e,m}'}:=(\int R_{e,m}/N_{e,m} d\mu_{\Lambda})^{-1}\sum_{j}(F_{e,m}^{jN_{e,m}})_{*}\big(\mu_{\Lambda}|_{\{R_{e,m}/N_{e,m}>j\}}\big),\\
    \mu_{\widetilde{\Delta_{e,m}'}}:=(\int \widetilde{R_{e,m}}/N_{e,m} d\mu_{\widetilde{\Lambda}})^{-1}\sum_{j}(\widetilde{F_{e,m}}^{jN_{e,m}})_{*}\big(\mu_{\widetilde{\Lambda}}|_{\{\widetilde{R_{e,m}}/N_{e,m}>j\}}\big),
\end{gather*}respectively, and are mixing. Clearly, $(\widetilde{\pi_{\Delta_{e,m}}})_*\mu_{\Delta'_{e,m}}=\mu_{\widetilde{\Delta'_{e,m}}}$. Since $(\Delta_{m}, F^{R_m}, \mu_{\Delta_m})$ is mixing (see Lemma \ref{mixingmlarge}), then, by using the same argument as that on page 607 of \cite{Young}, we have \begin{align*}
    (\pi_{e,m})_{*}\mu_{\Delta_{e,m}'}=\mu_{\Delta_{m}}.
\end{align*}

The following diagrams summarize all towers discussed so far.
\[ \begin{tikzcd}[arrows={-Stealth}]
\widetilde{\Delta_{e,m}'} \arrow{r}{\widetilde{F_{e,m}}^{N_{e,m}}} &\widetilde{\Delta_{e,m}'}  & \arrow{l}{\widetilde{\pi_{\Delta_{e,m}}}}\Delta'_{e,m} \arrow{r}{F_{e,m}^{N_{e,m}}} \arrow{d}{\pi_{e,m} }   & \Delta'_{e,m} \arrow{d}{\pi_{e,m} }  \arrow{r}{\text{inclusion}}&    \Delta_{e,m} \arrow{r}{F_{e,m}} \arrow{d}{\pi_{e,m} }   & \Delta_{e,m} \arrow{d}{\pi_{e,m} } &   &  & \\ &&
\Delta_m \arrow{r}{(F^{R_m})^{N_{e,m}}} &\Delta_m &  \Delta_m \arrow{r}{F^{R_m}}  & \Delta_m \arrow{r}{\text{inclusion}} & \Delta \arrow{d}{\pi } \arrow{r}{F }  & \Delta \arrow{d}{\pi }  &\\%
&&&& & &\mathcal{M}\arrow{r}{f} & \mathcal{M}
\end{tikzcd}\]

\[ \begin{tikzcd}[arrows={-Stealth}]
 & & &\Delta'_{e,m} \arrow{r}{F_{e,m}^{N_{e,m}}} \arrow{d}{\pi_{e,m} }   & \Delta'_{e,m} \arrow{d}{\pi_{e,m} }  \arrow{r}{\text{inclusion}}&    \Delta_{e,m} \arrow{r}{F_{e,m}} \arrow{d}{\pi_{e,m} }   & \Delta_{e,m} \arrow{d}{\pi_{e,m} } &   &  & \\ \Lambda\arrow{r}{F_{e,m}^{R_{e,m}}}&\Lambda\arrow{r}{\text{inclusion}}&X\arrow{r}{\text{inclusion}}&
\Delta_m \arrow{r}{(F^{R_m})^{N_{e,m}}} &\Delta_m &  \Delta_m \arrow{r}{F^{R_m}}  & \Delta_m \arrow{r}{\text{inclusion}} & \Delta \arrow{d}{\pi } \arrow{r}{F }  & \Delta \arrow{d}{\pi }  &\\%
&&&&& & &\mathcal{M}\arrow{r}{f} & \mathcal{M}
\end{tikzcd}\]

Now the families of partitions $(\mathcal{Q}^m_{k})_{k \ge 0}$ of $\Delta_{e,m}'$, and $(\widetilde{\mathcal{Q}^m_{k}})_{k \ge 0}$ of $\widetilde{\Delta_{e,m}'}$ are defined as \begin{gather*}
    \mathcal{Q}^m_{0}:=\{\Lambda_i \times \{j\} \times \{l\} \subseteq  \Delta_{e,m}': i \ge 1, j \ge 0, l\ge 0\}, \quad \mathcal{Q}^m_k:= \bigvee_{0 \le i \le k} (F_{e,m}^{N_{e,m}})^{-i} \mathcal{Q}^m_0,\\
    \widetilde{\mathcal{Q}^m_{0}}:=\{\widetilde{\Lambda_i}\times \{j\} \times \{l\} \subseteq  \widetilde{\Delta_{e,m}'}: i \ge 1, j \ge 0, l\ge 0\}, \quad \widetilde{\mathcal{Q}^m_k}:= \bigvee_{0 \le i \le k} (\widetilde{F_{e,m}}^{N_{e,m}})^{-i} \widetilde{\mathcal{Q}^m_0}.
\end{gather*}Then we have the following results.
\begin{lemma}\label{decayofcorrelationofthinkerhyperbolicyoungtower}
There is a constant $C_{\alpha}>0$, such that $\diam \left(\pi \circ \pi_{e,m} \circ F_{e,m}^{N_{e,m}k}(Q)\right) \le C_{\alpha}\beta^{\frac{\alpha N_{e,m} k}{2m}}$ and $\widetilde{\pi_{\Delta_{e,m}}}Q \in \widetilde{\mathcal{Q}^m_{2k}}$  for any $Q \in \mathcal{Q}^m_{2k}$ and any $k > m+1$. Moreover, there are constants $\beta_m\in (0,1)$ and $C_m>0$, such that for any $n>2k\ge 2$, any $(\widetilde{Q_i})_{i \ge 1} \subseteq \widetilde{\mathcal{Q}^m_k}$, and any bounded function $\widetilde{h}: \widetilde{\Delta_{e,m}'} \to \mathbb{R}$, we have the following estimate for decay of correlations
\begin{equation}\label{decorrelationexpandingthickertower}
    \Big|\int \mathbbm{1}_{\bigcup_{i \ge 1} \widetilde{Q_i}}  \widetilde{h} \circ (\widetilde{F_{e,m}}^{N_{e,m}})^{n} d\mu_{\widetilde{\Delta_{e,m}'}}-\mu_{\widetilde{\Delta_{e,m}'}}\Big(\bigcup_{i \ge 1} \widetilde{Q_i}\Big) \int \widetilde{h} d\mu_{\widetilde{\Delta_{e,m}'}}\Big| \le C_m\beta_m^{n-2k} \mu_{\widetilde{\Delta_{e,m}'}}\Big(\bigcup_{i \ge 1} \widetilde{Q_i}\Big)||\widetilde{h}||_{\infty}.
\end{equation}

Besides, we also have the following estimate for decay of correlations. For any $n>2k\ge 2$, any $(Q_i)_{i \ge 1} \subseteq \mathcal{Q}^m_k$, and for any $\sigma(\bigcup_{k\ge 0}\mathcal{Q}^m_k)$-measurable function $h: \Delta_{e,m}' \to \mathbb{R}$
\begin{align}\label{decorrelationhyperbolicthickertower}
    \Big|\int \mathbbm{1}_{\bigcup_{i \ge 1} Q_i}  h \circ (F_{e,m}^{N_{e,m}})^{n} d\mu_{\Delta_{e,m}'}-\mu_{\Delta_{e,m}'}\Big(\bigcup_{i \ge 1} Q_i\Big) \int h d\mu_{\Delta_{e,m}'}\Big| \le C_m\beta_m^{n-2k} \mu_{\Delta_{e,m}'}\Big(\bigcup_{i \ge 1} Q_i\Big)||h||_{\infty}.
\end{align}
\end{lemma}\begin{proof}
Since the return time for $\widetilde{\Delta_{e,m}'}$ is $\widetilde{R_{e,m}}/N_{e,m}$, we have $\gcd\{\widetilde{R_{e,m}}/N_{e,m}\}=1$. The return map for $\widetilde{F_{e,m}}^{N_{e,m}}: \widetilde{\Delta_{e,m}'}\to \widetilde{\Delta_{e,m}'}$ is  $\big(\widetilde{F_{e,m}}^{N_{e,m}}\big)^{\widetilde{R_{e,m}}/N_{e,m}}: \widetilde{\Lambda}\to \widetilde{\Lambda}$, and it has the distortion:\begin{align*}
    \log \frac{\det D \big(\widetilde{F_{e,m}}^{N_{e,m}}\big)^{\widetilde{R_{e,m}}/N_{e,m}}(x)}{\det D \big(\widetilde{F_{e,m}}^{N_{e,m}}\big)^{\widetilde{R_{e,m}}/N_{e,m}}(y)} \le C \beta^{\widetilde{s}\Big(\big(\widetilde{F_{e,m}}^{N_{e,m}}\big)^{\widetilde{R_{e,m}}/N_{e,m}}(x),\big(\widetilde{F_{e,m}}^{N_{e,m}}\big)^{\widetilde{R_{e,m}}/N_{e,m}}(y)\Big)}.
\end{align*}

Therefore, $(\widetilde{\Delta_{e,m}'}, \widetilde{F_{e,m}}^{N_{e,m}})$ is a mixing expanding Young tower. Besides, using (\ref{thinkeryoungtowerreturntime}), (\ref{propertyofsrb}) and (\ref{referencemeasureandsrb}), we have \begin{align*}
    m\{x\in \widetilde{\Lambda}: \widetilde{R_{e,m}}(x)/N_{e,m}>n\}&\le C \mu_{\widetilde{\Lambda}}\{x\in \widetilde{\Lambda}: \widetilde{R_{e,m}}(x)/N_{e,m}>n\}\\
    &=C \left(\widetilde{\pi_{\Lambda}}\right)_{*}\mu_{\Lambda}\{x\in \widetilde{\Lambda}: \widetilde{R_{e,m}}(x)/N_{e,m}>n\}\\
    &=C \mu_{\Lambda}\{x\in \Lambda: \widetilde{R_{e,m}}\circ \widetilde{\pi_{\Lambda}}(x)/N_{e,m}>n\}\\
    &=C \mu_{\Lambda}\{x\in \Lambda: R_{e,m}(x)/N_{e,m}>n\}\\
    &=C \int \mu_{\gamma^u}\{x\in \gamma^u: R_{e,m}(x)/N_{e,m}>n\}d\mu_{\Lambda}\\
    &\le C^2 \int \Leb_{\gamma^u}\{x\in \gamma^u: R_{e,m}(x)/N_{e,m}>n\}d\mu_{\Lambda}\le C^3 \big(\beta^{\frac{N_{e,m}}{2m}}\big)^n,
\end{align*}where we applied Lemma \ref{returntailofthickertower} to the last inequality. By making use of Theorem 2 in \cite{Young2} we get (\ref{decorrelationexpandingthickertower}).  Then (\ref{decorrelationhyperbolicthickertower}) 
 follows from (\ref{decorrelationexpandingthickertower}), (\ref{thinkeryoungtowerreturntime}) and $\left(\widetilde{\pi_{\Delta_{e,m}}}\right)_{*}\mu_{\Delta'_{e,m}}=\mu_{\widetilde{\Delta'_{e,m}}}$.

Next we estimate $\diam \left(\pi \circ \pi_{e,m} \circ F_{e,m}^{N_{e,m}k}(Q)\right)$. For any $\hat{\gamma}^s \subseteq Q$ (here $\hat{\gamma}^s=\gamma^s \times \{l'\} \times \{l\}$ for some $\gamma^s \subseteq \Lambda$ and $l',l \ge 0$), assume that $\pi_{\mathbb{N}}\big(\pi_{e,m}\circ F_{e,m}^{N_{e,m}k}(\hat{\gamma}^s)\big)=j\le m$, $j' \in [0, m]$ is the first non-negative number such that $\pi_{e,m}\circ F^{j'}_{e,m}(\hat{\gamma}^s) \subseteq X$, and the disks $\pi \circ \pi_{e,m} \circ F_{e,m}^{j'}(\hat{\gamma}^s), \pi \circ \pi_{e,m} \circ F_{e,m}^{j'+1}(\hat{\gamma}^s), \cdots, \pi \circ \pi_{e,m} \circ F_{e,m}^{N_{e,m}k}(\hat{\gamma}^s)$ visit $X$ exactly $q'$ times. Then 
\[q'+1  \ge \frac{N_{e,m}k+1}{m+1}\ge N_{e,m}k/(2m), \quad g^{q'-1} \circ \pi_{e,m}\circ F^{j'}_{e,m}(\hat{\gamma}^s)= \pi_{e,m}\circ F^{N_{e,m}k-j}_{e,m}(\hat{\gamma}^s) \subseteq  X\]
Then by Assumption \ref{assumption} and by Definition \ref{cmz}, there is a constant $C_{\alpha}>0$, such that \begin{align*}
     \diam \big(\pi \circ \pi_{e,m} \circ F_{e,m}^{N_{e,m}k}(\hat{\gamma}^s)\big)&\le C \diam \big(\pi \circ \pi_{e,m} \circ F_{e,m}^{N_{e,m}k-j}(\hat{\gamma}^s) \big)^{\alpha}\\
     &\le C^{1+\alpha} \beta^{(q'-1)\alpha} \diam \big(\pi \circ \pi_{e,m} \circ F_{e,m}^{j'}(\hat{\gamma}^s) \big)^{\alpha} \le C_{\alpha}\beta^{\frac{N_{e,m}k\alpha}{2m}}/2.
\end{align*}

On the other hand, for any $\hat{\gamma}^u \subseteq Q$ (here $\hat{\gamma}^u=\gamma^u \times \{l'\} \times \{l\}$ for some $\gamma^u \subseteq \Lambda$ and same $l',l \ge 0$), suppose that $\pi_{\mathbb{N}}\big(\pi_{e,m}\circ F_{e,m}^{N_{e,m}k}(\hat{\gamma}^u)\big)=i'\le m$, $\pi_{\mathbb{N}}\big(\pi_{e,m}\circ F_{e,m}^{N_{e,m}2k}(\hat{\gamma}^u)\big)=i\le m$, and the disks $\pi \circ \pi_{e,m} \circ F_{e,m}^{N_{e,m}k-i'}(\hat{\gamma}^u),\pi \circ \pi_{e,m} \circ F_{e,m}^{N_{e,m}k-i'+1}(\hat{\gamma}^u), \cdots, \pi \circ \pi_{e,m} \circ F_{e,m}^{N_{e,m}2k}(\hat{\gamma}^u)$ visit $X$ exactly $q$ times, then $q\ge \frac{N_{e,m}k+i'+1}{m+1}\ge N_{e,m}k/(2m)$ and \begin{gather*}
     \pi_{e,m}\circ F_{e,m}^{N_{e,m}k-i'}(\hat{\gamma}^u)\subseteq X, \quad  \pi_{e,m} \circ F_{e,m}^{N_{e,m}k-i'}(\hat{\gamma}^u)=g^{-(q-1)}\pi_{e,m}\circ F_{e,m}^{N_{e,m}2k-i}(\hat{\gamma}^u) \subseteq X.
\end{gather*} 

By Assumption \ref{assumption} and by Definition \ref{cmz}, there is a constant $C_{\alpha}>0$ such that \begin{align*}
    \diam \big(\pi \circ \pi_{e,m} \circ F_{e,m}^{N_{e,m}k}(\hat{\gamma}^u)\big)&\le C \diam \big(\pi \circ \pi_{e,m} \circ F_{e,m}^{N_{e,m}k-i'}(\hat{\gamma}^u) \big)^{\alpha}\\
    &\le C^{1+\alpha}\beta^{\alpha (q-1)}\diam \big(\pi \circ \pi_{e,m} \circ F_{e,m}^{2N_{e,m}k-i}(\hat{\gamma}^u)\big)^{\alpha}\le C_{\alpha} \beta^{\frac{N_{e,m}k\alpha}{2m}}/2.
\end{align*}

Finally, since any $\hat{\gamma}^u, \hat{\gamma}^s\subseteq Q$ intersect exactly at one point, then for any $x,y \in Q$ there are $o\in \Lambda$ and $\hat{\gamma}^u=\gamma^u(o) \times \{l'\} \times \{l\} \subseteq Q, \hat{\gamma}^s=\gamma^u(o) \times \{l'\} \times \{l\} \subseteq Q$, such that $x \in \hat{\gamma}^u$, $y\in \hat{\gamma}^s$, and \begin{align*}
    d\Big(\pi &\circ \pi_{e,m} \circ  F_{e,m}^{N_{e,m}k}(x), \pi \circ \pi_{e,m} \circ F_{e,m}^{N_{e,m}k}(y)\Big)\\
    &\le \diam \Big(\pi \circ \pi_{e,m} \circ F_{e,m}^{N_{e,m}k}\big(\hat{\gamma}^u\big)\Big)+\diam \Big(\pi \circ \pi_{e,m} \circ F_{e,m}^{N_{e,m}k}\big(\hat{\gamma}^s\big)\Big)\le C_{\alpha} \beta^{\frac{N_{e,m}k\alpha}{2m}}.\\
\end{align*} 

Therefore, $\diam \left(\pi \circ \pi_{e,m} \circ F_{e,m}^{N_{e,m}k}(Q)\right)\le  C_{\alpha} \beta^{\frac{N_{e,m}k\alpha}{2m}}$.\end{proof}

\section{Poisson limit laws for non-mixing hyperbolic Young towers}\label{section6} Now we are ready to present sufficient conditions for (\ref{inducingforanym}). The approach for Poisson approximations, developed in \cite{Su, peneijm, peneetds}, works for mixing hyperbolic Young towers. However, from the previous section, we know that our Young tower for $(\Delta_{e,m}, F_{e,m}, \mu_{\Delta_{e,m}})$ is generally non-mixing. In this section we will prove Poisson limit laws for the dynamics $F^{R_m}$, which can be described by the non-mixing hyperbolic Young tower $(\Delta_{e,m}, F_{e,m}, \mu_{\Delta_{e,m}})$. 

For any $i \ge 0$, we let
\begin{gather*}
    X_i:=\mathbbm{1}_{\pi_{\Delta_m}\pi^{-1}S_r} \circ (F^{R_m})^i, \\
    \pmb{X}_i:=\Big(\mathbbm{1}_{\pi_{\Delta_m}\pi^{-1}S_r} \circ (F^{R_m})^{N_{e,m}i}, \mathbbm{1}_{\pi_{\Delta_m}\pi^{-1}S_r} \circ (F^{R_m})^{N_{e,m}i+1}, \cdots, \mathbbm{1}_{\pi_{\Delta_m}\pi^{-1}S_r} \circ (F^{R_m})^{N_{e,m}(i+1)-1}\Big).
\end{gather*} 

Observe that $(\pmb{X}_i)_{i\ge 0}$ is stationary, since $\mu_{\Delta_m}$ is $(F^{R_m})^{N_{e,m}}$-invariant. Denote by $\{\hat{\pmb{X}}_i\}_{i\ge 0}$ i.i.d. random vectors defined on a probability space $(\hat{\Omega}, \hat{\mathbb{P}})$, such that for each $i \ge 0$ \begin{align*}
    \pmb{X}_i=_d\hat{\pmb{X}}_i,
\end{align*} i.e., they have the same distribution. Throughout this section the notation $h(\underbrace{\sbullet[.5], \sbullet[.5], \cdots, \sbullet[.5]}_\text{k})$ means that a function $h$ is defined on $\mathbb{R}^k$ for some $k\ge 1$, and $h \in [0,1]$ means that a function $h$ takes values in $[0,1]$. Further, $\mathbb{E}$ is the expectation of $\mu_{\Delta_m} \otimes \hat{\mathbb{P}}$. Denote by $\pmb{0}$ the zero vector in $\{0,1\}^{N_{e,m}}$.  For any vector $\pmb a \in \{0,1\}^{N_{e,m}}$, $\pmb a\ge 1$ means that at least one of the coordinates of $\pmb a$ is not zero, and $\pmb a=\pmb{0}$ means that all coordinates of $\pmb a$ are zero.

The scheme of our proof can be roughly described as follows. In order to give sufficient conditions for (\ref{inducingforanym}), we will approximate $(X_i)_{i \ge 0}$ by i.i.d. random variables in Lemma \ref{poissonappro}. To achieve this, we first approximate $(X_i)_{i\ge 0}$ by i.i.d. random vectors $(\hat{\pmb{X}}_i)_{i\ge 0}$ in Lemma \ref{poissonappro1}, then approximate the i.i.d. random vectors $(\hat{\pmb{X}}_i)_{i\ge 0}$ by i.i.d. random variables in Lemma \ref{poissonappro2}. Now we turn to the proofs.

\begin{lemma}\label{poissonappro1}
For any $n\ge 1$ and any integer $p\in (0,n)$,
\begin{align*}
    \sup_{h\in[0,1]}\Big|\mathbb{E}\Big[h(\pmb{X}_0,\pmb{X}_1, \cdots, \pmb{X}_n)-h(\hat{\pmb{X}}_0, \hat{\pmb{X}}_1, \cdots, \hat{\pmb{X}}_n)\Big]\Big|\precsim_m R_1+R_2+R_3, 
    \end{align*}
    where
    \begin{align*}
    &R_1:=\sum_{0 \le l \le n-p}\sup_{\pmb{a}\ge 1}\sup_{h\in[0,1]} \Big| \mathbb{E}\Big[\mathbbm{1}_{\pmb{X}_0=\pmb{a}}  h(\pmb{X}_p,\cdots, \pmb{X}_{n-l})\Big]-\mathbb{E}\mathbbm{1}_{\pmb{X}_0=\pmb{a}}  \mathbb{E}h(\pmb{X}_p,\cdots, \pmb{X}_{n-l})\Big|\\
    &R_2:=\sup_{\pmb{a}\ge 1}n\mathbb{E} \left(\mathbbm{1}_{\pmb{X}_0=\pmb{a}}  \mathbbm{1}_{\sum_{1\le j \le p-1}\pmb{X}_j\ge 1}\right)\\
    &R_3:=pn \mu_{\Delta_m}(\pi_{\Delta_m}\pi^{-1}S_r)^2+ p\mu_{\Delta_m}(\pi_{\Delta_m}\pi^{-1}S_r),
\end{align*} and a constant in $``\precsim_m"$ depends only on $m$. 
\end{lemma}

\begin{proof}Now we can estimate\begin{align*}    \sup_{h\in[0,1]}&\Big|\mathbb{E}\Big[h(\pmb{X}_0,\pmb{X}_1, \cdots, \pmb{X}_n)-h(\hat{\pmb{X}}_0, \hat{\pmb{X}}_1, \cdots, \hat{\pmb{X}}_n)\Big]\Big|\\
    &= \sup_{h\in[0,1]}\Big|\sum_{0 \le l \le n} \mathbb{E} h\left(\hat{\pmb{X}}_0, \cdots, \hat{\pmb{X}}_{l-1},\pmb{X}_l, \cdots, \pmb{X}_n\right)- \mathbb{E} h\left(\hat{\pmb{X}}_0, \cdots, \hat{\pmb{X}}_{l-1},\hat{\pmb{X}}_l, \pmb{X}_{l+1}, \cdots, \pmb{X}_n\right)\Big|\\
    &\le \sup_{h\in[0,1]}\Big|\sum_{0 \le l \le n}  \mathbb{E} h_l\Big(\pmb{X}_l,\pmb{X}_{l+1}, \cdots,\pmb{X}_n\Big)-  \mathbb{E}h_l\left(\hat{\pmb{X}}_l, \pmb{X}_{l+1}, \cdots, \pmb{X}_n\right)\Big|,
\end{align*}
where $h_l(\sbullet[.5]):=h(\hat{\pmb{X}}_0, \cdots,\hat{\pmb{X}}_{l-1}, \sbullet[.5])$. Since $\hat{\pmb{X}}_0, \cdots, \hat{\pmb{X}}_{l-1}$ are independent of other random variables, then, without loss of generality, we can assume that the function $h_l$ does not depend on $\hat{\pmb{X}}_0, \cdots, \hat{\pmb{X}}_{l-1}$. Note that $\pmb{X}_l =_d \hat{\pmb{X}}_l$ are $\{0,1\}^{N_{e,m}}$-valued random vectors. Thus
\begin{align*}
    &\Big|\mathbb{E} h_l\Big(\pmb{X}_l,\pmb{X}_{l+1}, \cdots, \pmb{X}_n\Big)-  \mathbb{E}h_l\left(\hat{\pmb{X}}_l, \pmb{X}_{l+1}, \cdots, \pmb{X}_n\right)\Big|\\
    &=\Big|\mathbb{E}\Big[ \mathbbm{1}_{\pmb{X}_l=\pmb{0}}  h_l(\pmb{0},\pmb{X}_{l+1}, \cdots,\pmb{X}_n)\Big]+\sum_{\pmb{a}\ge 1}\mathbb{E}\Big[ \mathbbm{1}_{\pmb{X}_l=\pmb{a}}  h_l(\pmb{a}, \pmb{X}_{l+1}, \cdots,\pmb{X}_n)\Big]\\
    &\quad-\mathbb{E}\mathbbm{1}_{\hat{\pmb{X}}_l=\pmb{0}} \mathbb{E}h_l(\pmb{0},\pmb{X}_{l+1},\cdots, \pmb{X}_n)- \sum_{\pmb{a}\ge1}\mathbb{E}\mathbbm{1}_{\hat{\pmb{X}}_l=\pmb{a}} \mathbb{E}h_l(\pmb{a},\pmb{X}_{l+1},\cdots, \pmb{X}_n)\Big|\\
    &= \Big|\sum_{\pmb{a}\ge 1}\mathbb{E}\Big[ \mathbbm{1}_{\pmb{X}_l=\pmb{a}} h_l(\pmb{0},\pmb{X}_{l+1}, \cdots,\pmb{X}_n)\Big]+\sum_{\pmb{a}\ge 1}\mathbb{E}\Big[ \mathbbm{1}_{\pmb{X}_l=\pmb{a}}  h_l(\pmb{a}, \pmb{X}_{l+1}, \cdots,\pmb{X}_n)\Big]\\
    &\quad-\sum_{\pmb{a}\ge 1}\mathbb{E}\mathbbm{1}_{\hat{\pmb{X}}_l=\pmb{a}}  \mathbb{E}h_l(\pmb{0},\pmb{X}_{l+1},\cdots, \pmb{X}_n)-\sum_{\pmb{a}\ge 1} \mathbb{E}\mathbbm{1}_{\hat{\pmb{X}}_l=\pmb{a}} \mathbb{E}h_l(\pmb{a},\pmb{X}_{l+1},\cdots, \pmb{X}_n)\Big|\\
    &\le 2 \sum_{\pmb{a}\ge 1}\sup_{h\in[0,1]} \Big|\mathbb{E}\Big[ \mathbbm{1}_{\pmb{X}_l=\pmb{a}}  h(\pmb{X}_{l+1}, \cdots, \pmb{X}_n)\Big]-\mathbb{E}\mathbbm{1}_{\pmb{X}_l=\pmb{a}}  \mathbb{E}h(\pmb{X}_{l+1}, \cdots, \pmb{X}_n)\Big|.
\end{align*}

Therefore,
\begin{align}
    \sup_{h\in[0,1]}&\Big|\mathbb{E}\Big[h(\pmb{X}_0,\pmb{X}_1, \cdots, \pmb{X}_n)-h(\hat{\pmb{X}}_0, \hat{\pmb{X}}_1, \cdots, \hat{\pmb{X}}_n)\Big]\Big|\nonumber \\
    &\le 2 \sum_{0 \le l \le n} \sum_{\pmb{a}\ge 1} \sup_{h \in [0,1]} \Big|\mathbb{E}\Big[ \mathbbm{1}_{\pmb{X}_l=\pmb{a}}  h(\pmb{X}_{l+1}, \cdots, \pmb{X}_n)\Big]-\mathbb{E}\mathbbm{1}_{\pmb{X}_l=\pmb{a}}  \mathbb{E}h(\pmb{X}_{l+1}, \cdots, \pmb{X}_n)\Big|.\label{sum}
\end{align}

We start with estimating the terms with $l\le n-p$ in (\ref{sum}).
\begin{align*}
    \Big|\mathbb{E}&\Big[ \mathbbm{1}_{\pmb{X}_l=\pmb{a}}  h(\pmb{X}_{l+1},\cdots, \pmb{X}_n)\Big]-\mathbb{E}\mathbbm{1}_{\pmb{X}_l=\pmb{a}}  \mathbb{E}h(\pmb{X}_{l+1}, \cdots, \pmb{X}_n)\Big|\\
    &= \Big|\mathbb{E}\Big[ \mathbbm{1}_{\pmb{X}_l=\pmb{a}}  h(\pmb{X}_{l+1},\cdots, \pmb{X}_n)\Big]-\mathbb{E}\Big[ \mathbbm{1}_{\pmb{X}_l=\pmb{a}}  h(\pmb{0},\cdots, \pmb{0},\pmb{X}_{l+p},\cdots,\pmb{X}_n)\Big]\\
    &\quad+\mathbb{E}\Big[ \mathbbm{1}_{\pmb{X}_l=\pmb{a}}  h(\pmb{0},\cdots, \pmb{0},\pmb{X}_{l+p},\cdots,\pmb{X}_n)\Big]-\mathbb{E}\mathbbm{1}_{\pmb{X}_l=\pmb{a}}  \mathbb{E}h(\pmb{X}_{l+1}, \cdots, \pmb{X}_n)\\
    &\quad+\mathbb{E}\mathbbm{1}_{\pmb{X}_l=\pmb{a}}  \mathbb{E}h(\pmb{0},\cdots, \pmb{0},\pmb{X}_{l+p},\cdots,\pmb{X}_n)-\mathbb{E}\mathbbm{1}_{\pmb{X}_l=\pmb{a}} \mathbb{E}h(\pmb{0},\cdots, \pmb{0},\pmb{X}_{l+p},\cdots,\pmb{X}_n)\Big|\\
    &= \Big|\mathbb{E}\Big\{ \mathbbm{1}_{\pmb{X}_l=\pmb{a}}  \Big[h(\pmb{X}_{l+1},\cdots, \pmb{X}_n)-h(\pmb{0},\cdots, \pmb{0},\pmb{X}_{l+p},\cdots,\pmb{X}_n)\Big]\Big\}\\
    &\quad+\mathbb{E}\mathbbm{1}_{\pmb{X}_l=\pmb{a}}  \mathbb{E}\Big[h(\pmb{0},\cdots, \pmb{0},\pmb{X}_{l+p},\cdots,\pmb{X}_n)-h(\pmb{X}_{l+1}, \cdots, \pmb{X}_n)\Big]\\
    &\quad+\mathbb{E}\Big[ \mathbbm{1}_{\pmb{X}_l=\pmb{a}}  h(\pmb{0},\cdots, \pmb{0},\pmb{X}_{l+p},\cdots,\pmb{X}_n)\Big]-\mathbb{E}\mathbbm{1}_{\pmb{X}_l=\pmb{a}}  \mathbb{E}h(\pmb{0},\cdots, \pmb{0},\pmb{X}_{l+p},\cdots,\pmb{X}_n)\Big|.
\end{align*}

Observe that 
\[ |h(\pmb{X}_{l+1}, \cdots, \pmb{X}_n)- h(\pmb{0}, \cdots, \pmb{0},\pmb{X}_{l+p}, \cdots, \pmb{X}_n)| \le 2  \mathbbm{1}_{\sum_{l+1\le j \le l+p-1}\pmb{X}_j\ge 1}.\]

Now, in view of stationarity of $(\pmb{X}_i)_{i\ge 0}$, we can continue estimates as 
\begin{align*}
    &\le \Big| \mathbb{E}\Big[\mathbbm{1}_{\pmb{X}_l=\pmb{a}}  h(\pmb{0}, \cdots, \pmb{0},\pmb{X}_{l+p}, \cdots, \pmb{X}_n)\Big]-\mathbb{E}\mathbbm{1}_{\pmb{X}_l=\pmb{a}}  \mathbb{E}h\left(\pmb{0}, \cdots, \pmb{0},\pmb{X}_{l+p}, \cdots, \pmb{X}_n\right)\Big|\\
    &\quad+2\mathbb{E} \Big(\mathbbm{1}_{\pmb{X}_l=\pmb{a}}  \mathbbm{1}_{\sum_{l+1\le j \le l+p-1}\pmb{X}_j\ge 1}\Big)+2\mathbb{E}\mathbbm{1}_{\pmb{X}_l=\pmb{a}}  \mathbb{E}\mathbbm{1}_{\sum_{l+1\le j \le l+p-1}\pmb{X}_j\ge 1}\\
    &\le \Big| \mathbb{E}\Big[\mathbbm{1}_{\pmb{X}_l=\pmb{a}}   h(\pmb{0}, \cdots, \pmb{0},\pmb{X}_{l+p}, \cdots, \pmb{X}_n)\Big]-\mathbb{E}\mathbbm{1}_{\pmb{X}_l=\pmb{a}}  \mathbb{E}h(\pmb{0}, \cdots, \pmb{0},\pmb{X}_{l+p}, \cdots, \pmb{X}_n)\Big|\\
    &\quad+2\mathbb{E} \Big(\mathbbm{1}_{\pmb{X}_0=\pmb{a}}  \mathbbm{1}_{\sum_{1\le j \le p-1}\pmb{X}_j\ge 1}\Big)+2\mathbb{E}\mathbbm{1}_{\pmb{X}_0=\pmb{a}}  \mathbb{E}\mathbbm{1}_{\sum_{1\le j \le p-1}\pmb{X}_j\ge 1}.
\end{align*}

Observe that \begin{gather*}
    \{\sum_{1\le j \le p-1}\pmb{X}_j\ge 1\}=\bigcup_{N_{e,m}\le j \le N_{e,m}p-1} (F^{R_m})^{-j} (\pi_{\Delta_m}\pi^{-1}S_r),\quad \pmb{a}\ge 1,\\
    \{\pmb{X}_0=\pmb{a}\}\subseteq \bigcup_{0\le j \le N_{e,m}-1} (F^{R_m})^{-j} (\pi_{\Delta_m}\pi^{-1}S_r).
\end{gather*}  

Hence, we can continue the sequence of inequalities above as 
\begin{align*}
    &\precsim_m \Big| \mathbb{E}\Big[\mathbbm{1}_{\pmb{X}_l=\pmb{a}}   h(\pmb{0}, \cdots, \pmb{0},\pmb{X}_{l+p}, \cdots, \pmb{X}_n)\Big]-\mathbb{E}\mathbbm{1}_{\pmb{X}_l=\pmb{a}}  \mathbb{E}h(\pmb{0}, \cdots, \pmb{0},\pmb{X}_{l+p}, \cdots, \pmb{X}_n)\Big|\\
    &\quad+\mathbb{E} \left(\mathbbm{1}_{\pmb{X}_0=\pmb{a}}  \mathbbm{1}_{\sum_{1\le j \le p-1}\pmb{X}_j\ge 1}\right)+ p \mu_{\Delta_m}(\pi_{\Delta_m}\pi^{-1}S_r)^2.
\end{align*}

Therefore for the terms with $l\le n-p$ in (\ref{sum}) we have 
\begin{align*}
    \Big|\mathbb{E}&\Big[ \mathbbm{1}_{\pmb{X}_l=\pmb{a}}  h(\pmb{X}_{l+1}, \cdots, \pmb{X}_n)\Big]-\mathbb{E}\mathbbm{1}_{\pmb{X}_l=\pmb{a}}  \mathbb{E}h(\pmb{X}_{l+1}, \cdots, \pmb{X}_n)\Big|\\
    &\precsim_m \Big|\mathbb{E}\Big[\mathbbm{1}_{\pmb{X}_l=\pmb{a}}   h(\pmb{0}, \cdots, \pmb{0},\pmb{X}_{l+p}, \cdots, \pmb{X}_n)\Big]-\mathbb{E}\mathbbm{1}_{\pmb{X}_l=\pmb{a}} \mathbb{E}h(\pmb{0}, \cdots, \pmb{0},\pmb{X}_{l+p}, \cdots, \pmb{X}_n)\Big|\\
    &\quad+\mathbb{E} \left(\mathbbm{1}_{\pmb{X}_0=\pmb{a}} \mathbbm{1}_{\sum_{1\le j \le p-1}\pmb{X}_j\ge 1}\right)+p  \mu_{\Delta_m}(\pi_{\Delta_m}\pi^{-1}S_r)^2.
\end{align*}

Consider now the terms with $l> n-p$ in (\ref{sum}). Since $\pmb{a}\ge 1$ and \begin{gather*}
    \{\pmb{X}_l=\pmb{a}\}\subseteq \bigcup_{lN_{e,m}\le j \le (l+1)N_{e,m}-1} (F^{R_m})^{-j} (\pi_{\Delta_m}\pi^{-1}S_r), \quad ||h||_{\infty} \le 1,
\end{gather*} then
\begin{align*}
    \Big|\mathbb{E}\Big[ \mathbbm{1}_{\pmb{X}_l=\pmb{a}}  h(\pmb{X}_{l+1},\cdots, \pmb{X}_n)\Big]-\mathbb{E}\mathbbm{1}_{\pmb{X}_l=\pmb{a}}  \mathbb{E}h(\pmb{X}_{l+1},\cdots, \pmb{X}_n) \Big|\precsim_m \mu_{\Delta_m}(\pi_{\Delta_m}\pi^{-1}S_r).
\end{align*}

 Therefore 
 \begin{align*}
     (\ref{sum})&=2 \sum_{0 \le l \le n}\sup_{\pmb{a}\ge 1} \sup_{h\in[0,1]}\Big|\mathbb{E}\Big[ \mathbbm{1}_{\pmb{X}_l=\pmb{a}}  h(\pmb{X}_{l+1},\cdots, \pmb{X}_n)\Big]-\mathbb{E}\mathbbm{1}_{\pmb{X}_l=\pmb{a}}  \mathbb{E}h(\pmb{X}_{l+1},\cdots,\pmb{X}_n)\Big|\\
     &\precsim_m \sum_{0 \le l \le n-p} \sup_{\pmb{a}\ge 1}\sup_{h\in[0,1]} \Big|\mathbb{E}\Big[\mathbbm{1}_{\pmb{X}_l=\pmb{a}}  h(\pmb{X}_{l+p},\cdots, \pmb{X}_n)\Big]-\mathbb{E}\mathbbm{1}_{\pmb{X}_l=\pmb{a}}  \mathbb{E}h(\pmb{X}_{l+p},\cdots, \pmb{X}_n)\Big|\\
     &\quad+\sup_{\pmb{a}\ge 1}n\mathbb{E} \Big(\mathbbm{1}_{\pmb{X}_0=\pmb{a}}  \mathbbm{1}_{\sum_{1\le j \le p-1}\pmb{X}_j\ge 1}\Big)+p n \mu_{\Delta_m}(\pi_{\Delta_m}\pi^{-1}S_r)^2+p\mu_{\Delta_m}(\pi_{\Delta_m}\pi^{-1}S_r).
 \end{align*}

By making use of stationarity of $(\pmb{X}_i)_{i \ge 0}$, the last expression above can be estimated as
\begin{align*}
    &\precsim_m \sum_{0 \le l \le n-p}\sup_{\pmb{a}\ge 1}\sup_{h\in[0,1]} \Big| \mathbb{E}\Big[\mathbbm{1}_{\pmb{X}_0=\pmb{a}}  h(\pmb{X}_p,\cdots, \pmb{X}_{n-l})\Big]-\mathbb{E}\mathbbm{1}_{\pmb{X}_0=\pmb{a}}  \mathbb{E}h(\pmb{X}_p,\cdots, \pmb{X}_{n-l})\Big|\\
    &\quad+\sup_{\pmb{a}\ge 1}n\mathbb{E} \left(\mathbbm{1}_{\pmb{X}_0=\pmb{a}}  \mathbbm{1}_{\sum_{1\le j \le p-1}\pmb{X}_j\ge 1}\right)+pn \mu_{\Delta_m}(\pi_{\Delta_m}\pi^{-1}S_r)^2+ p\mu_{\Delta_m}(\pi_{\Delta_m}\pi^{-1}S_r), 
\end{align*}which concludes a proof.
\end{proof}

Denote by $\{\hat{X}_i\}_{i\ge 0}$ i.i.d. random variables, which do not depend on $(X_i)_{i \ge 0}$ and $(\hat{\pmb{X}}_i)_{i\ge 0}$, and which are defined on a probability space $(\hat{\Omega}, \hat{\mathbb{P}})$ such that for each $i \ge 0$,\begin{equation*}
    X_i=_d\hat{X}_i.
\end{equation*} 

Define now random vectors \begin{align*}
    \pmb{Y}_i:=(\hat{X}_{iN_{e,m}}, \hat{X}_{iN_{e,m}+1}, \cdots, \hat{X}_{(i+1)N_{e,m}-1}).
\end{align*} 

As the next step we will approximate $(\hat{\pmb{X}}_i)_{i \ge 0}$ by $(\pmb{Y}_i)_{i\ge 0}$.

\begin{lemma}\label{poissonappro2}
For any $n\ge 1$,\begin{align*}
    \sup_{h\in[0,1]}\Big|&\mathbb{E}\Big[h(\pmb{Y}_0,\pmb{Y}_1, \cdots, \pmb{Y}_n)-h(\hat{\pmb{X}}_0, \hat{\pmb{X}}_1, \cdots, \hat{\pmb{X}}_n)\Big]\Big|\\
    &\precsim_m n\mathbb{E}\big(\mathbbm{1}_{X_0=1}\mathbbm{1}_{\sum_{1\le j\le N_{e,m}-1}X_j\ge1}\big) +n\mu_{\Delta_m}(\pi_{\Delta_m}\pi^{-1}S_r)^2,
    \end{align*}where a constant in $``\precsim_m"$ depends only on $m$. 
\end{lemma}
\begin{proof} Assume that $\hat{\pmb{X}}_0:=(Z_0,Z_1,\cdots,Z_{N_{e,m}-1})$. Note that all $Z_i$ are not independent, and for all $0\le i\le N_{e,m}-1$, we have \begin{gather*}
    Z_i=_dX_i=_d\hat{X}_i. 
\end{gather*}

We can start now the next estimate.\begin{align*}
    \sup_{h\in[0,1]}&\Big|\mathbb{E}\Big[h(\pmb{Y}_0,\pmb{Y}_1, \cdots, \pmb{Y}_n)-h(\hat{\pmb{X}}_0, \hat{\pmb{X}}_1, \cdots, \hat{\pmb{X}}_n)\Big]\Big|\\
    &= \sup_{h\in[0,1]}\Big|\sum_{0 \le l \le n} \mathbb{E} h\left(\hat{\pmb{X}}_0, \cdots, \hat{\pmb{X}}_{l-1},\pmb{Y}_l, \cdots, \pmb{Y}_n\right)- \mathbb{E} h\left(\hat{\pmb{X}}_0, \cdots, \hat{\pmb{X}}_{l-1},\hat{\pmb{X}}_l, \pmb{Y}_{l+1}, \cdots, \pmb{Y}_n\right)\Big|\\
    &\le \sup_{h\in[0,1]}\Big|\sum_{0 \le l \le n}  \mathbb{E} h'_l(\pmb{Y}_l)-  \mathbb{E}h'_l(\hat{\pmb{X}}_l)\Big|\le \sum_{0 \le l \le n} \sup_{h\in[0,1]}\Big| \mathbb{E} h'_l(\pmb{Y}_l)-  \mathbb{E}h'_l(\hat{\pmb{X}}_l)\Big|,
\end{align*}where $h'_l(\sbullet[.5]):=h(\hat{\pmb{X}}_1, \cdots,\hat{\pmb{X}}_{l-1}, \sbullet[.5], \pmb{Y}_{l+1}, \cdots, \pmb{Y}_{n})$. Since $\hat{\pmb{X}}_1, \cdots, \hat{\pmb{X}}_{l-1}, \pmb{Y}_{l+1}, \cdots, \pmb{Y}_{n}$ do not depend on $\pmb{Y}_l$ and $\hat{\pmb{X}}_l$. Then, without a loss of generality, $h'_l$ can be viewed as a function which does not depend on $\hat{\pmb{X}}_1, \cdots, \hat{\pmb{X}}_{l-1}, \pmb{Y}_{l+1}, \cdots, \pmb{Y}_{n}$. By stationarity of $(\pmb{Y}_i)_{i\ge 0}$ and $(\hat{\pmb{X}}_i)_{i\ge 0}$, we have\begin{align*}
    \sup_{h\in[0,1]}\Big| \mathbb{E} h'_l(\pmb{Y}_l)&-  \mathbb{E}h'_l(\hat{\pmb{X}}_l)\Big|\le \sup_{h\in[0,1]}\Big| \mathbb{E} h(\pmb{Y}_l)-  \mathbb{E}h(\hat{\pmb{X}}_l)\Big|= \sup_{h\in[0,1]}\Big| \mathbb{E} h(\pmb{Y}_0)-  \mathbb{E}h(\hat{\pmb{X}}_0)\Big|\\
    &= \sup_{h\in[0,1]}\Big|\sum_{0 \le l \le N_{e,m}-1} \mathbb{E} h\left(\hat{X}_0, \cdots, \hat{X}_{l-1},Z_l, \cdots, Z_{N_{e,m}-1}\right)\\
    &\quad - \mathbb{E} h\left(\hat{X}_0, \cdots, \hat{X}_{l-1}, \hat{X}_l, Z_{l+1}, \cdots, Z_{N_{e,m}-1}\right)\Big|\\
    &\le \sup_{h\in[0,1]}\Big|\sum_{0 \le l \le N_{e,m}-1}  \mathbb{E} h_l(Z_l, Z_{l+1}, \cdots, Z_{N_{e,m}-1})- \mathbb{E}h_l(\hat{X}_l,Z_{l+1}, \cdots, Z_{N_{e,m}-1})\Big|\\
    &\le \sum_{0 \le l \le N_{e,m}-1}  \sup_{h\in[0,1]}\Big|\mathbb{E} h_l(Z_l, Z_{l+1}, \cdots, Z_{N_{e,m}-1})- \mathbb{E}h_l(\hat{X}_l,Z_{l+1}, \cdots, Z_{N_{e,m}-1})\Big|,
\end{align*}where $h_l(\sbullet[.5]):=h(\hat{X}_1, \cdots,\hat{X}_{l-1}, \sbullet[.5])$. As before, $h_l$ can be regarded as a function which does not depend on $\hat{X}_1, \cdots, \hat{X}_{l-1}$. Note that $X_l =_d \hat{X}_l=_dZ_l$ are $\{0,1\}$-valued random variables. Thus
\begin{align*}
    &\Big|\mathbb{E} h_l\Big(Z_l,Z_{l+1}, \cdots,Z_{N_{e,m}-1}\Big)-  \mathbb{E}h_l\left(\hat{X}_l, Z_{l+1}, \cdots, Z_{N_{e,m}-1}\right)\Big|\\
    &=\Big|\mathbb{E}\Big[ \mathbbm{1}_{Z_l=0}  h_l(0,Z_{l+1}, \cdots,Z_{N_{e,m}-1})\Big]+\mathbb{E}\Big[ \mathbbm{1}_{Z_l=1}  h_l(1, Z_{l+1}, \cdots,Z_{N_{e,m}-1})\Big]\\
    &\quad-\mathbb{E}\mathbbm{1}_{\hat{X}_l=0} \mathbb{E}h_l(0,Z_{l+1},\cdots, Z_{N_{e,m}-1})- \mathbb{E}\mathbbm{1}_{\hat{X}_l=1} \mathbb{E}h_l(1,Z_{l+1},\cdots, Z_{N_{e,m}-1})\Big|\\
    &=\Big|\mathbb{E}\Big[ \mathbbm{1}_{Z_l=1} h_l(0,Z_{l+1}, \cdots,Z_{N_{e,m}-1})\Big]+\mathbb{E}\Big[ \mathbbm{1}_{Z_l=1}  h_l(1, Z_{l+1}, \cdots,Z_{N_{e,m}-1})\Big]\\
    &\quad-\mathbb{E}\mathbbm{1}_{\hat{X}_l=1}  \mathbb{E}h_l(0,Z_{l+1},\cdots, Z_{N_{e,m}-1})- \mathbb{E}\mathbbm{1}_{\hat{X}_l=1} \mathbb{E}h_l(1,Z_{l+1},\cdots, Z_{N_{e,m}-1})\Big|\\
    &\le 2 \sup_{h\in[0,1]} \Big|\mathbb{E}\Big[ \mathbbm{1}_{Z_l=1}  h(Z_{l+1}, \cdots, Z_{N_{e,m}-1})\Big]-\mathbb{E}\mathbbm{1}_{Z_l=1}  \mathbb{E}h(Z_{l+1}, \cdots, Z_{N_{e,m}-1})\Big|\\
    &=2 \sup_{h\in[0,1]} \Big|\mathbb{E}\Big[ \mathbbm{1}_{Z_l=1}  h(Z_{l+1}, \cdots, Z_{N_{e,m}-1})-\mathbbm{1}_{Z_l=1}h(0, \cdots, 0)\Big]\\
    & \quad -\mathbb{E}\mathbbm{1}_{Z_l=1}  \mathbb{E}\Big[h(Z_{l+1}, \cdots, Z_{N_{e,m}-1})-h(0,\cdots, 0)\Big]\Big|.
\end{align*}

Using that $|h(Z_{l+1}, \cdots, Z_{N_{e,m}-1})-h(0,\cdots, 0)| \le2\mathbbm{1}_{\sum_{l+1\le j\le N_{e,m}-1}Z_j\ge1} $, stationarity of $(X_i)_{i=0}^{ N_{e,m}-1}$ and of $(Z_i)_{ i=0}^{ N_{e,m}-1 }$, and the relation  $(X_0,X_1,\cdots,X_{N_{e,m}-1})=_d(Z_0,Z_1,\cdots,Z_{N_{e,m}-1})$, we can continue the estimate above as\begin{align*}
    &\precsim \mathbb{E}\big(\mathbbm{1}_{Z_l=1}\mathbbm{1}_{\sum_{l+1\le j\le N_{e,m}-1}Z_j\ge1}\big)  +\mathbb{E}\mathbbm{1}_{Z_l=1}  \mathbb{E}\mathbbm{1}_{\sum_{l+1\le j\le N_{e,m}-1}Z_j\ge1} \\
    &\precsim \mathbb{E}\big(\mathbbm{1}_{Z_0=1}\mathbbm{1}_{\sum_{1\le j\le N_{e,m}-1}Z_j\ge1} \big) +\mathbb{E}\mathbbm{1}_{Z_0=1}  \mathbb{E}\mathbbm{1}_{\sum_{1\le j\le N_{e,m}-1}Z_j\ge1} \\
    &\precsim\mathbb{E}\big(\mathbbm{1}_{X_0=1}\mathbbm{1}_{\sum_{1\le j\le N_{e,m}-1}X_j\ge1}\big)  +\mathbb{E}\mathbbm{1}_{X_0=1}  \mathbb{E}\mathbbm{1}_{\sum_{1\le j\le N_{e,m}-1}X_j\ge1}\\
    &\precsim\mathbb{E}\big(\mathbbm{1}_{X_0=1}\mathbbm{1}_{\sum_{1\le j\le N_{e,m}-1}X_j\ge1} \big) +\mu_{\Delta_m}(\pi_{\Delta_m}\pi^{-1}S_r) \mathbb{E}\mathbbm{1}_{\sum_{1\le j\le N_{e,m}-1}X_j\ge1}\\
    &\precsim\mathbb{E}\big(\mathbbm{1}_{X_0=1}\mathbbm{1}_{\sum_{1\le j\le N_{e,m}-1}X_j\ge1} \big) +N_{e,m}\mu_{\Delta_m}(\pi_{\Delta_m}\pi^{-1}S_r)^2,
\end{align*}where the last inequality is due to $\{\sum_{1\le j\le N_{e,m}-1}X_j\ge1\}\subseteq \bigcup_{0\le j\le N_{e,m}-1}(F^{R_m})^{-j}(\pi_{\Delta_m}\pi^{-1}S_r) $.

By combining all estimates above we get\begin{align*}
    \sup_{h\in[0,1]}&\Big|\mathbb{E}\Big[h(\pmb{Y}_0,\pmb{Y}_1, \cdots, \pmb{Y}_n)-h(\hat{\pmb{X}}_0, \hat{\pmb{X}}_1, \cdots, \hat{\pmb{X}}_n)\Big]\Big|\le \sum_{0 \le l \le n} \sup_{h\in[0,1]}\Big| \mathbb{E} h'_l(\pmb{Y}_l)-  \mathbb{E}h'_l(\hat{\pmb{X}}_l)\Big|\\
    &\le \sum_{0\le l\le n}\sum_{0 \le l' \le N_{e,m}-1}  \sup_{h\in[0,1]}\Big|\mathbb{E} h_{l'}(Z_{l'}, \cdots, Z_{N_{e,m}-1})- \mathbb{E}h_{l'}(\hat{X}_{l'},Z_{l'+1}, \cdots, Z_{N_{e,m}-1})\Big|\\
    &\le  \sum_{0\le l\le n}\sum_{0 \le l' \le N_{e,m}-1}  \big[\mathbb{E}\big(\mathbbm{1}_{X_0=1}\mathbbm{1}_{\sum_{1\le j\le N_{e,m}-1}X_j\ge1}  \big)+N_{e,m}\mu_{\Delta_m}(\pi_{\Delta_m}\pi^{-1}S_r)^2\big]\\
    &\precsim_m n\mathbb{E}\big(\mathbbm{1}_{X_0=1}\mathbbm{1}_{\sum_{1\le j\le N_{e,m}-1}X_j\ge1}\big)  +n\mu_{\Delta_m}(\pi_{\Delta_m}\pi^{-1}S_r)^2,
\end{align*} which concludes a proof.
\end{proof}

By making use of Lemmas \ref{poissonappro1} and \ref{poissonappro2}, we can simplify now (\ref{inducingforanym}).
\begin{lemma}\label{poissonappro}
 For any $\epsilon\in (0,1)$, and for any disjoint bounded intervals $J_1, J_2, \cdots, J_k\subseteq [0,\infty)$, let \begin{gather*}
     n:=\max\{i: i\cdot \mu_{\Delta_m}(\pi_{\Delta_m}\pi^{-1}S_r)\in J_j, j=1,2,\cdots, k\}, \quad p:=n^{1-\epsilon}.
 \end{gather*} Then
\begin{align*}\Big|\mu_{\Delta_m}\Big(\mathcal{N}_i^{r}(\bigcup_{j \le k}J_j)=0\Big)-\mathbb{P}\Big(\mathcal{P}(\bigcup_{j \le k}J_j)=0\Big)\Big|\precsim_m T_1+T_2+T_3, 
    \end{align*}
    where
    \begin{align}
    &T_1:=\sum_{0 \le l \le n-p}\sup_{\pmb{a}\ge 1}\sup_{h\in[0,1]} \Big| \mathbb{E}\Big[\mathbbm{1}_{\pmb{X}_0=\pmb{a}}  h(\pmb{X}_p,\cdots, \pmb{X}_{n-l})\Big]-\mathbb{E}\mathbbm{1}_{\pmb{X}_0=\pmb{a}}  \mathbb{E}h(\pmb{X}_p,\cdots, \pmb{X}_{n-l})\Big|\label{correlationinpoissonlaw}\\
    &T_2:=\mu_{\Delta_m}(\pi_{\Delta_m}\pi^{-1}S_r)^{-1}\mathbb{E} \left(\mathbbm{1}_{X_0=1}  \mathbbm{1}_{\sum_{1\le j \le pN_{e,m}-1}X_j\ge 1}\right)\label{shortreturn}\\
    &T_3:=\mu_{\Delta_m}(\pi_{\Delta_m}\pi^{-1}S_r)^{\epsilon}\nonumber,
\end{align}where a constant in $``\precsim_m"$ depends only on $m$. Since $T_3\precsim_m \mu_{\mathcal{M}}(S_r)^{\epsilon}\to 0$ when $r\to 0$, then, in order to prove (\ref{inducingforanym}), we just need to show that (\ref{correlationinpoissonlaw}) and (\ref{shortreturn}) converge to zero as $r\to 0$.

We will say in what follows that (\ref{shortreturn}) is a short return.  
\end{lemma}
\begin{proof}Let $J_i':=\{j: j\cdot \mu_{\Delta_m}(\pi_{\Delta_m}\pi^{-1}S_r)\in J_i\}$, $ X_{J_i'}:=\sum_{j \in J_i'} \mathbbm{1}_{\pi_{\Delta_m}\pi^{-1}S_r} \circ (F^{R_m})^j$, and $\hat{X}_{J_i'}:=\sum_{j \in J_i'}\hat{X}_{j}$, where $\{\hat{X}_i\}_{i\ge 0}$ are i.i.d. random variables such that $\hat{X}_i=_dX_i=\mathbbm{1}_{\pi_{\Delta_m}\pi^{-1}S_r} \circ (F^{R_m})^i$. Then we have
\begin{align}
\Big|\mu_{\Delta_m}\Big(\mathcal{N}_i^{r}(\bigcup_{j \le k}J_j)=0\Big)-\mathbb{P}\Big(&\mathcal{P}(\bigcup_{j \le k}J_j)=0\Big)\Big|\le \sup_{h \in [0,1]}\Big|\mathbb{E}\Big[h(X_{J_1'}, \cdots, X_{J_k'})-h\Big(\mathcal{P}(J_1), \cdots, \mathcal{P}(J_k)\Big)\Big]\Big|\nonumber\\
&\le \sup_{h\in[0,1]}\Big|\mathbb{E}\Big[h(X_{J_1'}, \cdots, X_{J_k'})-h(\hat{X}_{J_1'}, \cdots, \hat{X}_{J_k'})\Big]\Big|\nonumber\\
&\quad +\sup_{h\in[0,1]}\Big|\mathbb{E}\Big[h\Big(\mathcal{P}(J_1), \cdots, \mathcal{P}(J_k)\Big)-h(\hat{X}_{J_1'}, \cdots, \hat{X}_{J_k'})\Big]\Big|\nonumber\\
&\le \sup_{h\in[0,1]}\Big|\mathbb{E}\Big[h(X_0,X_1, \cdots, X_n)-h(\hat{X}_0,\hat{X}_1, \cdots, \hat{X}_n)\Big]\Big|\nonumber\\
&\quad +\sup_{h\in[0,1]}\Big|\mathbb{E}\Big[h\Big(\mathcal{P}(J_1), \cdots, \mathcal{P}(J_k)\Big)-h(\hat{X}_{J_1'}, \cdots, \hat{X}_{J_k'})\Big]\Big|. \label{20}
\end{align}
    
By applying Theorem 2  from \cite{chenmethod} to $(\hat{X}_{i})_{i \ge 0}$ we get  \begin{align*}
\sup_{h\in[0,1]} \Big|\mathbb{E}\Big[h\Big(P(J_1), \cdots, P(J_k)\Big)-h\left(\hat{X}_{J'_1}, \cdots, \hat{X}_{J'_k}\right)\Big]\Big|\precsim n \mu_{\Delta_m}(\pi_{\Delta_m}\pi^{-1}S_r)^2.
\end{align*} 

Since $(X_0, \cdots, X_n)$, $(\hat{X}_0,\hat{X}_1, \cdots, \hat{X}_n)$  are some entries of $(\pmb{X}_0,\cdots, \pmb{X}_n)$, $(\pmb{Y}_0,\pmb{Y}_1, \cdots, \pmb{Y}_n)$, respectively, then\[\sup_{h\in[0,1]}\Big|\mathbb{E}\Big[h(X_0, \cdots, X_n)-h(\hat{X}_0, \cdots, \hat{X}_n)\Big]\Big|\le \sup_{h\in[0,1]}\Big|\mathbb{E}\Big[h(\pmb{X}_0, \cdots, \pmb{X}_n)-h(\pmb{Y}_0, \cdots, \pmb{Y}_n)\Big]\Big|.\]

Now we can continue the estimate (\ref{20}) as\begin{align*}
    &\precsim_m \sup_{h\in[0,1]}\Big|\mathbb{E}\Big[h(X_0,X_1, \cdots, X_n)-h(\hat{X}_0,\hat{X}_1, \cdots, \hat{X}_n)\Big]\Big|+n \mu_{\Delta_m}(\pi_{\Delta_m}\pi^{-1}S_r)^2\\
    &\precsim_m  \sup_{h\in[0,1]}\Big|\mathbb{E}\Big[h(\pmb{X}_0,\pmb{X}_1, \cdots, \pmb{X}_n)-h(\hat{\pmb{X}}_0, \hat{\pmb{X}}_1, \cdots, \hat{\pmb{X}}_n)\Big]\Big|\\
    &\quad +\sup_{h\in[0,1]}\Big|\mathbb{E}\Big[h(\pmb{Y}_0,\pmb{Y}_1, \cdots, \pmb{Y}_n)-h(\hat{\pmb{X}}_0, \hat{\pmb{X}}_1, \cdots, \hat{\pmb{X}}_n)\Big]\Big|+n \mu_{\Delta_m}(\pi_{\Delta_m}\pi^{-1}S_r)^2\\
    &\precsim_m R_1+R_2+R_3+n\mathbb{E}\mathbbm{1}_{X_0=1}\mathbbm{1}_{\sum_{1\le j\le N_{e,m}-1}X_j\ge1}  +n\mu_{\Delta_m}(\pi_{\Delta_m}\pi^{-1}S_r)^2,
\end{align*}where the last ``$\precsim_m$" is due to Lemmas \ref{poissonappro1} and \ref{poissonappro2}.

Using $n\precsim  \mu_{\Delta_m}(\pi_{\Delta_m}\pi^{-1}S_r)^{-1}$ (where a constant in $``\precsim"$ depends on $\max\{n': n'\in J_i, i=1,2,\cdots, k\}<\infty$), $p=n^{1-\epsilon}$ and stationarity of $(X_i)_{i\ge 0}$, we obtain \begin{gather*}
    R_2=\sup_{\pmb{a}\ge 1}n\mathbb{E} \left(\mathbbm{1}_{\pmb{X}_0=\pmb{a}}  \mathbbm{1}_{\sum_{1\le j \le p-1}\pmb{X}_j\ge 1}\right)\precsim \mu_{\Delta_m}(\pi_{\Delta_m}\pi^{-1}S_r)^{-1}\mathbb{E} \left(\mathbbm{1}_{X_0=1}  \mathbbm{1}_{\sum_{1\le j \le pN_{e,m}-1}X_j\ge 1}\right),\\
    n\mathbb{E}\mathbbm{1}_{X_0=1}\mathbbm{1}_{\sum_{1\le j\le N_{e,m}-1}X_j\ge1}\precsim \mu_{\Delta_m}(\pi_{\Delta_m}\pi^{-1}S_r)^{-1}\mathbb{E} \left(\mathbbm{1}_{X_0=1}  \mathbbm{1}_{\sum_{1\le j \le pN_{e,m}-1}X_j\ge 1}\right).
\end{gather*} 

Then we can continue the estimate above as\begin{align*}
    &\precsim_m \sum_{0 \le l \le n-p}\sup_{\pmb{a}\ge 1}\sup_{h\in[0,1]} \Big| \mathbb{E}\Big[\mathbbm{1}_{\pmb{X}_0=\pmb{a}}  h(\pmb{X}_p,\cdots, \pmb{X}_{n-l})\Big]-\mathbb{E}\mathbbm{1}_{\pmb{X}_0=\pmb{a}}  \mathbb{E}h(\pmb{X}_p,\cdots, \pmb{X}_{n-l})\Big|\\
    &\quad + \mu_{\Delta_m}(\pi_{\Delta_m}\pi^{-1}S_r)^{-1}\mathbb{E} \left(\mathbbm{1}_{X_0=1}  \mathbbm{1}_{\sum_{1\le j \le pN_{e,m}-1}X_j\ge 1}\right)+\mu_{\Delta_m}(\pi_{\Delta_m}\pi^{-1}S_r)^{\epsilon},
\end{align*} which concludes a proof.
\end{proof}
\section{Short returns}\label{sectionshortreturn}
From Lemma \ref{poissonappro}, we just need to prove (\ref{correlationinpoissonlaw}) and (\ref{shortreturn}). In this section we will estimate short returns (\ref{shortreturn}). The papers \cite{penebacktoball,peneijm} provided effective methods to estimate (\ref{shortreturn}) for Sinai billiards with bounded horizons and for diamond billiards. Some specific properties (e.g. bounded free paths and complexity of singularities for these two billiards) were used there. In contrast, we are using here only the hyperbolic product structure $\Lambda$ in hyperbolic Young towers.

First of all we will show that the short return (\ref{shortreturn}) on $\Delta_m$ can be reduced to the short return on $X$.
\begin{lemma}[Reduce (\ref{shortreturn})]\begin{gather*}
   \mathbb{E} \left(\mathbbm{1}_{X_0=1}  \mathbbm{1}_{\sum_{1\le j \le pN_{e,m}-1}X_j\ge 1}\right)\precsim_m \int \mathbbm{1}_{\pi_{X}\pi^{-1}S_r}  \mathbbm{1}_{\bigcup_{1\le j \le pN_{e,m}}(f^{R})^{-j}(\pi_{X}\pi^{-1}S_r)}d\mu_{X},\\
   \mu_{\Delta_m}(\pi_{\Delta_m}\pi^{-1}S_r)\approx_m \mu_{X}(\pi_{X}\pi^{-1}S_r).
\end{gather*}
\end{lemma}
\begin{proof} Since $S_r$ is a section, then for any $x\in \Delta_m$,
\begin{gather*}
    \mathbbm{1}_{\pi_{\Delta_m}\pi^{-1}S_r}(x) \mathbbm{1}_{\bigcup_{j=1}^{pN_{e,m}-1}(F^{R_m})^{-j}(\pi_{\Delta_m}\pi^{-1}S_r)}(x) \le \mathbbm{1}_{\pi_{X}\pi^{-1}S_r}(\pi_Xx)  \mathbbm{1}_{\bigcup_{j=1}^{pN_{e,m}}(f^{R})^{-j}(\pi_{X}\pi^{-1}S_r)}(\pi_Xx),\\
    (\pi_X)_{*}(\mu_{\Delta_m}|_{\pi_{\Delta_m}\pi^{-1}S_r})=(\int \min\{R,m+1\} d\mu_{X})^{-1}\mu_X|_{\pi_{X}\pi^{-1}S_r}.
\end{gather*}

Therefore, \begin{align*}
     \mathbb{E} \left(\mathbbm{1}_{X_0=1}  \mathbbm{1}_{\sum_{1\le j \le pN_{e,m}-1}X_j\ge 1}\right)&\le \int_{\pi_{\Delta_m}\pi^{-1}S_r} \mathbbm{1}_{\pi_{X}\pi^{-1}S_r}\circ \pi_X  \mathbbm{1}_{\bigcup_{j=1}^{pN_{e,m}}(f^{R})^{-j}(\pi_{X}\pi^{-1}S_r)}\circ \pi_X d\mu_{\Delta_m}\\
     &\precsim_m \int \mathbbm{1}_{\pi_{X}\pi^{-1}S_r} \mathbbm{1}_{\bigcup_{j=1}^{pN_{e,m}}(f^{R})^{-j}(\pi_{X}\pi^{-1}S_r)} d\mu_{X}.
\end{align*}

The relation $\mu_{\Delta_m}(\pi_{\Delta_m}\pi^{-1}S_r)\approx_m \mu_{X}(\pi_{X}\pi^{-1}S_r)$ holds since $S_r$ is a section, which concludes a proof of the lemma.
\end{proof}

Therefore, in order to prove that (\ref{shortreturn}) converges to zero, the relation \begin{gather*}
    \lim_{r\to 0}\frac{1}{\mu_{X}(\pi_{X}\pi^{-1}S_r)}\int \mathbbm{1}_{\pi_{X}\pi^{-1}S_r} \mathbbm{1}_{\bigcup_{1\le j\le pN_{e,m}-1}(f^{R})^{-j}(\pi_{X}\pi^{-1}S_r)} d\mu_{X}=0
\end{gather*} will be proved in Lemma \ref{thelastlemmaofshortreturn}, accompanied by  several technical lemmas in the following subsections. We will use for that decay of correlations in the dynamical system $(X, (f^R)^{N_e}, \mu_X)$ and mixing hyperbolic Young towers $(\Delta_e', F_e^{N_e}, \mu_{\Delta_e'})$, where $N_e$ is the one from Lemma \ref{resultsofyoungpaper}. 

Throughout this section, in order to simplify notations, we let $g=f^R$, and still use $\pi_e$ to denote the semi-conjugacy from $\Delta_e'$ to $X$, i.e., $\pi_e: \Delta_e' \to X$ satisfies \begin{align*}
     g^{N_e}  \circ \pi_e = \pi_e \circ F_e^{N_e}.
\end{align*}

\[ \begin{tikzcd}
\Delta'_e \arrow{r}{F_e^{N_e}} \arrow{d}{\pi_e} & \Delta'_e \arrow{d}{\pi_e} \arrow{r}{\text{inclusion}} & \Delta_e\\%
X \arrow{r}{(f^R)^{N_e}}& X &  \end{tikzcd}, \quad \begin{tikzcd}
\Delta'_e \arrow{r}{F_e^{N_e}} \arrow{d}{\pi_e} & \Delta'_e \arrow{d}{\pi_e} \arrow{r}{\text{inclusion}} & \Delta_e\\%
X \arrow{r}{g^{N_e}}& X &  \end{tikzcd}.\]

This is possible since $(\pi_{e})_*\mu_{\Delta_e'}=\mu_X$ by Lemma \ref{resultsofyoungpaper}. Besides, throughout this section, we define $\pi_\partial: \mathcal{M}=\partial Q \times [-\pi/2, \pi/2] \to \partial Q$ by \[\pi_{\partial}(q,v)=q,\]
and suppose that\begin{align}\label{eachxiisaconstant}
    R|_{X_n}=R_n' \text{ for some }R_n' \in \mathbb{N}.
\end{align} It is the case because $X$ is divided by a part of $\mathbb{S}$ into countably many pieces $X=\bigcup_i X_i$, so that $R$ is constant on each $X_i$ (see Definition \ref{inducesystem}).

\subsection{Return statistics}In this subsection, we will prove that for a.e. $\pmb{m}\in\mathcal{M}$, \begin{align*}
    \liminf_{r\to 0}\frac{\log \mathcal{Z}_{2r}(\pmb{m})}{-\log \mu_{\Delta_m}(\pi_{\Delta_m}\pi^{-1}S_r)}\ge 1,
\end{align*}where for any $\pmb{m}\in \mathcal{M}$,  \begin{gather*}
    \mathcal{Z}_r(\pmb{m}):=\min\{n \ge 1: g^n(\pi_X\pi^{-1}\pmb{m})\in \pi_X \pi^{-1}\big(B_r(\pi_{\partial}\pmb{m})\times [-\pi/2, \pi/2]\big)\},\\
    S_r \text{ is the section contained in the quasi-section }B_r(\pi_{\partial}\pmb{m})\times [-\pi/2, \pi/2].
\end{gather*}This inequality will be proved in Lemma \ref{returnlimit}, followed by a series of technical lemmas in this subsection.

Define for any $n \ge 1, r>0$, \begin{gather*}
    A_r(\pi_{\partial}\pmb{m}):=\bigcup_{0\le i\le N_e-1}g^{i}\pi_{X}\pi^{-1}\big(B_r(\pi_{\partial}\pmb{m})\times [-\pi/2, \pi/2]\big)\\
    G_{n,r}:=\{\pmb{m}\in \mathcal{M}: (g^{N_e})^n(\pi_X \pi^{-1}\pmb{m})\in A_r(\pi_\partial \pmb{m})\}.
\end{gather*}

Fix $p'\in \mathbb{N}$, which will be exactly determined later. From Lemma \ref{resultsofyoungpaper} almost surely \begin{gather*}
    \bigcup_{Z\in \mathcal{Q}_{2p'}}\pi_eF_e^{N_ep'}Z=g^{N_ep'}(\pi_e\bigcup_{Z\in \mathcal{Q}_{2p'}}Z)=X=\bigcup_iX_i \text{ almost surely}.
\end{gather*} 

For each $Z \in \mathcal{Q}_{2p'}$, we know that $\pi_eF_e^{N_ep'}Z$ is the intersection of a family of stable disks and a family of unstable disks. Thus by Assumption \ref{assumption} $\pi_eF_e^{N_ep'}Z$ is completely contained in some $X_i$. Therefore for each $X_i$ we define a family of sets $(Z_{ji})_{j} \subseteq  \mathcal{Q}_{2p'}$ such that each $Z_{ji}$ satisfies the relation \begin{gather*}
  \pi_e F_e^{N_ep'}Z_{ji}\subseteq X_i.
\end{gather*} 

Therefore $\bigcup_{j}\pi_e F_e^{N_ep'}Z_{ji}\subseteq X_i$. 
In fact we have a stronger result.

\begin{lemma}\label{decomposexi}
$\bigcup_{j} Z_{ji}=F_e^{-N_ep'}\pi_e^{-1}X_i$ almost surely, which implies that $\bigcup_{j}\pi_e F_e^{N_ep'} Z_{ji}=X_i$ almost surely.
\end{lemma}
\begin{proof}
For a.e. $x\in X_i$ there is $z\in Z\subseteq \Delta_e$ for some $Z\in \mathcal{Q}_{2p'}$, such that $\pi_eF_e^{N_ep'}z=x$. By Assumption \ref{assumption} we know that $\pi_eF_e^{N_ep'}Z$ is completely contained in some $X_j$, which must be $X_i$, since $\pi_eF_e^{N_ep'}Z$ also contains $x\in X_i$. From the way how $Z_{ji}$ was chosen, we know that $Z$ is one of the $Z_{ji}$. Therefore $\bigcup_{j} Z_{ji}=F_e^{-N_ep'}\pi_e^{-1}X_i$ almost surely. 
\end{proof}

Since $\mathcal{M}=\bigcup_{i}\bigcup_{k <R_i'}f^kX_i$, in order to estimate $\mu_{\mathcal{M}}(G_{n,r})$, we will estimate a measure of $f^kX_i \bigcap G_{n,r}=f^k(\bigcup_{j}\pi_e F_e^{N_ep'} Z_{ji}) \bigcap G_{n,r}$, where $k< R_i'$.  And finally we will sum up these estimates.

For each $Z_{ji}$, we choose and fix  $m_{jik} \in f^k(\pi_e F_e^{N_ep'} Z_{ji})\bigcap G_{n,r}$. Then the following statement holds.
\begin{lemma}\label{decomposegnr}
There is $C_{\alpha}>0$ such that for each $k< R_i'$, \begin{gather*}
    \mu_{\mathcal{M}}(f^kX_i\bigcap G_{n,r})\le \sum_j\mu_{\Delta_e'}\Big(Z_{ji}\bigcap F_e^{-N_en}F_e^{-N_ep'}\pi_e^{-1} A_{r+C_{\alpha}\beta^{\alpha p'N_e}}(\pi_{\partial}m_{jik})\Big).
\end{gather*}
\end{lemma}
\begin{proof}
For any $\pmb{m}\in f^k(\pi_e F_e^{N_ep'} Z_{ji})\bigcap G_{n,r}$ we have \begin{gather*}
  (g^{N_e})^n(\pi_X\pi^{-1}\pmb{m})\in A_r(\pi_{\partial} \pmb{m}), \quad (g^{N_e})^n(\pi_X\pi^{-1}m_{jik})\in A_r(\pi_{\partial} m_{jik}).
\end{gather*}

Since $Z_{ji}$ has a product structure, then $f^k(\pi_e F_e^{N_ep'} Z_{ji})$ is close to a product structure in $\mathcal{M}$. So it is an intersection of families of stable and unstable disks, where each stable disk $\gamma^s$ and each unstable disk $\gamma^u$ intersect exactly at one point in $f^k\pi_e F_e^{N_ep'} Z_{ji}$. However, the angles between unstable and stable disks are not uniformly bounded away from $0$. Then there is $z\in Z_{ji}$ such that $\pmb{m}\in f^k\pi_e F_e^{N_ep'}\gamma^u(z)$ and $m_{jik}\in f^k\pi_e F_e^{N_ep'}\gamma^s(z)$. Then by Lemma \ref{resultsofyoungpaper},\begin{gather*}
    \diam \big(\pi_eF_e^{N_ep'}\gamma^u(z)\big)\le C\beta^{p'N_e},\quad \diam \big(\pi_eF_e^{N_ep'}\gamma^s(z)\big)\le C\beta^{p'N_e}.
\end{gather*} 

By Assumption \ref{assumption} we have\begin{gather*}
    \diam \big(f^k\pi_eF_e^{N_ep'}\gamma^u(z)\big)\le C^{1+\alpha}\beta^{\alpha p'N_e},\quad \diam \big(f^k\pi_eF_e^{N_ep'}\gamma^s(z)\big)\le C^{1+\alpha}\beta^{\alpha p'N_e}.
\end{gather*} 

Therefore, $d(\pmb{m},m_{jik})\le C_{\alpha}\beta^{\alpha p'N_e}$ for some $C_{\alpha}>0$. Having this estimate, we can compare $A_r(\pi_{\partial}\pmb{m})$ and $A_r(\pi_{\partial}m_{jik})$.

\textbf{Claim}: $A_r(\pi_{\partial}\pmb{m})\subseteq A_{r+C_{\alpha}\beta^{\alpha p'N_e}}(\pi_{\partial}m_{jik})$.

For each $x\in A_r(\pi_{\partial}\pmb{m})$ there exists $0\le i \le N_e-1$, such that\begin{align*}
    x\in g^{i}\pi_{X}\pi^{-1}\big(B_r(\pi_{\partial}\pmb{m})\times [-\pi/2, \pi/2]\big)\subseteq g^{i}\pi_{X}\pi^{-1}\big(B_{r+C_{\alpha}\beta^{\alpha p'N_e}}(\pi_{\partial}m_{jik})\times [-\pi/2, \pi/2]\big),
\end{align*}which means that $x\in A_{r+C_{\alpha}\beta^{\alpha p' N_e}}(\pi_{\partial}m_{jik})$. Hence, this claim holds.

\textbf{Claim}: $Z_{ji}\bigcap F_e^{-N_ep'}\pi_e^{-1}f^{-k}G_{n,r}\subseteq Z_{ji}\bigcap F_e^{-N_ep'}\pi_e^{-1}f^{-k}\pi \pi_X^{-1} (g^{N_e})^{-n}A_{r+C_{\alpha}\beta^{\alpha p'N_e}}(\pi_{\partial}m_{jik})$.

For any $z \in Z_{ji}\bigcap F_e^{-N_ep'}\pi_e^{-1}f^{-k}G_{n,r}$ we have $f^k\pi_eF_e^{N_ep'}z\in f^k(\pi_e F_e^{N_ep'} Z_{ji})\bigcap G_{n,r}$, and \begin{gather*}
  (g^{N_e})^n(\pi_X\pi^{-1}f^k\pi_eF_e^{N_ep'}z)\in A_r(\pi_{\partial} f^k\pi_eF_e^{N_ep'}z)\subseteq A_{r+C_{\alpha}\beta^{\alpha p'N_e}}(\pi_{\partial}m_{jik}),
\end{gather*} i.e., $z \in Z_{ji}\bigcap F_e^{-N_ep'}\pi_e^{-1}f^{-k}\pi \pi_X^{-1} (g^{N_e})^{-n}A_{r+C_{\alpha}\beta^{\alpha p'N_e}}(\pi_{\partial}m_{jik})$. So, this claim holds.

Using the claims above, Lemma \ref{decomposexi}, the relations $f_* \mu_{\mathcal{M}}=\mu_{\mathcal{M}}$ and $(F_e^{N_e})_{*}\mu_{\Delta_e'}=\mu_{\Delta_e'}$ we can estimate\begin{align*}
    \mu_{\mathcal{M}}(f^kX_i\bigcap G_{n,r})&=\mu_{\mathcal{M}}(X_i\bigcap f^{-k}G_{n,r})=\mu_{\Delta_e'}(\pi_e^{-1}X_i\bigcap \pi_e^{-1}f^{-k}G_{n,r})\\
    &=\mu_{\Delta_e'}(F_e^{-N_ep'}\pi_e^{-1}X_i\bigcap F_e^{-N_ep'}\pi_e^{-1}f^{-k}G_{n,r})\\
    &=\mu_{\Delta_e'}(\bigcup_{j}Z_{ji}\bigcap F_e^{-N_ep'}\pi_e^{-1}f^{-k}G_{n,r})\\
    &\le \sum_j\mu_{\Delta_e'}\Big(Z_{ji}\bigcap F_e^{-N_ep'}\pi_e^{-1}f^{-k}\pi \pi_X^{-1} (g^{N_e})^{-n}A_{r+C_{\alpha}\beta^{\alpha p'N_e}}(\pi_{\partial}m_{jik})\Big).
\end{align*}

Since $\pi_eF_e^{p'N_e}Z_{ji} \subseteq X_i$ and $k<R'_i$, then $\pi_X \pi^{-1}f^k=\pi_X F^k\pi^{-1}$ is an identity map on $\pi_eF_e^{p'N_e}Z_{ji}$. Using this we can continue our estimate above as\begin{align*}
    &\le \sum_j\mu_{\Delta_e'}\Big(Z_{ji}\bigcap F_e^{-N_ep'}\pi_e^{-1} (g^{N_e})^{-n}A_{r+C_{\alpha}\beta^{\alpha p'N_e}}(\pi_{\partial}m_{jik})\Big)\nonumber\\
    &\le \sum_j\mu_{\Delta_e'}\Big(Z_{ji}\bigcap F_e^{-N_en}F_e^{-N_ep'}\pi_e^{-1} A_{r+C_{\alpha}\beta^{\alpha p'N_e}}(\pi_{\partial}m_{jik})\Big).
\end{align*}

Thus the lemma is proved.
\end{proof}

In order to proceed with further estimates we need to study $F_e^{-N_ep'}\pi_e^{-1} A_{r+C_{\alpha}\beta^{\alpha p'N_e}}(\pi_{\partial}m_{jik})$. Define a family of sets $(Z')\subseteq \mathcal{Q}_{2p'}$ such that \begin{align}\label{z'cover}
    \bigcup Z' \supseteq F_e^{-N_ep'}\pi_e^{-1} A_{r+C_{\alpha}\beta^{\alpha p'N_e}}(\pi_{\partial}m_{jik}),
\end{align}and each $Z'$ in this family satisfies to \begin{gather*}
    Z' \bigcap F_e^{-N_ep'}\pi_e^{-1} A_{r+C_{\alpha}\beta^{\alpha p'N_e}}(\pi_{\partial}m_{jik}) \neq \emptyset.
\end{gather*}

By Lemma \ref{resultsofyoungpaper}, and by definitions of $Z'$ and $A_{r+C_{\alpha}\beta^{\alpha p'N_e}}(\pi_{\partial}m_{jik})$, we have $\diam \big(\pi_eF_e^{N_ep'}Z'\big)\le C\beta^{p'N_e}$, and there exists the smallest integer $k'\in[0,N_e)$, which depends on $Z'$, such that\begin{align*}
    g^{-k'}\pi_e F_e^{N_ep'}Z'\bigcap \pi_{X}\pi^{-1}\big(B_{r+C_{\alpha}\beta^{\alpha p' N_e}}(\pi_{\partial}m_{jik})\times [-\pi/2, \pi/2]\big)\neq \emptyset,
\end{align*} that is, $\pi_e F_e^{N_ep'-k'}Z'\bigcap \pi_{X}\pi^{-1}\big(B_{r+C_{\alpha}\beta^{\alpha p' N_e}}(\pi_{\partial}m_{jik})\times [-\pi/2, \pi/2]\big)\neq \emptyset$. 

For each $X_{i'}$, we collect all $Z'$, e.g., $\bigcup_{j'} Z_{j'i'}'$ such that $\bigcup_{j'}\pi_e F_e^{N_ep'-k'}Z'_{j'i'}\subseteq X_{i'}$. Then for each $k'\in [0,N_{e})$, we collect $Z'_{j'i'}$ from $(Z'_{j'i'})_{j'}$, say, $Z'_{j'i'k'}$, such that
\begin{align*}
    \pi_e F_e^{N_ep'-k'}Z'_{j'i'k'}\bigcap \pi_{X}\pi^{-1}\big(B_{r+C_{\alpha}\beta^{\alpha p' N_e}}(\pi_{\partial}m_{jik})\times [-\pi/2, \pi/2]\big)\neq \emptyset.
\end{align*}

Hence, for each fixed $i'$ and $k'\in [0,N_e)$, there exists a set of all such allowable $j'$, which we denote by $S(i',k')$. Therefore, $\bigcup_{j'}\pi_e F_e^{N_ep'}Z'_{j'i'}=\bigcup_{k'=0}^{N_e-1}\bigcup_{j' \in S(i',k')}\pi_e F_e^{N_ep'}Z'_{j'i'k'}$. Let  a function $K_{i'}: \pi_e F_e^{N_ep'-k'} \bigcup_{j' \in S(i',k')}Z'_{j'i'k'} \to  \mathbb{N}_0$ be defined as
\begin{align*}
    K_{i'}(x):&=\min \{n\ge 0: x \in \pi_e F_e^{N_ep'-k'} Z'_{j'i'k'} \text{ for some }j' \in S(i',k') \text{ such that }\\
    &F^n \pi_e F_e^{N_ep'-k'} Z'_{j'i'k'}\bigcap \pi^{-1}\big(B_{r+C_{\alpha}\beta^{\alpha p' N_e}}(\pi_{\partial}m_{jik})\times [-\pi/2, \pi/2]\big)\neq \emptyset\}.
\end{align*}

Clearly, $F^{K_{i'}}: \pi_e F_e^{N_ep'-k'} \bigcup_{j' \in S(i',k')}Z'_{j'i'k'} \to \Delta \bigcap (X_{i'}\times \mathbb{N})$ is an injective function, and pushes $\pi_e F_e^{N_ep'-k'} \bigcup_{j' \in S(i',k')}Z'_{j'i'k'} $ up to $\Delta$, so that it intersect $\pi^{-1}\Big(B_{r+C_{\alpha}\beta^{\alpha p' N_e}}(\pi_{\partial}m_{jik})\times [-\pi/2, \pi/2]\Big)$. With these preparations, we can now estimate $\mu_{\Delta_e'}(\bigcup Z')$.

\begin{lemma}\label{estimatez'}
$\mu_{\Delta_e'}(\bigcup Z')\precsim r+C'_{\alpha}\beta^{\alpha p' N_e}$ for some $ C_{\alpha}'>0$.
\end{lemma}
\begin{proof}
Now using $(\pi_e)_*\mu_{\Delta_e'}=\mu_X, (F^{N_e}_e)_*\mu_{\Delta_e'}=\mu_{\Delta_e'}$ and $g_*\mu_X=\mu_X$, we have\begin{align}
    \mu_{\Delta_e'}(\bigcup Z')&\le \mu_{\Delta_e'}(F_e^{-N_ep'}\pi_e^{-1} \pi_e F_e^{N_ep'}\bigcup Z')\nonumber\\
    &=\mu_{X}(\pi_e F_e^{N_ep'}\bigcup Z')\nonumber\\
    &\le \sum_{i'}\mu_{X}(\pi_e F_e^{N_ep'} \bigcup_{j'}Z'_{j'i'})\nonumber\\
    &\le  \sum_{i'}\sum_{k'}\mu_{X}(\pi_e F_e^{N_ep'} \bigcup_{j' \in S(i',k')}Z'_{j'i'k'})\nonumber\\
    &= \sum_{i'}\sum_{k'}\mu_{X}(g^{-k'}\pi_e F_e^{N_ep'} \bigcup_{j' \in S(i',k')}Z'_{j'i'k'})\nonumber\\
    &=  \sum_{i'}\sum_{k'}\mu_{X}(\pi_e F_e^{N_ep'-k'} \bigcup_{j' \in S(i',k')}Z'_{j'i'k'})\nonumber\\
    &\precsim \sum_{i'}\sum_{k'}\mu_{\Delta}(\pi_e F_e^{N_ep'-k'} \bigcup_{j' \in S(i',k')}Z'_{j'i'k'})\nonumber\\
    &\precsim  \sum_{i'}\sum_{k'}\mu_{\Delta}\Big(F^{K_{i'}}\pi_e F_e^{N_ep'-k'} \bigcup_{j' \in S(i',k')}Z'_{j'i'k'}\Big)\label{measurez'},
\end{align}where``$\precsim$" in (\ref{measurez'}) is due to the injectivity and the measure of $\Delta$.

\textbf{Claim}: There is a constant $C_{\alpha}'>0$ such that $F^{K_{i'}}\pi_e F_e^{N_ep'-k'} \bigcup_{j'\in S(i',k')}Z'_{j'i'k'}\subseteq \big(X_{i'}\times \mathbb{N}_0\big) \bigcap \pi^{-1}\big(B_{r+C'_{\alpha}\beta^{\alpha p' N_e}}(\pi_{\partial}m_{jik})\times [-\pi/2, \pi/2]\big)$.

For any $x \in  \bigcup_{j'\in S(i',k')} \pi_e F_e^{N_ep'-k'}Z'_{j'i'k'}$,  by the definition of $K_{i'}$, there is $j' \in S(i',k')$ such that $x \in \pi_e F_e^{N_ep'-k'} Z'_{j'i'k'}$ and 

\[F^{K_{i'}(x)}\pi_e F_e^{N_ep'-k'} Z'_{j'i'k'} \bigcap \pi^{-1}\big(B_{r+C_{\alpha}\beta^{\alpha p' N_e}}(\pi_{\partial}m_{jik})\times [-\pi/2, \pi/2]\big) \neq \emptyset.\] 

For any $z,y \in Z'_{j'i'k'}$ there is $o\in Z'_{j'i'k'}$ such that $z \in \gamma^u(o)$, and $y\in \gamma^s(o)$. By Definition \ref{cmz} and Assumption \ref{assumption} we have \begin{gather*}
    \diam \big(\pi_eF_e^{N_ep'-k'} \gamma^s(o)\big)\le C \beta^{N_ep'-k'}, \quad \diam \big(\pi_eF_e^{N_ep'-k'} \gamma^u(o)\big)\le C \beta^{N_ep'+k'},
    \end{gather*}
    \begin{align*}
        d(f^{K_{i'}(x)}\pi_e&F_e^{N_ep'-k'}z,   f^{K_{i'}(x)}\pi_eF_e^{N_ep'-k'}y)\\
        &\le \diam \big(f^{K_{i'}(x)}\pi_eF_e^{N_ep'-k'} \gamma^u(o)\big)+\diam \big(f^{K_{i'}(x)}\pi_eF_e^{N_ep'-k'} \gamma^s(o)\big)\le 2C^{\alpha}\beta^{\alpha N_ep'-\alpha N_e}. 
    \end{align*}
   
We can now estimate the distance between  $\pi F^{K_{i'}(x)}x$ and $B_{r+C_{\alpha}\beta^{\alpha p' N_e}}(\pi_{\partial}m_{jik})\times [-\pi/2, \pi/2]$ as follows \begin{align*}
   \dist \big(\pi F^{K_{i'}(x)}x, B_{r+C_{\alpha}\beta^{\alpha p' N_e}}(\pi_{\partial}&m_{jik})\times [-\pi/2, \pi/2]\big)\\
   &\le \diam \big(\pi F^{K_{i'}(x)}\pi_e F_e^{N_ep'-k'} Z'_{j'i'k'}\big)\\
   &=\diam \big(f^{K_{i'}(x)}\pi_e F_e^{N_ep'-k'} Z'_{j'i'k'}\big)\le  2^{\alpha}C^{1+\alpha}\beta^{\alpha N_ep'-\alpha N_e}.
\end{align*}

Therefore, \[\pi F^{K_{i'}}\bigcup_{j'\in S(i',k')}\pi_e F_e^{N_ep'-k'} Z'_{j'i'k'}\subseteq B_{r+C'_{\alpha}\beta^{\alpha p' N_e}}(\pi_{\partial}m_{jik})\times [-\pi/2, \pi/2],\]
where $C_{\alpha}':=2^{\alpha}C^{1+\alpha}\beta^{ -\alpha N_e}+C_{\alpha}$. Hence. the claim holds.

Now, using the claims above and the relation $\Delta \bigcap \bigcup_{i'} (X_{i'}\times \mathbb{N}_0)=\Delta$, we can continue an estimate of (\ref{measurez'}) as

\begin{align*}
    &\le \sum_{i'}\sum_{k'}\mu_{\Delta}\big[\big(X_{i'}\times \mathbb{N}_0\big) \bigcap \pi^{-1}\big(B_{r+C'_{\alpha}\beta^{\alpha p' N_e}}(\pi_{\partial}m_{jik})\times [-\pi/2, \pi/2]\big)\big]\\
    &\le N_e\sum_{i'}\mu_{\Delta}\big[\big(X_{i'}\times \mathbb{N}_0\big) \bigcap \pi^{-1}\big(B_{r+C'_{\alpha}\beta^{\alpha p' N_e}}(\pi_{\partial}m_{jik})\times [-\pi/2, \pi/2]\big)\big]\\
    &\precsim \mu_{\Delta}\big[\pi^{-1}\big(B_{r+C'_{\alpha}\beta^{\alpha p' N_e}}(\pi_{\partial}m_{jik})\times [-\pi/2, \pi/2]\big)\big]\\
    &\precsim \mu_{\mathcal{M}}\big(B_{r+C'_{\alpha}\beta^{\alpha p' N_e}}(\pi_{\partial}m_{jik})\times [-\pi/2, \pi/2]\big)\precsim r+C'_{\alpha}\beta^{\alpha p' N_e},
\end{align*}where the last $``\precsim"$ holds because $\frac{d\mu_{\mathcal{M}}}{d\Leb_{\mathcal{M}}}\in L^{\infty}$, and because $\partial Q$ has an uniformly bounded curvature.
\end{proof}

We choose $p'=n/4$, and estimate now $\mu_{\mathcal{M}}(G_{n,r})$.
\begin{lemma}\label{gnr}
$\mu_{\mathcal{M}}(G_{n,r})\le C' \beta^{n/2}+C'(r+C'_{\alpha}\beta^{\alpha N_e n/4}) $ for some $C'>0$.
\end{lemma}
\begin{proof}
By (\ref{z'cover}) and Lemma \ref{decomposegnr} we have\begin{align*}
    \mu_{\mathcal{M}}(G_{n,r})&=\sum_{i}\sum_{k <R_i'}\mu_{\mathcal{M}}(f^kX_i\bigcap G_{n,r})\\
    &\le \sum_{i}\sum_{k <R_i'} \sum_j\mu_{\Delta_e'}\Big(Z_{ji}\bigcap F_e^{-N_en}F_e^{-N_ep'}\pi_e^{-1} A_{r+C_{\alpha}\beta^{\alpha p'N_e}}(\pi_{\partial}m_{jik})\Big)\\
    &\le \sum_{i}\sum_{k <R_i'} \sum_j\mu_{\Delta_e'}\Big(Z_{ji}\bigcap F_e^{-N_en}\bigcup Z'\Big)-\mu_{\Delta_e'}(Z_{ji})\mu_{\Delta_e'}(\bigcup Z')+\mu_{\Delta_e'}(Z_{ji})\mu_{\Delta_e'}(\bigcup Z')\\
     &\precsim \sum_{i}\sum_{k <R_i'} \sum_j  \beta^{n-2p'}\mu_{\Delta_e'}(Z_{ji})+\mu_{\Delta_e'}(Z_{ji})(r+C'_{\alpha}\beta^{\alpha p' N_e}),
\end{align*} where the last $``\precsim"$ is due to  (\ref{decorrelationhyperbolic}) and Lemma \ref{estimatez'}. 

By making use of Lemma \ref{decomposexi}, $(\pi_e)_{*}\mu_{\Delta_e'}=\mu_X$, $\int Rd\mu_X< \infty$ and $(F_e^{N_e})_{*}\mu_{\Delta_e'}=\mu_{\Delta_e'}$, we can now continue the estimate above as \begin{align*}
    &\precsim\sum_{i}\sum_{k <R_i'}  \beta^{n-2p'}\mu_{\Delta_e'}(F_e^{-N_ep'}\pi_e^{-1}X_i)+\mu_{\Delta_e'}(F_e^{-N_ep'}\pi_e^{-1}X_i)(r+C'_{\alpha}\beta^{\alpha p' N_e})\\
    &\precsim\sum_{i}\sum_{k <R_i'}  \beta^{n-2p'}\mu_{\Delta_e'}(\pi_e^{-1}X_i)+\mu_{\Delta_e'}(\pi_e^{-1}X_i)(r+C'_{\alpha}\beta^{\alpha p' N_e})\\
    &\precsim\sum_{i}\sum_{k <R_i'}  \beta^{n-2p'}\mu_{X}(X_i)+\mu_{X}(X_i)(r+C'_{\alpha}\beta^{\alpha p' N_e})\\
    &\precsim\sum_{i}R_i' \beta^{n-2p'}\mu_{X}(X_i)+R_i'\mu_{X}(X_i)(r+C'_{\alpha}\beta^{\alpha p' N_e})\\
    &\precsim\int R d\mu_X \beta^{n-2p'}+\int Rd\mu_{X}(r+C'_{\alpha}\beta^{\alpha p' N_e})\precsim \beta^{n/2}+r+C'_{\alpha}\beta^{\alpha N_e n/4},
\end{align*}which concludes a proof of this lemma.
\end{proof}

Now we can estimate the range of $\mathcal{Z}_{r}(\pmb{m})$ for a particular $r>0$.
\begin{lemma}\label{rangeofzr}
Let any $\delta>0$ be sufficiently small. Then for a.e. $\pmb{m}\in \mathcal{M}$, there exists $N_{\pmb{m}} \in \mathbb{N}$ such that for any $n>N_e N_{\pmb{m}}$, \begin{align*}
    \mathcal{Z}_{\lfloor n/N_e\rfloor^{-(1+\delta)}}(\pmb{m})\in \{1,2, \cdots, N_eN_{\pmb{m}}-N_e\}\bigcup (n-N_e,\infty).
\end{align*}
\end{lemma}
\begin{proof}
Let $r=n^{-(1+\delta)}$. By Lemma \ref{gnr} we have\begin{align*}
    \mu_{\mathcal{M}}(G_{n,n^{-(1+\delta)}}) \precsim \beta^{n/2}+(n^{-(1+\delta)}+C_{\alpha}'\beta^{\alpha N_e n/4})\precsim n^{-1-\delta}.
\end{align*}

By Borel-Cantelli lemma, for a.e. $\pmb{m}\in \mathcal{M}$ there exists $N_{\pmb{m}}>1$ such that for any $n > N_{\pmb{m}}$, \begin{align*}
    (g^{N_e})^n(\pi_X \pi^{-1}\pmb{m})\notin \bigcup_{0\le i\le N_e-1}g^{i}\pi_{X}\pi^{-1}\big(B_{n^{-(1+\delta)}}(\pi_{\partial}\pmb{m})\times [-\pi/2, \pi/2]\big),
\end{align*}which means that for any $n> N_e N_{\pmb{m}}$, \begin{align*}
    g^{n-N_e}(\pi_X \pi^{-1}\pmb{m}) \notin \pi_{X}\pi^{-1}\big(B_{\lfloor n/N_e\rfloor^{-(1+\delta)}}(\pi_{\partial}\pmb{m})\times [-\pi/2, \pi/2]\big).
\end{align*}

Furthermore, for any $n\ge m>N_eN_{\pmb{m}}$, \begin{align*}
   g^{m-N_e}(\pi_X \pi^{-1}\pmb{m}) \notin \pi_{X}\pi^{-1}\big(B_{\lfloor n/N_e\rfloor^{-(1+\delta)}}(\pi_{\partial}\pmb{m})\times [-\pi/2, \pi/2]\big). 
\end{align*}

Therefore, \begin{align*}
    \mathcal{Z}_{\lfloor n/N_e\rfloor^{-(1+\delta)}}(\pmb{m})&=\min\{m \ge 1: g^m(\pi_X\pi^{-1}\pmb{m})\in \pi_X \pi^{-1}\big(B_{\lfloor n/N_e\rfloor^{-(1+\delta)}}(\pi_{\partial}\pmb{m})\times [-\pi/2, \pi/2]\big)\}\\
    &\in \{1,2, \cdots, N_eN_{\pmb{m}}-N_e\}\bigcup (n-N_e,\infty).
\end{align*}

Thus the lemma is proved.
\end{proof}

To rule out the set $\{1,2,\cdots, N_eN_{\pmb{m}}-N_e\}$, we need the following lemma.

\begin{lemma}[Aperiodicity]\label{aperiodic}\ \par
For a.e. $\pmb{m}\in \mathcal{M}$ and for any $k\in \mathbb{N}$, $g^k(\pi_X\pi^{-1}\pmb{m})\notin \pi_X \pi^{-1}(\{\pi_{\partial}\pmb{m}\}\times [-\pi/2, \pi/2])$.
\end{lemma}
\begin{proof}
For any $q \in \partial Q$, let $\bar{A}_0(q):=\pi_X \pi^{-1} (\{q\}\times [-\pi/2, \pi/2])$. If $q $ is a periodic point, then there exists $x\in \bar{A}_0(q)$ such that there is $k\ge 1$ satisfying $g^k(x) \in \bar{A}_0(q)$. Therefore,  \begin{gather*}
   g(x) \in g(\bar{A}_0(q)), \quad g(x) \in g^{-(k-1)}(\bar{A}_0(q)).
\end{gather*}

If $k>1$, then by Assumption \ref{assumption},\begin{gather*}
    \mathcal{T}_{g(x)}g\big(\bar{A}_0(q)\big)\subseteq (Df^i)\mathcal{T}(\{q\}\times [-\pi/2, \pi/2])\subseteq \interior C^u \text{ for some } i\ge 1,\\
    \mathcal{T}_{g(x)}g^{-(k-1)}\big(\bar{A}_0(q)\big)\subseteq (Df^{-j})\mathcal{T}(\{q\}\times [-\pi/2, \pi/2])\subseteq \interior C^s \text{ for some } j\ge 1.
\end{gather*} 

If $k=1$, then by Assumption \ref{assumption},\begin{gather*}
    \mathcal{T}_{g(x)}g\big(\bar{A}_0(q)\big)\subseteq (Df^i)\mathcal{T}(\{q\}\times [-\pi/2, \pi/2])\subseteq \interior C^u \text{ for some } i\ge 1,
    \end{gather*}and there exists $j \ge 0$ such that \begin{eqnarray*}
 \mathcal{T}_{g(x)}g^{-(k-1)}\big(\bar{A}_0(q)\big)\subseteq (Df^{-j})\mathcal{T}(\{q\}\times [-\pi/2, \pi/2])\subseteq 
\begin{cases}
\interior C^s, & \text{if }j>0\\
\mathcal{T}(\{q\}\times [-\pi/2, \pi/2]),  & \text{if } j=0 \\
\end{cases}.
\end{eqnarray*}

Then by Assumption \ref{assumption}, $\dim \big\{\mathcal{T}_{g(x)}g\big(\bar{A}_0(q)\big) \bigcap \mathcal{T}_{g(x)}g^{-(k-1)}\big(\bar{A}_0(q)\big)\big\}<1$. On the other hand, since $\bar{A}_0(q)$ is a countable union of one-dimensional connected submanifolds, such are as well $g\big(\bar{A}_0(q)\big)$ and $g^{-(k-1)}\big(\bar{A}_0(q)\big)$. Their intersection is a union of countably-many points. Thus, \begin{align*}
    x\in g^{-1} \big[g\big(\bar{A}_0(q)\big)\bigcap g^{-(k-1)}\big(\bar{A}_0(q)\big)\big]
\end{align*} belongs to a union of countably-many points contained in $\bar{A}_0(q)$. Therefore, for any $q \in \partial Q$\begin{gather*}
    \Leb_{\bar{A}_0(q)} \{x\in X: g^{k}x \in \bar{A}_0(q) \text{ for some }k \ge 1\}=0,
\end{gather*}where $\Leb_{\bar{A}_0(q)}$ is the Lebesgue measure conditioned on the submanifold $\bar{A}_0(q)$. The same notations will be used below. 

By lifting it to $\Delta$ we have\begin{align*}
    \Leb_{\pi^{-1}(\{q\}\times [-\pi/2, \pi/2])} \{x\in \Delta: g^k(\pi_X x)\in \pi_X \pi^{-1} (\{q\}\times [-\pi/2, \pi/2]) \text{ for some } k\ge 1\}=0.
\end{align*}

Since $\pi$ is an isomorphism, then \begin{align*}
    \Leb_{\{q\}\times [-\pi/2, \pi/2]} \{\pmb{m}\in \mathcal{M}: g^k(\pi_X \pi^{-1}\pmb{m})\in \pi_X \pi^{-1} (\{q\}\times [-\pi/2, \pi/2]) \text{ for some } k\ge 1\}=0.
\end{align*}

By Fubini's theorem,\begin{align*}
    &\mu_{\mathcal{M}}\{\pmb{m}\in \mathcal{M}: g^k(\pi_X\pi^{-1}\pmb{m})\in \pi_X \pi^{-1}(\{\pi_{\partial}\pmb{m}\}\times [-\pi/2, \pi/2])\text{ for some }k \ge 1\}\\
    &\precsim \Leb_{\mathcal{M}} \{\pmb{m}\in \mathcal{M}: g^k(\pi_X\pi^{-1}\pmb{m})\in \pi_X \pi^{-1}(\{\pi_{\partial}\pmb{m}\}\times [-\pi/2, \pi/2])\text{ for some }k \ge 1\}\\
    &\precsim \int_{\partial Q} \Leb_{\{q\}\times [-\pi/2, \pi/2]} \{\pmb{m}\in \mathcal{M}: g^k(\pi_X\pi^{-1}\pmb{m})\in \pi_X \pi^{-1}(\{\pi_{\partial}\pmb{m}\}\times [-\pi/2, \pi/2])\text{ for some }k \ge 1\}dq\\
    &\precsim \int_{\partial Q} \Leb_{\{q\}\times [-\pi/2, \pi/2]} \{\pmb{m}\in \mathcal{M}: g^k(\pi_X\pi^{-1}\pmb{m})\in \pi_X \pi^{-1}(\{q\}\times [-\pi/2, \pi/2])\text{ for some }k \ge 1\}dq=0,
\end{align*}which concludes a proof of this lemma.
\end{proof}

Now we can proceed to proving the main result of this subsection.
\begin{lemma}\label{returnlimit}
For a.e. $\pmb{m}\in \mathcal{M}$,\begin{align*}
    \liminf_{r\to 0}\frac{\log \mathcal{Z}_{2r}(\pmb{m})}{-\log \mu_{\Delta_m}(\pi_{\Delta_m}\pi^{-1}S_r)}\ge 1,
\end{align*}where $S_r$ is the section contained in the quasi-section $B_r(\pi_\partial\pmb{m})\times [-\pi/2, \pi/2]$.
\end{lemma}
\begin{proof}
From Lemma \ref{aperiodic} for a.e. $\pmb{m}\in \mathcal{M}$ we have that for any $k=1,2,\cdots, N_eN_{\pmb{m}}-N_e$, and any $j<R(g^k\pi_X \pi^{-1}\pmb{m})$,\begin{gather*}\pi F^j(g^k\pi_X \pi^{-1}\pmb{m})\notin \{\pi_{\partial} \pmb{m}\}\times [-\pi/2, \pi/2].
\end{gather*} 

If it is not the case, then there is $j<R(g^k\pi_X \pi^{-1}\pmb{m})$ such that \begin{gather*}
    F^j(g^k\pi_X \pi^{-1}\pmb{m})\in \pi^{-1} (\{\pi_{\partial} \pmb{m}\}\times [-\pi/2, \pi/2]),
\end{gather*}which implies that $g^k\pi_X \pi^{-1}\pmb{m}\in \pi_X\pi^{-1} (\{\{\pi_{\partial} \pmb{m}\}\times [-\pi/2, \pi/2])$. But it is in contradiction with Lemma \ref{aperiodic}. Choose now a small $r_{\pmb{m}}>0$ such that for any $r\in (0,r_{\pmb{m}})$, any $k=1,2,\cdots, N_eN_{\pmb{m}}-N_e$, and any $j<R(g^k\pi_X \pi^{-1}\pmb{m})$, \begin{gather*}\pi F^j(g^k\pi_X \pi^{-1}\pmb{m})\notin B_{r}(\pi_{\partial} \pmb{m})\times [-\pi/2, \pi/2].
\end{gather*}

This implies that for any $k=1,2,\cdots, N_eN_{\pmb{m}}-N_e$, and for any $r\in (0,r_{\pmb{m}})$ \begin{gather*}g^k\pi_X \pi^{-1}\pmb{m}\notin \pi_X\pi^{-1}\big(B_{r}(\pi_{\partial} \pmb{m})\times [-\pi/2, \pi/2]\big).
\end{gather*}

Furthermore, for any $\delta>0$, and any $k=1,2,\cdots, N_eN_{\pmb{m}}-N_e$, if $n>N_er_{\pmb{m}}^{-1/(1+\delta)}$, then \begin{gather*}g^k\pi_X \pi^{-1}\pmb{m}\notin \pi_X\pi^{-1}\big(B_{\lfloor n/N_e\rfloor^{-(1+\delta)}}(\pi_{\partial} \pmb{m})\times [-\pi/2, \pi/2]\big).
\end{gather*}

It follows from Lemma \ref{rangeofzr} that for any $n>\max \{N_eN_{\pmb{m}}, N_er_{\pmb{m}}^{-1/(1+\delta)}\}$,\begin{align*}
    \mathcal{Z}_{\lfloor n/N_e\rfloor^{-(1+\delta)}}(\pmb{m})&=\min\Big\{m \ge 1: g^m(\pi_X\pi^{-1}\pmb{m})\in \pi_X \pi^{-1}\Big(B_{\lfloor n/N_e\rfloor^{-(1+\delta)}}(\pi_{\partial}\pmb{m})\times [-\pi/2, \pi/2]\Big)\Big\}\\
    &>n-N_e.
\end{align*}

Therefore, \begin{gather*}
    \liminf_{n\to \infty}\frac{\log  \mathcal{Z}_{\lfloor n/N_e\rfloor^{-(1+\delta)}}(\pmb{m})}{\log n} \ge 1.
\end{gather*}

Then for any sufficiently small $r\in(0, r_{\pmb{m}})$ such that $r\in (\lfloor (1+n)/N_e\rfloor^{-(1+\delta)}, \lfloor n/N_e\rfloor^{-(1+\delta)})$ for a sufficienly large $n>\max\{N_eN_{\pmb{m}}, N_er_{\pmb{m}}^{-1/(1+\delta)}\}$, we have $\mathcal{Z}_r(\pmb{m})\ge n-N_e$, and \begin{gather*}
     \liminf_{r\to 0}\frac{\log  \mathcal{Z}_{r}(\pmb{m})}{\log r^{-1/(1+\delta)}} \ge \liminf_{n\to \infty}\frac{\log (n-N_e)}{\log (n+1)/N_e}=1\\
     \implies \liminf_{r\to 0}\frac{\log  \mathcal{Z}_{r}(\pmb{m})}{-\log r} \ge 1/(1+\delta).
\end{gather*}

Note that $B_r(\pi_\partial \pmb{m})\times [-\pi/2, \pi/2]$ is a quasi-section, and if $r>0$ is sufficiently small, then, because $\partial Q$ has a uniformly bounded curvature and $\frac{d\mu_{\mathcal{M}}}{d\Leb_{\mathcal{M}}} \in L^{\infty}$, we have   \begin{align*}
    \mu_{\Delta_m}(\pi_{\Delta_m}\pi^{-1}S_r)\approx_m \mu_{\Delta}(\pi^{-1}S_r)\approx (\pi_*\mu_{\Delta})\big(B_r(\pi_\partial \pmb{m})\times [-\pi/2, \pi/2]\big)\approx r.
\end{align*}

Then \begin{align*}
    \liminf_{r\to 0}\frac{\log  \mathcal{Z}_{2r}(\pmb{m})}{-\log \mu_{\Delta_m}(\pi_{\Delta_m}\pi^{-1}S_r)}&=\liminf_{r\to 0}\frac{\log  \mathcal{Z}_{2r}(\pmb{m})}{-\log \mu_{\Delta_m}(\pi_{\Delta_m}\pi^{-1}S_{2r})} \frac{{\log \mu_{\Delta_m}(\pi_{\Delta_m}\pi^{-1}S_{2r})}}{{\log \mu_{\Delta_m}(\pi_{\Delta_m}\pi^{-1}S_{r})}}\\
    &=\liminf_{r\to 0}\frac{\log  \mathcal{Z}_{r}(\pmb{m})}{-\log r}  \ge 1/(1+\delta).
\end{align*} 

By letting now $\delta\to 0$ we conclude a proof of this lemma.
\end{proof}
\subsection{Short returns on $X$}
In this subsection, we will prove short returns on $X$, i.e., \begin{gather*}
    \lim_{r\to 0}\frac{1}{\mu_{X}(\pi_{X}\pi^{-1}S_r)}\int \mathbbm{1}_{\pi_{X}\pi^{-1}S_r} \mathbbm{1}_{\bigcup_{1\le j\le pN_{e,m}-1}(f^{R})^{-j}(\pi_{X}\pi^{-1}S_r)} d\mu_{X}=0.
\end{gather*}

Recall that $S_r:=S_r(\pmb{m})$ is a section in $B_{r}(\pi_{\partial} \pmb{m}) \times [-\pi/2, \pi/2]$. Let 
\begin{align*}
    \Lambda_{r,pN_{e,m}}(\pmb{m}):=\{y \in S_r(\pmb{m}): \text{ there exists }&i \ge 0, \text{ such that } y, f(y) \cdots f^i(y) \text{ visit }\\
    &X \text{ at most }pN_{e,m} \text{ times and }f^i(y) \in S_{r}(\pmb{m})\}.
\end{align*}

Since $\Lambda_{r,pN_{e,m}}(\pmb{m})\subseteq S_r(\pmb{m})$, then $\mu_{\mathcal{M}}\{\Lambda_{r, pN_{e,m}}(\pmb{m})\}\approx \mu_X\{\pi_X \pi^{-1}\Lambda_{r,pN_{e,m}}(\pmb{m})\}$. Now we have the following lemma.

\begin{lemma}\label{shortreturninclusion}For any $\pmb{m}\in \mathcal{M}$, \begin{gather*}
      \Lambda_{r,pN_{e,m}}(\pmb{m})\subseteq S_r(\pmb{m})\bigcap \{y\in \mathcal{M}: \mathcal{Z}_{2r}(y)\le pN_{e,m}\},\\
      \pi_X\pi^{-1}S_r(\pmb{m})\bigcap {\bigcup_{1\le j\le pN_{e,m}-1}g^{-j}\pi_X\pi^{-1}S_r(\pmb{m})} \subseteq \pi_X\pi^{-1}\Lambda_{r,pN_{e,m}}(\pmb{m}),
  \end{gather*}where we recall that $\mathcal{Z}_r(\pmb{m}):=\min\{n \ge 1: g^n(\pi_X\pi^{-1}\pmb{m})\in \pi_X \pi^{-1}\big(B_r(\pi_{\partial}\pmb{m})\times [-\pi/2, \pi/2]\big)\}$.
\end{lemma}
\begin{proof}
Let $x \in \Lambda_{r,pN_{e,m}}(\pmb{m})$, then $x \in S_r(\pmb{m})$, and there is $i\in \mathbb{N}$ such that $x, f(x), \cdots f^i(x)$ visit $X$ at most $pN_{e,m}$ times and $f^i(x) \in S_r(\pmb{m})$. Then $\pi_X\pi^{-1}x, \pi_X\pi^{-1}f(x), \cdots \pi_X\pi^{-1}f^i(x)$ visit $X$ at most $pN_{e,m}$ or $pN_{e,m}+1$ times. Thus, 
\begin{gather*}
    x, f^i(x)\in S_{r}(\pmb{m}) \subseteq B_{2r}(\pi_{\partial}x) \times [-\pi/2, \pi/2], \\
    \pi_X\pi^{-1}x, \pi_X \pi^{-1}f^i(x)\in \pi_X \pi^{-1}[B_{2r}(\pi_{\partial}x) \times [-\pi/2, \pi/2]].
\end{gather*}

Then $\mathcal{Z}_{2r}(x)\le pN_{e,m}$, that is, $\Lambda_{r,pN_{e,m}}(\pmb{m})\subseteq S_r(\pmb{m})\bigcap \{y\in \mathcal{M}: \mathcal{Z}_{2r}(y)\le pN_{e,m}\}$.

Now let $x \in \pi_X \pi^{-1}S_r(\pmb{m})$, $g^jx \in \pi_X \pi^{-1}S_r(\pmb{m})$ for some $j \in [1,pN_{e,m}-1]$. Then there are $k_1 < R(x)$,  $k_2<R(g^jx)$, such that $ y:=F^{k_1}(x)\in \pi^{-1}S_r(\pmb{m})$ and $F^{k_2}g^j(x)\in \pi^{-1}S_r(\pmb{m})$. So, there is $i\in \mathbb{N}$ such that $y, Fy, \cdots, F^iy$ visit $X$ at most $pN_{e,m}$ times and $F^iy\in \pi^{-1}S_r(\pmb{m})$, i.e., $\pi y \in \Lambda_{r,pN_{e,m}}(\pmb{m})$ and $x=\pi_X y \in \pi_X \pi^{-1}\Lambda_{r,pN_{e,m}}(\pmb{m})$. Therefore, we prove \[ \pi_X\pi^{-1}S_r(\pmb{m})\bigcap {\bigcup_{1\le j\le pN_{e,m}-1}g^{-j}\pi_X\pi^{-1}S_r(\pmb{m})} \subseteq \pi_X\pi^{-1}\Lambda_{r,pN_{e,m}}(\pmb{m}),\]which concludes the proof.\end{proof}

Now we can prove the main result of this subsection.
\begin{lemma}\label{thelastlemmaofshortreturn} For $\Leb_{\partial Q}$-a.e. $q\in \partial Q$ (and $\pmb{m}\in \pi_{\partial}^{-1}\{q\}$)
 \begin{gather*}
    \lim_{r\to 0}\frac{1}{\mu_{X}\big(\pi_X\pi^{-1}S_r\big)}\int_{\pi_X\pi^{-1}S_r}\mathbbm{1}_{\bigcup_{j=1}^{ pN_{e,m}-1}g^{-j}\pi_X\pi^{-1}S_r} d\mu_{X}=0.
\end{gather*}
\end{lemma}
\begin{proof} Recall that  $B_{r}(\pi_{\partial}\pmb{m})\times [-\pi/2,\pi/2]$ is a quasi-section, $S_r=S_r(\pmb{m})$ is a section, and $\Lambda_{r,pN_{e,m}}(\pmb{m})\subseteq S_r(\pmb{m})$. By using these and Lemma \ref{shortreturninclusion} we have \begin{align}
    \mu_X\Big\{\pi_X\pi^{-1}S_r(\pmb{m})&\bigcap {\bigcup_{1\le j\le pN_{e,m}-1}g^{-j}\pi_X\pi^{-1}S_r(\pmb{m})}\Big\}\nonumber\\
    &\le \mu_X(\pi_X\pi^{-1}\Lambda_{r,pN_{e,m}}(\pmb{m}))\nonumber\\
    &\precsim \mu_{\Delta}(\pi^{-1}\Lambda_{r,pN_{e,m}}(\pmb{m}))\nonumber\\
    &\precsim \mu_{\mathcal{M}}\big(\Lambda_{r,pN_{e,m}}(\pmb{m})\big) \nonumber\\
    &\precsim \mu_{\mathcal{M}}\Big[ S_r(\pmb{m})\bigcap \{y\in \mathcal{M}: \mathcal{Z}_{2r}(y)\le pN_{e,m}\}\Big]\nonumber\\
    &\precsim \mu_{\mathcal{M}}\Big[ \{B_r(\pi_\partial\pmb{m})\times [-\pi/2, \pi/2]\}\bigcap \{y\in \mathcal{M}: \mathcal{Z}_{2r}(y)\le pN_{e,m}\}\Big]\nonumber.
    \end{align}

 Besides, for any $\epsilon \in (0,1)$, there is $C_{\epsilon}>0$ such that \begin{gather*}
    p=n^{1-\epsilon}=C_{\epsilon}^{\pm 1} \mu_{\Delta_m}(\pi_{\Delta_m}\pi^{-1}S_r)^{-1+\epsilon}.
\end{gather*} Now we can continue the estimate above as
    \begin{align}
    &\precsim \mu_{\mathcal{M}}\Big[ \{B_r(\pi_\partial\pmb{m})\times [-\pi/2, \pi/2]\}\bigcap \{y\in \mathcal{M}: \mathcal{Z}_{2r}(y)\le C_{\epsilon} \mu_{\Delta_m}(\pi_{\Delta_m}\pi^{-1}S_r)^{-1+\epsilon}N_{e,m}\}\Big]\nonumber\\
    &\precsim \mu_{\mathcal{M}}\Big[ \{B_r(\pi_\partial\pmb{m})\times [-\pi/2, \pi/2]\}\bigcap \{y\in \mathcal{M}: \frac{-C_{\epsilon,m}+\log \mathcal{Z}_{2r}(y)}{-\log \mu_{\Delta_m}(\pi_{\Delta_m}\pi^{-1}S_r)}\le 1-\epsilon \}\Big]\nonumber\\
    &\precsim \mu_{\mathcal{M}}\Big[ \{B_r(\pi_\partial\pmb{m})\times [-\pi/2, \pi/2]\}\bigcap \{y\in \mathcal{M}: \inf_{r'<r_0}\frac{-C_{\epsilon,m}+\log \mathcal{Z}_{2r'}(y)}{-\log \mu_{\Delta_m}(\pi_{\Delta_m}\pi^{-1}S_{r'})}\le 1-\epsilon \}\Big]\label{2}
\end{align} where  $C_{\epsilon,m}=\log C_{\epsilon}+\log N_{e,m} $ and any $r_0>r$. 

Before proceeding to estimating (\ref{2}), we will need some preparations. By Lemma \ref{returnlimit}, we have \begin{align}
    \lim_{r_0 \to 0}\mu_{\mathcal{M}} \Big\{y\in \mathcal{M}: \inf_{r'<r_0}\frac{-C_{\epsilon,m}+\log \mathcal{Z}_{2r'}(y)}{-\log \mu_{\Delta_m}(\pi_{\Delta_m}\pi^{-1}S_{r'})}\le 1-\epsilon \Big\}=0.\label{inflimit0}
\end{align}

Let $h:=\frac{d\mu_{\mathcal{M}}}{d\Leb_{\mathcal{M}}}$ be the density of $\mu_{\mathcal{M}}$, and $\Leb_{[-\pi/2, \pi/2]}$, and $\Leb_{\partial Q}$ are the Lebesgue measures on $[-\pi/2, \pi/2]$ and $\partial Q$, respectively. For any $q \in \partial Q$, define now a measure $\mu_{q}$ \begin{align}
    \mu_q(A):=\int_A h(q, \cdot)d\Leb_{[-\pi/2, \pi/2]}\label{fiberwisemeasure}
\end{align} for any measurable $A\subseteq [-\pi/2, \pi/2]$.  Let also \begin{align*}
    U_{q, r_0}:=\big\{v\in [-\pi/2, \pi/2]: \inf_{r'<r_0}\frac{-C_{\epsilon,m}+\log \mathcal{Z}_{2r'}(q,v)}{-\log \mu_{\Delta_m}(\pi_{\Delta_m}\pi^{-1}S_{r'})}\le 1-\epsilon \big\}.
\end{align*}

Then (\ref{inflimit0}) is equivalent to \begin{align*}
    \lim_{r_0\to 0}\int_{\partial Q} \mu_q(U_{q, r_0})d\Leb_{\partial Q}=0.
\end{align*}

For any $\delta>0$ let $T_{r_0, \delta}:=\{q \in \partial Q: \mu_q(U_{q,r_0})>\delta\}$. Then $\lim_{r_0 \to 0}\Leb_{\partial Q}(T_{r_0,\delta})=0$. Now, using the Lebesgue differentiation theorem (which holds also on $\partial Q$), for any $r_0, \delta>0$, there is a full measure set $Q_{r_0, \delta}$ in $\partial Q$ such that for a.e. $q \in Q_{r_0, \delta}$, \begin{align*}
    \lim_{r \to 0}\frac{1}{\Leb_{\partial Q}\big(B_{r}(q)\big)} \int_{B_{r}(q)} \mathbbm{1}_{T_{r_0,\delta}}d\Leb_{\partial Q}=\mathbbm{1}_{T_{r_0,\delta}}(q).
\end{align*}

Hence, if $q \notin T_{r_0,\delta}$, then \begin{align*}
    \mu_q(U_{q,r_0})< \delta,\quad \int_{B_{r}(q)} \mathbbm{1}_{T_{r_0,\delta}}d\Leb_{\partial Q}=o\big(\Leb_{\partial Q}\big\{B_{r}(q)\big\}\big).
\end{align*}

Now we choose $r_0=1/n, \delta=1/k$. Then \begin{align*}
    \Leb_{\partial Q} (\bigcap_{n}T_{1/n,1/k})=0,\quad \Leb_{\partial Q} (\bigcup_{k}\bigcap_{n}T_{1/n,1/k})=0.
\end{align*}

Choose $\pi_{\partial }\pmb{m}\in \bigcap_{n,k} Q_{1/n,1/k}\bigcap (\bigcap_{k}\bigcup_{n}T^c_{1/n,1/k})$, which has a full measure. Then for any $k \gg 1$ there is $n_{k,\pi_{\partial} \pmb{m}} \in \mathbb{N}$ such that for any $n \ge n_{k,\pi_{\partial} \pmb{m}}$, $\pi_{\partial}\pmb{m}\notin T_{1/n, 1/k}$. (Here we used that $T_{1/t,1/k} \supseteq T_{1/(1+t),1/k} $ for any $t\ge 1$). 

For any $k \gg 1$, and any sufficiently small $r \in (0,1)$, there is $n \ge n_{k,\pi_{\partial} \pmb{m}}$ such that $r \in [\frac{1}{n+1}, \frac{1}{n})$. We choose $r_0=n^{-1}$, and continue the estimates of (\ref{2}) as  \begin{align*}
    &\precsim\mu_{\mathcal{M}}\{(q,v)\in \mathcal{M}:q\in B_r(\pi_{\partial}\pmb{m}), v\in U_{q, 1/n}\}\\
    &\precsim\int_{B_r(\pi_{\partial}\pmb{m})}\mu_{q}(U_{q,1/n})d\Leb_{\partial Q}\\
    &\precsim\int_{B_r(\pi_{\partial}\pmb{m})\bigcap T_{1/n,1/k}}\mu_{q}(U_{q,1/n})d\Leb_{\partial Q}+\int_{B_r(\pi_{\partial}\pmb{m})\bigcap T^c_{1/n,1/k}}\mu_{q}(U_{q,1/n})d\Leb_{\partial Q}\\
    &\precsim \int_{B_r(\pi_{\partial}\pmb{m})}\mathbbm{1}_{ T_{1/n,1/k}}d\Leb_{\partial Q}+\int_{B_r(\pi_{\partial}\pmb{m})\bigcap T^c_{1/n,1/k}}k^{-1} d\Leb_{\partial Q}\\
    &= o\big(\Leb_{\partial Q}\big\{B_{r}(\pi_{\partial}\pmb{m})\big\}\big)+ O\big(\Leb_{\partial Q}\big\{B_{r}(\pi_{\partial} \pmb{m})\big\}\big)k^{-1},
\end{align*}which means that \begin{align*}
    \lim_{r\to 0}\frac{\mu_X\Big\{\pi_X\pi^{-1}S_r(\pmb{m})\bigcap {\bigcup_{1\le j\le pN_{e,m}}g^{-j}\pi_X\pi^{-1}S_r(\pmb{m})}\Big\}}{\Leb_{\partial Q}\big\{B_{r}(\pi_{\partial} \pmb{m})\big\}}=O(1/k),
\end{align*}does not depend on $n$ anymore. Let $k \to \infty$.  
Then for any $q\in \bigcap_{n,k} Q_{1/n,1/k}\bigcap (\bigcap_{k}\bigcup_{n}T^c_{1/n,1/k})$,\begin{align*}
    \lim_{r\to 0}\frac{\mu_X\Big\{\pi_X\pi^{-1}S_r(\pmb{m})\bigcap {\bigcup_{1\le j\le pN_{e,m}}g^{-j}\pi_X\pi^{-1}S_r(\pmb{m})}\Big\}}{\Leb_{\partial Q}\big\{B_{r}(q)\big\}}=0.\end{align*}

Finally, because $B_r(q)\times [-\pi/2, \pi/2]$ is a quasi-section, we have that if $r>0$ is small enough, then $\mu_{X}\big(\pi_X\pi^{-1}S_r(\pmb{m})\big)\approx \mu_{\mathcal{M}}(B_r(q)\times [-\pi/2, \pi/2])\approx r\approx \Leb_{\partial Q}\big\{B_{r}(q)\big\}$. It concludes a proof of this lemma.
\end{proof}

\section{Conclusion of a proof of Theorem \ref{thm}}\label{section8}

Recall that (\ref{shortreturn}) is already proved. Then it follows from Lemma \ref{poissonappro} that, what is required, is to estimate (\ref{correlationinpoissonlaw}). In this section several technical lemmas are dedicated to a proof of (\ref{correlationinpoissonlaw}), followed by a proof of Theorem \ref{thm} at the end of this section.

From (\ref{correlationinpoissonlaw}) and throughout this section, $n\approx r^{-1}, p:=n^{1-\epsilon}$, and $\epsilon\in (0,1)$ is sufficiently small number in Lemma \ref{poissonappro}. Notice that, in Lemma \ref{poissonappro}, $n$ also depends on any fixed bounded intervals $J_1, J_2 \cdots, J_{k'}$ in $\mathbb{R}^+\bigcup \{0\}$. Since they are fixed, from now on we drop them and write $n \approx r^{-1}$ only. Now consider\begin{align*}
    \sum_{0 \le l \le n-p}\sup_{\pmb{a}\ge 1}\sup_{h\in[0,1]} \Big| \mathbb{E}\Big[\mathbbm{1}_{\pmb{X}_0=\pmb{a}}  h(\pmb{X}_p,\cdots, \pmb{X}_{n-l})\Big]-\mathbb{E}\mathbbm{1}_{\pmb{X}_0=\pmb{a}}  \mathbb{E}h(\pmb{X}_p,\cdots, \pmb{X}_{n-l})\Big|.
\end{align*}

Since there are finitely many $\pmb{a}\ge 1$, we just need to estimate\begin{align*}
    n\cdot \sup_{h\in[0,1]} \Big| \mathbb{E}\Big[\mathbbm{1}_{\pmb{X}_0= \pmb{a}}  h(\pmb{X}_p,\cdots, \pmb{X}_{n-l})\Big]-\mathbb{E}\mathbbm{1}_{\pmb{X}_0= \pmb{a}}  \mathbb{E}h(\pmb{X}_p,\cdots, \pmb{X}_{n-l})\Big|
\end{align*}

To simplify notations we set throughout this section \begin{gather*}
    T=F^{R_m},\quad U=F_{e,m}^{N_{e,m}},\quad  H=\pi_{\Delta_m}\pi^{-1}S_r.
\end{gather*} 

\[ \begin{tikzcd}
\Delta'_{e,m} \arrow{r}{F_{e,m}^{N_{e,m}}} \arrow[swap]{d}{\pi_{e,m}} & \Delta'_{e,m} \arrow{d}{\pi_{e,m}} \arrow{r}{\text{inclusion}} & \Delta_{e,m}\\%
\Delta_m \arrow{r}{(F^{R_m})^{N_{e,m}}}& \Delta_m
\end{tikzcd}, \quad \begin{tikzcd}
\Delta'_{e,m} \arrow{r}{U} \arrow[swap]{d}{\pi_{e,m}} & \Delta'_{e,m} \arrow{d}{\pi_{e,m}} \arrow{r}{\text{inclusion}} & \Delta_{e,m}\\%
\Delta_m \arrow{r}{T^{N_{e,m}}}& \Delta_m
\end{tikzcd}.
\]

Let an integer $k > m+1$, which will be determined later. For any $i=0,1,\cdots N_{e,m}-1$,\begin{gather*}
    B_i=U^{-k}\pi_{e,m}^{-1}T^{-i}H, \quad \pmb{B}=(\mathbbm{1}_{B_0},\cdots, \mathbbm{1}_{B_{N_{e,m}-1}}).
\end{gather*}

We lift now the dynamical system $(\Delta_m, (F^{R_m})^{N_{e,m}}, \mu_{\Delta_m})$ to the mixing hyperbolic Young tower $(\Delta'_{e,m}, F_{e,m}^{N_{e,m}}, \mu_{\Delta'_{e,m}})$, as is shown in the next lemma.
\begin{lemma}\label{3}
\begin{align*}
    &n\cdot \sup_{h\in[0,1]} \Big| \mathbb{E}\Big[\mathbbm{1}_{\pmb{X}_0=\pmb{a}}  h(\pmb{X}_p,\cdots, \pmb{X}_{n-l})\Big]-\mathbb{E}\mathbbm{1}_{\pmb{X}_0=\pmb{a}}  \mathbb{E}h(\pmb{X}_p,\cdots, \pmb{X}_{n-l})\Big|\\
   &=n\cdot \sup_{h\in[0,1]} \Big| \int \Big[\mathbbm{1}_{\pmb{B}=\pmb{a}}  h(\pmb{B}\circ U^p,\cdots, \pmb{B}\circ U^{n-l})\Big]d\mu_{\Delta_{e,m}'}\\
   & \quad -\int \mathbbm{1}_{\pmb{B}=\pmb{a}}d\mu_{\Delta_{e,m}'}  \int h(\pmb{B}\circ U^p,\cdots, \pmb{B}\circ U^{n-l}) d\mu_{\Delta_{e,m}'}\Big|.
\end{align*}
\end{lemma}
\begin{proof} Note that $\pmb{X}_0=(\mathbbm{1}_{H}, \mathbbm{1}_{H}\circ T \cdots, \mathbbm{1}_{H} \circ T^{N_{e,m}-1})$, $\pmb{X}_i=\pmb{X}_0 \circ T^{N_{e,m}i}$, $\pmb{B}=\pmb{X}_0\circ \pi_{e,m}\circ U^k$. We have\begin{align*}
    &n\cdot \sup_{h\in[0,1]} \Big| \mathbb{E}\Big[\mathbbm{1}_{\pmb{X}_0=\pmb{a}}  h(\pmb{X}_p,\cdots, \pmb{X}_{n-l})\Big]-\mathbb{E}\mathbbm{1}_{\pmb{X}_0=\pmb{a}}  \mathbb{E}h(\pmb{X}_p,\cdots, \pmb{X}_{n-l})\Big|\\
    &=n\cdot \sup_{h\in[0,1]} \Big| \mathbb{E}\Big[\mathbbm{1}_{\pmb{X}_0=\pmb{a}}  h(\pmb{X}_0 \circ T^{N_{e,m}p},\cdots, \pmb{X}_0 \circ T^{N_{e,m}(n-l)})\Big]\\
    &\quad -\mathbb{E}\mathbbm{1}_{\pmb{X}_0=\pmb{a}}  \mathbb{E}h(\pmb{X}_0 \circ T^{N_{e,m}p},\cdots, \pmb{X}_0 \circ T^{N_{e,m}(n-l)})\Big|\\
    &=n\cdot \sup_{h\in[0,1]} \Big| \int \Big[\mathbbm{1}_{\pmb{X}_0\circ \pi_{e,m}=\pmb{a}}  h(\pmb{X}_0  \circ \pi_{e,m}\circ U^{p},\cdots, \pmb{X}_0 \circ \pi_{e,m}\circ U^{n-l})\Big]d\mu_{\Delta_{e,m}'}\\
    &\quad -\int \mathbbm{1}_{\pmb{X}_0\circ \pi_{e,m}=\pmb{a}} d\mu_{\Delta_{e,m}'} \int h(\pmb{X}_0  \circ \pi_{e,m}\circ U^{p},\cdots, \pmb{X}_0 \circ \pi_{e,m}\circ U^{n-l})d\mu_{\Delta_{e,m}'}\Big|\\
    &=n\cdot \sup_{h\in[0,1]} \Big| \int \Big[\mathbbm{1}_{\pmb{B}=\pmb{a}}  h(\pmb{B}\circ U^p,\cdots, \pmb{B}\circ U^{n-l})\Big]d\mu_{\Delta_{e,m}'}\\
    & \quad -\int \mathbbm{1}_{\pmb{B}=\pmb{a}}d\mu_{\Delta_{e,m}'}  \int h(\pmb{B}\circ U^p,\cdots, \pmb{B}\circ U^{n-l}) d\mu_{\Delta_{e,m}'}\Big|,
\end{align*}where the last equality holds because $U^k_* \mu_{\Delta'_{e,m}}=\mu_{\Delta'_{e,m}}$.
\end{proof}

Now we will cover $\{\pmb{B}=\pmb{a}\}$ and $B_i$ by elements $Q\in \mathcal{Q}_{2k}^m$ (here $k > m+1$ will be determined later). Define \begin{gather*}
    \overline{B_i}:=\bigcup_{Q \bigcap B_i \neq \emptyset}Q,\quad  \partial B_i:=\bigcup_{Q \bigcap \overline{B_i}\setminus B_i \neq \emptyset} Q,\\
    A:=\{\pmb{B}=\pmb{a}\}, \quad \overline{A}:=\bigcup_{Q \bigcap A\neq \emptyset}Q, \quad \partial A:=\bigcup_{Q \bigcap \overline{A}\setminus A \neq \emptyset}Q,\\
    \pmb{\overline{B}}:=(\mathbbm{1}_{\overline{B_0}}, \cdots, \mathbbm{1}_{\overline{B_{N_{e,m}-1}}}), \quad \pmb{\partial{B}}:=(\mathbbm{1}_{\partial{B_0}}, \cdots, \mathbbm{1}_{\partial{B_{N_{e,m}-1}}}).
\end{gather*}

Clearly, the following properties hold \begin{gather*}
    \overline{B_i}\setminus B_i\subseteq \partial B_{i}, \quad \overline{A}\setminus A \subseteq \partial A, \quad A \subseteq \overline{A}\subseteq \{\pmb{\overline{B}}\ge 1\}.
\end{gather*}

Furthermore, for any $Q\subseteq \partial A$, there are $x,y \in Q$ and $i\in [0,N_{e,m}-1]$ such that $x \in B_i$, but $y \notin B_i$ or $x \notin B_i$ but $y \in B_i$. It means that $Q \subseteq \partial B_i$. Therefore \[\overline{A}\setminus A \subseteq \partial A \subseteq \bigcup_{i \le N_{e,m}-1}\partial B_i=\{\pmb{\partial B}\ge 1\}.\]
\begin{lemma}\label{4} The term in Lemma \ref{3} can be further estimated as
\begin{align*}
    &\sup_{h \in [0,1]}\Big| \int \Big[\mathbbm{1}_{A}  h(\pmb{B}\circ U^p,\cdots, \pmb{B}\circ U^{n-l})\Big]d\mu_{\Delta_{e,m}'} -\mu_{\Delta_{e,m}'} (A)\int h(\pmb{B}\circ U^p,\cdots, \pmb{B}\circ U^{n-l}) d\mu_{\Delta_{e,m}'}\Big|\\
    &\precsim_m [1+n\mu_{\Delta_{e,m}'}(\pmb{\overline{B}}\ge 1)]\mu_{\Delta_{e,m}'} (\pmb{\partial{B}}\ge 1)+\beta_m^{p-2k} \mu_{\Delta_{e,m}'}(\pmb{\overline{B}}\ge 1),
\end{align*}where a constant in $``\precsim_m"$ does not depend on $A,\pmb{B}$, and $\beta_m$ is the one in Lemma \ref{decayofcorrelationofthinkerhyperbolicyoungtower}.
\end{lemma} 

\begin{proof} Using $\overline{A}\setminus A \subseteq \partial A \subseteq  \{\pmb{\partial B}\ge 1\}$ and $A \subseteq  \{\pmb{B}\ge 1\}$, we have
 \begin{align*}
     &\Big| \int \Big[\mathbbm{1}_{A}  h(\pmb{B}\circ U^p,\cdots, \pmb{B}\circ U^{n-l})\Big]d\mu_{\Delta_{e,m}'} -\mu_{\Delta_{e,m}'} (A)\int h(\pmb{B}\circ U^p,\cdots, \pmb{B}\circ U^{n-l}) d\mu_{\Delta_{e,m}'}\Big|\\
     &\le \Big| \int \mathbbm{1}_{A}  \Big[h(\pmb{B}\circ U^p,\cdots, \pmb{B}\circ U^{n-l})- h(\pmb{\overline{B}}\circ U^p,\cdots, \pmb{\overline{B}}\circ U^{n-l})\Big]d\mu_{\Delta_{e,m}'}\Big|\\
     &\quad +\Big|\int \mathbbm{1}_{A}  h(\pmb{\overline{B}}\circ U^p,\cdots, \pmb{\overline{B}}\circ U^{n-l})d\mu_{\Delta_{e,m}'}-\mu_{\Delta_{e,m}'} (A) \int h(\pmb{\overline{B}}\circ U^p,\cdots, \pmb{\overline{B}}\circ U^{n-l}) d\mu_{\Delta_{e,m}'}\Big|\\
     &\quad +\mu_{\Delta_{e,m}'} (A) \Big|\int h(\pmb{B}\circ U^p,\cdots, \pmb{B}\circ U^{n-l}) -h(\pmb{\overline{B}}\circ U^p,\cdots, \pmb{\overline{B}}\circ U^{n-l}) d\mu_{\Delta_{e,m}'}\Big|\\
     &\le \Big| \int \mathbbm{1}_{A}  \Big[h(\pmb{B}\circ U^p,\cdots, \pmb{B}\circ U^{n-l})- h(\pmb{\overline{B}}\circ U^p,\cdots, \pmb{\overline{B}}\circ U^{n-l})\Big]d\mu_{\Delta_{e,m}'}\Big|\\
     &\quad +\Big|\int \mathbbm{1}_{\overline{A}}  h(\pmb{\overline{B}}\circ U^p,\cdots, \pmb{\overline{B}}\circ U^{n-l})d\mu_{\Delta_{e,m}'}-\mu_{\Delta_{e,m}'}(\overline{A}) \int h(\pmb{\overline{B}}\circ U^p,\cdots, \pmb{\overline{B}}\circ U^{n-l})d\mu_{\Delta_{e,m}'}\Big|\\
     &\quad +\Big|\int (\mathbbm{1}_{A}-\mathbbm{1}_{\overline{A}} ) h(\pmb{\overline{B}}\circ U^p,\cdots, \pmb{\overline{B}}\circ U^{n-l})d\mu_{\Delta_{e,m}'}\Big|\\
     &\quad +\Big|\mu_{\Delta_{e,m}'} (\overline{A}) -\mu_{\Delta_{e,m}'} (A) \Big|\int h(\pmb{\overline{B}}\circ U^p,\cdots, \pmb{\overline{B}}\circ U^{n-l}) d\mu_{\Delta_{e,m}'}\\
     &\quad +\mu_{\Delta_{e,m}'}(A) \Big|\int h(\pmb{B}\circ U^p,\cdots, \pmb{B}\circ U^{n-l}) -h(\pmb{\overline{B}}\circ U^p,\cdots, \pmb{\overline{B}}\circ U^{n-l}) d\mu_{\Delta_{e,m}'}\Big|\\
     &\le \Big| \int \mathbbm{1}_{A}  \Big[h(\pmb{B}\circ U^p,\cdots, \pmb{B}\circ U^{n-l})- h(\pmb{\overline{B}}\circ U^p,\cdots, \pmb{\overline{B}}\circ U^{n-l})\Big]d\mu_{\Delta_{e,m}'}\Big|\\
     &\quad +\Big|\int \mathbbm{1}_{\overline{A}}  h(\pmb{\overline{B}}\circ U^p,\cdots, \pmb{\overline{B}}\circ U^{n-l})d\mu_{\Delta_{e,m}'}-\mu_{\Delta_{e,m}'}(\overline{A}) \int h(\pmb{\overline{B}}\circ U^p,\cdots, \pmb{\overline{B}}\circ U^{n-l})d\mu_{\Delta_{e,m}'}\Big|\\
     &\quad +2\mu_{\Delta_{e,m}'} (\partial A)+\mu_{\Delta_{e,m}'}(A) \Big|\int h(\pmb{B}\circ U^p,\cdots, \pmb{B}\circ U^{n-l}) -h(\pmb{\overline{B}}\circ U^p,\cdots, \pmb{\overline{B}}\circ U^{n-l}) d\mu_{\Delta_{e,m}'}\Big|\\
     &\le \Big| \int \mathbbm{1}_{A}  \Big[h(\pmb{B}\circ U^p,\cdots, \pmb{B}\circ U^{n-l})- h(\pmb{\overline{B}}\circ U^p,\cdots, \pmb{\overline{B}}\circ U^{n-l})\Big]d\mu_{\Delta_{e,m}'}\Big| +2\mu_{\Delta_{e,m}'} (\pmb{\partial B}\ge 1)\\
     &\quad +\Big|\int \mathbbm{1}_{\overline{A}}  h(\pmb{\overline{B}}\circ U^p,\cdots, \pmb{\overline{B}}\circ U^{n-l})d\mu_{\Delta_{e,m}'}-\mu_{\Delta_{e,m}'}(\overline{A}) \int h(\pmb{\overline{B}}\circ U^p,\cdots, \pmb{\overline{B}}\circ U^{n-l})d\mu_{\Delta_{e,m}'}\Big|\\
     &\quad +\mu_{\Delta_{e,m}'}(\pmb{B}\ge 1) \Big|\int h(\pmb{B}\circ U^p,\cdots, \pmb{B}\circ U^{n-l}) -h(\pmb{\overline{B}}\circ U^p,\cdots, \pmb{\overline{B}}\circ U^{n-l}) d\mu_{\Delta_{e,m}'}\Big|.
 \end{align*}
 
 \textbf{Claim}: $|h(\pmb{B}\circ U^p,\cdots, \pmb{B}\circ U^{n-l}) -h(\pmb{\overline{B}}\circ U^p,\cdots, \pmb{\overline{B}}\circ U^{n-l})|\le 2\mathbbm{1}_{\bigcup_{p\le i\le n-l} \pmb{\partial{B}}\circ U^{i}\ge 1}$.
 
 For any $x\notin \bigcup_{p\le i\le n-l} \{\pmb{\partial{B}}\circ U^{i}\ge 1\}$ we have that $\pmb{\partial{B}} \circ U^i(x)=0$ for any $i\in [p,n-l]$. Then $\pmb{\overline{B}}\circ U^i(x)=\pmb{B}\circ U^i(x)$ for all $i\in [p,n-l]$. Thus, the claim holds.
 
 Now, using Lemma \ref{decayofcorrelationofthinkerhyperbolicyoungtower}, and $A \subseteq \overline{A}\subseteq \{\pmb{\overline{B}}\ge 1\}$, we can continue the estimates as \begin{align*}
     &\precsim_m  \int \mathbbm{1}_{A}  \mathbbm{1}_{\bigcup_{p\le i\le n-l} \pmb{\partial B}\circ U^{i}\ge 1}d\mu_{\Delta_{e,m}'} +\mu_{\Delta_{e,m}'} (\pmb{\partial{B}}\ge 1)+\beta_m^{p-2k} \mu_{\Delta_{e,m}'}(\overline{A})\\
     &\quad +\mu_{\Delta_{e,m}'}(\pmb{B}\ge 1) \mu_{\Delta_{e,m}'}(\bigcup_{p\le i\le n-l} \pmb{\partial{B}}\circ U^{i}\ge 1)\\
     &\precsim_m  \int \mathbbm{1}_{\pmb{\overline{B}}\ge 1}  \mathbbm{1}_{\bigcup_{p\le i\le n-l} \pmb{\partial B}\circ U^{i}\ge 1}d\mu_{\Delta_{e,m}'} +\mu_{\Delta_{e,m}'} (\pmb{\partial{B}}\ge 1)+\beta_m^{p-2k} \mu_{\Delta_{e,m}'}(\pmb{\overline{B}}\ge 1)\\
     &\quad +\mu_{\Delta_{e,m}'}(\pmb{B}\ge 1) \mu_{\Delta_{e,m}'}(\bigcup_{p\le i\le n-l} \pmb{\partial{B}}\circ U^{i}\ge 1)\\
     &\precsim_m  \int \mathbbm{1}_{\pmb{\overline{B}}\ge 1}  \mathbbm{1}_{\bigcup_{p\le i\le n-l} \pmb{\partial B}\circ U^{i}\ge 1}d\mu_{\Delta_{e,m}'}-\int \mathbbm{1}_{\pmb{\overline{B}}\ge 1} d\mu_{\Delta_{e,m}'} \int \mathbbm{1}_{\bigcup_{p\le i\le n-l} \pmb{\partial{B}}\circ U^{i}\ge 1}d\mu_{\Delta_{e,m}'}\\
     &\quad +\mu_{\Delta_{e,m}'}(\pmb{\overline{B}}\ge 1) \mu_{\Delta_{e,m}'}(\bigcup_{p\le i\le n-l} \pmb{\partial{B}}\circ U^{i}\ge 1) +\mu_{\Delta_{e,m}'} (\pmb{\partial{B}}\ge 1)\\
     &\quad +\beta_m^{p-2k} \mu_{\Delta_{e,m}'}(\pmb{\overline{B}}\ge 1) +\mu_{\Delta_{e,m}'}(\pmb{B}\ge 1) \mu_{\Delta_{e,m}'}(\bigcup_{p\le i\le n-l} \pmb{\partial{B}}\circ U^{i}\ge 1)\\
     &\precsim_m \beta_m^{p-2k} \mu_{\Delta_{e,m}'}(\pmb{\overline{B}}\ge 1) +\mu_{\Delta_{e,m}'}(\pmb{\overline{B}}\ge 1) \mu_{\Delta_{e,m}'}(\bigcup_{p\le i\le n-l} \pmb{\partial{B}}\circ U^{i}\ge 1) \\
     &\quad +\mu_{\Delta_{e,m}'} (\pmb{\partial{B}}\ge 1)+\mu_{\Delta_{e,m}'}(\pmb{B}\ge 1) \mu_{\Delta_{e,m}'}(\bigcup_{p\le i\le n-l} \pmb{\partial{B}}\circ U^{i}\ge 1)\\
     &\precsim_m [1+n\mu_{\Delta_{e,m}'}(\pmb{\overline{B}}\ge 1)]\mu_{\Delta_{e,m}'} (\pmb{\partial{B}}\ge 1)+\beta_m^{p-2k} \mu_{\Delta_{e,m}'}(\pmb{\overline{B}}\ge 1),
 \end{align*}where the last ``$\precsim_m$" is due to $\{\pmb{B}\ge 1\} \subseteq \{\pmb{\overline{B}}\ge 1\}$. This concludes a proof of this lemma.
\end{proof}

In order to proceed with further estimates some preparations will be made. We need now to consider only $Q \in \mathcal{Q}^m_{2k}$, which is contained in $\partial B_i$ for any $0\le i\le N_{e,m}-1$. By Lemma \ref{semiconjugacyforthickeryoungtower}, and because $i+k\N_{e,m}< 2k N_{e,m}$, we have that $ T^i\pi_{e,m}U^kQ= \pi_{e,m}F_{e,m}^iU^kQ= \pi_{e,m}F_{e,m}^{i+kN_{e,m}}Q$ does not contain singularities of $\mathbb{S}$. Therefore, $\pi_X  T^i \pi_{e,m}U^kQ$ belongs to some $X_t$ with $R|_{X_t}=R_t'$, according to the Definition \ref{inducesystem} and (\ref{eachxiisaconstant}). It means that $T^i\pi_{e,m}U^kQ \subseteq X_t\times \mathbb{N}_0$.

By the definition of $\partial B_i$ and $B_i$, we know that $T^i \pi_{e,m}U^kQ \bigcap \pi_{\Delta_{m}}\pi^{-1}S_r \neq \emptyset$. Then there is the smallest constant $u_{t,Q}\in[0,R_t')$ only depending on $t, Q, m$ (we drop the symbol $m$ because we fix this large $m$ here, see Remark \ref{remarkbefore5}), such that $$F^{u_{t,Q}}T^i \pi_{e,m}U^kQ \bigcap \pi^{-1}S_r \neq \emptyset .$$ Therefore,  $F^{u_{t,Q}}$ pushes $T^i \pi_{e,m}U^kQ$ upward until it hits $(\pi^{-1}S_r) \bigcap (X_t \times \mathbb{N}_0)$.

For any fixed $0\le i\le N_{e,m}-1$, we collect all $Q_{t,j}\subseteq \partial B_i$ such that\[Q_{t,j}\in \mathcal{Q}_{2k}^m,\quad T^i \pi_{e,m}U^k (\bigcup_j Q_{t,j}) \subseteq X_t \times \mathbb{N}_0.\]

Define $u_t: T^i \pi_{e,m}U^k (\bigcup_j Q_{t,j}) \to \mathbb{N}_0 $ by 
\[u_t(x):=\min\big\{u_{t,Q_{t,j}}: x \in T^i \pi_{e,m}U^k Q_{t,j} \text{ for some } Q_{t,j}\subseteq \partial B_i\big\}.\]

Here $u_t$ depends on $i, k$. Since we temporarily fix $i, k$, we can drop the symbols $i,k$ to simplify notations. Clearly, $F^{u_t}: T^i \pi_{e,m}U^k (\bigcup_jQ_{t,j})\to \Delta \bigcap (X_t \times \mathbb{N}_0)$ is injective.

\begin{lemma}\label{sizeboundary}
 For any $x \in  T^i \pi_{e,m}U^k (\bigcup_jQ_{t,j})$, according to the definition of $u_t(x)$, there is a $ T^i \pi_{e,m}U^k Q_{t,j}$ containing $x$, such that $F^{u_t(x)}T^i \pi_{e,m}U^kQ_{t,j} \bigcap \pi^{-1}S_r \neq \emptyset$. Then, there exists a constant  $C_{\alpha, m}>0$, which does not depend on $x, Q_{t,j},i,t,k$, such that  \[\diam (\pi \circ F^{u_t(x)} \circ T^i\circ \pi_{e,m}\circ U^kQ_{t,j}) \le  C_{\alpha,m} \beta^{\frac{N_{e,m}k\alpha}{2m}}.\]
\end{lemma}
\begin{proof} For any $\hat{\gamma}^s\subseteq Q_{t,j}$ (here $\hat{\gamma}^s=\gamma^s \times \{l'\}\times \{l\}$ for some $\gamma^s \subseteq \Lambda$ and  $l',l\ge 0$), suppose that $\pi_{\mathbb{N}} \big( \pi_{e,m} \circ F_{e,m}^{i+N_{e,m}k}(\hat{\gamma}^s)\big)=w\le m$, $w'\in [0,m]$ is the first non-negative number such that $\pi_{e,m} \circ F_{e,m}^{w'}(\hat{\gamma}^s) \subseteq X$, and the disks $ \pi_{e,m} \circ F_{e,m}^{w'}(\hat{\gamma}^s), \cdots , \pi_{e,m} \circ F_{e,m}^{i+N_{e,m}k}(\hat{\gamma}^s)$ visit $X$ exactly $q'$ times. Then 
\[q'+1 \ge \frac{1+i+N_{e,m}k}{m+1}\ge N_{e,m}k/(2m), \quad g^{q'-1} \circ \pi_{e,m} \circ F_{e,m}^{w'}(\hat{\gamma}^s)=\pi_{e,m} \circ F_{e,m}^{i+N_{e,m}k-w}(\hat{\gamma}^s) \subseteq X.\]

 Therefore, $\pi_{e,m} \circ F_{e,m}^{i+N_{e,m}k-w}(\hat{\gamma}^s)$ is a smooth disk in $X$, and $ \pi_{e,m} \circ F_{e,m}^{j'''}(\hat{\gamma}^s) \subseteq X^c$ for any $N_{e,m}k+i-w<j'''\le N_{e,m}k+i$, and $F^{j'''} \pi_{e,m} \circ F_{e,m}^{i+N_{e,m}k}(\hat{\gamma}^s)\subseteq X^c$ for any $0< j'''\le u_t(x)$. Then, by Assumption \ref{assumption} and by Definition \ref{cmz},  there is $C_{\alpha}'>0$ such that \begin{align*}
    \diam \big(\pi \circ F^{u_t(x)}\circ \pi_{e,m} \circ F_{e,m}^{i+N_{e,m}k}(\hat{\gamma}^s)\big)&\le C \diam \big( \pi_{e,m} \circ F_{e,m}^{i+N_{e,m}k-w}(\hat{\gamma}^s)\big)^{\alpha}\\
    &\le  C^{1+\alpha} \beta^{\alpha (q'-1)} \diam \big( \pi_{e,m} \circ F_{e,m}^{w'}(\hat{\gamma}^s)\big)^{\alpha} \le C_{\alpha}' \beta^{\frac{\alpha N_{e,m}k}{2m}}.
\end{align*} 

On the other hand, for any $\hat{\gamma}^u\subseteq Q_{t,j}$ (here $\hat{\gamma}^u=\gamma^u \times \{l'\}\times \{l\}$ for some $\gamma^u \subseteq \Lambda$ and same $l',l\ge 0$), since $N_{e,m}2k-(i+N_{e,m}k)=N_{e,m}k-i>N_{e,m}(k-1)>N_{e,m}m$, then it follows that the  disks $\pi \circ \pi_{e,m} \circ F_{e,m}^{N_{e,m}k+i}(\hat{\gamma}^u), \pi \circ \pi_{e,m} \circ F_{e,m}^{N_{e,m}k+i+1}(\hat{\gamma}^u), \cdots, \pi \circ \pi_{e,m} \circ F_{e,m}^{N_{e,m}2k}(\hat{\gamma}^u)$ visit the base $X$ at least $N_{e,m}$ times. Therefore, $\pi \circ F^{u_t(x)}\circ \pi_{e,m} \circ F_{e,m}^{i+N_{e,m}k}(\hat{\gamma}^u)$ is not fully extended to $\pi \circ \pi_{e,m} \circ F_{e,m}^{N_{e,m}2k}(\hat{\gamma}^u)$.

Assume that $\pi_{\mathbb{N}}\big(\pi_{e,m}\circ F_{e,m}^{N_{e,m}k+i}(\hat{\gamma}^u)\big)=j'\le m$, $\pi_{\mathbb{N}}\big(\pi_{e,m}\circ F_{e,m}^{N_{e,m}2k}(\hat{\gamma}^u)\big)=j''\le m$, and the disks $\pi_{e,m}\circ F_{e,m}^{N_{e,m}k+i-j'}(\hat{\gamma}^u), \pi_{e,m}\circ F_{e,m}^{N_{e,m}k+i-j'+1}(\hat{\gamma}^u), \cdots, \pi_{e,m} \circ F_{e,m}^{N_{e,m}2k}(\hat{\gamma}^u)$ visit $X$ exactly $q$ times. Then $q\in [\frac{N_{e,m}k+j'-i+1}{m+1},N_{e,m}k+j'-i+1]$, and \begin{gather*}
     \pi_{e,m}\circ F_{e,m}^{N_{e,m}k+i-j'}(\hat{\gamma}^u)\subseteq  X,\quad \pi_{e,m} \circ F_{e,m}^{N_{e,m}k+i-j'}(\hat{\gamma}^u)=g^{-(q-1)}\pi_{e,m}\circ F_{e,m}^{N_{e,m}2k-j''}(\hat{\gamma}^u) \subseteq X.
\end{gather*} 

By Assumption \ref{assumption}, and by Definition \ref{cmz} there is $C_{\alpha,m}'>0$ such that \begin{align*}
    \diam \big(\pi \circ F^{u_t(x)}\circ \pi_{e,m} \circ F_{e,m}^{N_{e,m}k+i}(\hat{\gamma}^u)\big)&\le C \diam \big(\pi \circ \pi_{e,m} \circ F_{e,m}^{N_{e,m}k+i-j'}(\hat{\gamma}^u) \big)^{\alpha}\\
    &\le C^{1+\alpha}\beta^{\alpha (q-1)}\diam \big(\pi \circ \pi_{e,m} \circ F_{e,m}^{2N_{e,m}k-j''}(\hat{\gamma}^u)\big)^{\alpha}\\
    &\le C^{1+\alpha}\beta^{\alpha \frac{N_{e,m}k+j'-i}{m+1}-\alpha}\diam \big(\pi \circ \pi_{e,m} \circ F_{e,m}^{2N_{e,m}k-j''}(\hat{\gamma}^u)\big)^{\alpha}\\
    &\le C^{1+\alpha}\beta^{ \frac{\alpha N_{e,m}(k-1)}{m+1}-\alpha}\diam \big(\pi \circ \pi_{e,m} \circ F_{e,m}^{2N_{e,m}k-j''}(\hat{\gamma}^u)\big)^{\alpha}\\
    &\le C_{\alpha,m}' \beta^{\frac{N_{e,m}k\alpha}{2m}}.
\end{align*}

Similar to the argument in the proof of Lemma \ref{decayofcorrelationofthinkerhyperbolicyoungtower}, we have that there is a constant $C_{\alpha,m}>0$ such that \begin{align*}
    \diam \left(\pi \circ F^{u_t(x)} \circ T^i \circ \pi_{e,m} \circ U^{k}Q_{t,j}\right)= \diam \big(\pi \circ F^{u_t(x)} \circ \pi_{e,m} \circ F_{e,m}^{N_{e,m}k+i}Q_{t,j}\big) \le  C_{\alpha,m} \beta^{\frac{N_{e,m}k\alpha}{2m}},
\end{align*} which concludes a proof of this lemma.
\end{proof}
\begin{lemma}\label{measureofboundary}
For any $0\le i\le N_{e,m}-1$ we have $\mu_{\Delta_{e,m}'}(\partial{B_i})\le  C'_{\alpha, m}\beta^{\frac{N_{e,m}k\alpha }{2m}}$, where a constant $C'_{\alpha, m}$  does not depend on $i, k$.
\end{lemma}
\begin{proof} Consider any $x \in  T^i \pi_{e,m}U^k (\bigcup_jQ_{t,j}) \in \mathcal{Q}^m_{2k}$, where $Q_{t,j}\subseteq \partial B_i$. According to the definition of $u_t(x)$, there is a $ T^i \pi_{e,m}U^k Q_{t,j} $ containing $x$, such that $F^{u_t(x)}T^i \pi_{e,m}U^kQ_{t,j} \bigcap \pi^{-1}S_r \neq \emptyset$.

\textbf{Claim}: $F^{u_t(x)}T^i \pi_{e,m}U^kQ_{t,j} \bigcap \pi^{-1}S_r \neq \emptyset \implies$ \begin{align*}
    \pi F^{u_t(x)}T^i \pi_{e,m}U^kQ_{t,j}\subseteq &\big[B_{r+C_{\alpha, m}\beta^{\frac{N_{e,m}k\alpha}{2m}}}(q)\big\backslash B_{r-C_{\alpha, m}\beta^{\frac{N_{e,m}k\alpha}{2m}}}(q)\big] \times [-\pi/2, \pi/2]\\
    & \quad \bigcup N_{C_{\alpha, m}\beta^{\frac{N_{e,m}k\alpha}{2m}}}\Big[\partial \big(B_r(q)\times [-\pi/2, \pi/2]\setminus S_r\big)\Big],
\end{align*}where $ N_{C_{\alpha, m}\beta^{\frac{N_{e,m}k\alpha}{2m}}}\Big[\partial \big(B_r(q)\times [-\pi/2, \pi/2]\setminus S_r\big)\Big]$ is a $C_{\alpha, m}\beta^{\frac{N_{e,m}k\alpha}{2m}}-$neighborhood of the boundary of $\big(B_{r}(q)\times [-\pi/2, \pi/2]\big)\setminus S_r$, and $C_{\alpha, m}$ is the one in Lemma \ref{sizeboundary}.

By the definition of $\partial B_i$ there is $y\in Q_{t,j}$ such that $T^i\pi_{e,m}U^ky \notin H$. Then $F^{u_t(x)} T^i\pi_{e,m}U^ky \notin \pi^{-1}S_r$ (if it is not the case, then $ T^i\pi_{e,m}U^ky=\pi_{\Delta_m}F^{u_t(x)} T^i\pi_{e,m}U^ky\in \pi_{\Delta_m}\pi^{-1}S_r=H$). Together with  $F^{u_t(x)}T^i \pi_{e,m}U^kQ_{t,j} \bigcap \pi^{-1}S_r \neq \emptyset$, by Lemma \ref{sizeboundary} and Remark \ref{remarkonsection}, we can conclude that $\pi F^{u_t(x)} T^i\pi_{e,m}U^kQ_{t,j}$ is contained in a $C_{\alpha, m}\beta^{\frac{N_{e,m}k\alpha}{2m}}$-neighborhood of the boundaries of $B_{r}(q)\times [-\pi/2, \pi/2]$ and $\big(B_{r}(q)\times [-\pi/2, \pi/2]\big)\setminus S_r$. Thus, this claim holds.

If we denote $\big[B_{r+C_{\alpha, m}\beta^{\frac{N_{e,m}k\alpha}{2m}}}(q)\big\backslash B_{r-C_{\alpha, m}\beta^{\frac{N_{e,m}k\alpha}{2m}}}(q)\big] \times [-\pi/2, \pi/2]\bigcup N_{C_{\alpha, m}\beta^{\frac{N_{e,m}k\alpha}{2m}}}\Big[\partial \big(B_r(q)\times [-\pi/2, \pi/2]\setminus S_r\big)\Big]$ by $\mathcal{N}_{r,q}$, then from this claim we have 

\textbf{Claim}: $\pi F^{u_t}T^i\pi_{e,m}U^k\bigcup_jQ_{t,j} \subseteq \mathcal{N}_{r,q}$ and $\mu_{\Delta }(T^i\pi_{e,m}U^k \bigcup_j Q_{t,j}) \le \mu_{\Delta }(F^{u_t}T^i\pi_{e,m}U^k\bigcup_j Q_{t,j})$. 

For any $x\in T^i \pi_{e,m}U^k\bigcup_jQ_{t,j}$, $\pi F^{u_t(x)}x \in \mathcal{N}_{r,q}$, so $\pi F^{u_t}T^i\pi_{e,m}U^k\bigcup_jQ_{t,j} \subseteq \mathcal{N}_{r,q}$. The second relation is due to the injectivity of $F^{u_t}$ and the measure of $\Delta$. So, the claim holds.

On the other hand, from $\frac{d\mu_{\mathcal{M}}}{d\Leb_{\mathcal{M}}}\in L^{\infty}$ we have $\mu_{\mathcal{M}}(\mathcal{N}_{r,q})\le C'_{\alpha,m}\beta^{\frac{N_{e,m}k\alpha }{2m}}$ for some $C_{\alpha, m}'>0$. By using the claims above, and the relations $U_*\mu_{\Delta_{e,m}'}=\mu_{\Delta_{e,m}'}, T_{*}\mu_{\Delta_m}=\mu_{\Delta_m}$, we get that
\begin{align*}
     \mu_{\Delta_{e,m}'}(\partial{B_i})&\le \mu_{\Delta_{e,m}'}(U^{-k}\pi_{e,m}^{-1}\pi_{e,m}U^k\bigcup Q)\\
     &= \mu_{\Delta_{m}}(\pi_{e,m}U^k\bigcup Q)\\
     &\le \mu_{\Delta_{m}}(T^{-i}T^{i}\pi_{e,m}U^k\bigcup Q)\\
     &\precsim_m \mu_{\Delta}(T^i\pi_{e,m}U^k\bigcup_t \bigcup_j Q_{t,j})
     \end{align*}

    Further, by making use of $\Delta=\Delta\bigcap \bigcup_t (X_t \times \mathbb{N}_0)$, $T^i\pi_{e,m}U^k \bigcup_j Q_{t,j}\subseteq X_t \times \mathbb{N}_0$, and the claims above, we can continue the estimates above as
     \begin{align*}
     & \precsim_m \sum_{t}\mu_{\Delta \bigcap (X_t\times \mathbb{N}_0) }(T^i\pi_{e,m}U^k \bigcup_j Q_{t,j})\\
     & \precsim_m \sum_{t}\mu_{\Delta \bigcap (X_t\times \mathbb{N}_0) }(F^{u_t}T^i\pi_{e,m}U^k\bigcup_j Q_{t,j})\\
     & \precsim_m \sum_{t}\mu_{\Delta \bigcap (X_t\times \mathbb{N}_0) } \big\{ \pi^{-1}\big(\mathcal{N}_{r,q}\big)\big\}\\
      & \precsim_m \mu_{\Delta } \big\{ \pi^{-1}\big(\mathcal{N}_{r,q}\big)\big\} \precsim_m \mu_{\mathcal{M}} (\mathcal{N}_{r,q})\le C'_{\alpha, m}\beta^{\frac{N_{e,m}k\alpha }{2m}},
     \end{align*} where the last line is due to $\pi_* \mu_{\Delta}=\mu_{\mathcal{M}}$. Hence, a proof of this lemma is concluded.
\end{proof}

Finally, we can prove now the main result of this section.
\begin{lemma} If $r\to 0$, then
\begin{align*}
    n\cdot \sup_{h\in[0,1]} \Big| \mathbb{E}\Big[\mathbbm{1}_{\pmb{X}_0=\pmb{a}}  h(\pmb{X}_p,\cdots, \pmb{X}_{n-l})\Big]-\mathbb{E}\mathbbm{1}_{\pmb{X}_0=\pmb{a}}  \mathbb{E}h(\pmb{X}_p,\cdots, \pmb{X}_{n-l})\Big|\to 0.
\end{align*}

\end{lemma}
\begin{proof}At first, we estimate $\mu_{\Delta_{e,m}'}(\pmb{\overline{B}}\ge 1)$ and $\mu_{\Delta_{e,m}'}(\pmb{\partial{B}}\ge 1)$. By making use of Lemma \ref{measureofboundary}, $U_*\mu_{\Delta_{e,m}'}=\mu_{\Delta_{e,m}'}$ and of the section $S_r\subseteq B_r(q)\times [-\pi/2, \pi/2]$, we obtain
\begin{align*}
   \mu_{\Delta_{e,m}'}(\pmb{\overline{B}}\ge 1)&\le \sum_{0\le i\le N_{e,m}-1} \mu_{\Delta_{e,m}'}(\overline{B_{i}})\\
   &\le \sum_{0\le i\le N_{e,m}-1} \mu_{\Delta_{e,m}'}(B_{i})+\sum_{0\le i\le N_{e,m}-1} \mu_{\Delta_{e,m}'}(\partial{B_{i}})\\
   &=\sum_{0\le i\le N_{e,m}-1} \mu_{\Delta_{e,m}'}(U^{-k}\pi_{e,m}^{-1}T^{-i}H)+\sum_{0\le i\le N_{e,m}-1} \mu_{\Delta_{e,m}'}(\partial{B_{i}})\\
   &=\sum_{0\le i\le N_{e,m}-1} \mu_{\Delta_m}(H)+\sum_{0\le i\le N_{e,m}-1} \mu_{\Delta_{e,m}'}(\partial{B_{i}})\\
   &=N_{e,m}\mu_{\Delta_m}(H)+\sum_{0\le i\le N_{e,m}-1} \mu_{\Delta_{e,m}'}(\partial{B_{i}})\\
   &\precsim_{\alpha, m} \mu_{\Delta} (\pi^{-1}S_r)+ \beta^{\frac{N_{e,m}k\alpha }{2m}}\precsim_{\alpha,m} r+\beta^{\frac{N_{e,m}k\alpha }{2m}},\\
   \mu_{\Delta_{e,m}'}(\pmb{\partial{B}}\ge 1)&\le \sum_{0\le i\le N_{e,m}-1} \mu_{\Delta_{e,m}'}(\partial{B_i})\precsim_{\alpha,m} \beta^{\frac{N_{e,m}k\alpha }{2m}}.
\end{align*}

Recall that, in (\ref{fiberwisemeasure}), $\mu_{\mathcal{M}}(B_r(q)\times [-\pi/2, \pi/2])=\int_{B_r(q)} \mu_{q'}([-\pi/2, \pi/2])d\Leb_{\partial Q}$. Since $\frac{d\mu_{\mathcal{M}}}{d\Leb_{\mathcal{M}}}>0$ almost surely, then $\mu_{q'}([-\pi/2, \pi/2])>0$ $\Leb_{\partial Q}$-a.s. $q' \in \partial Q$. Therefore, by the Lebesgue differentiation theorem, \begin{align*}
    \lim_{r\to 0}\frac{\mu_{\mathcal{M}}(B_r(q)\times [-\pi/2, \pi/2])}{\Leb_{\partial Q} B_r(q)}=\lim_{r \to 0}\frac{\int_{B_r(q)} \mu_{q'}([-\pi/2, \pi/2])d\Leb_{\partial Q}}{{\Leb_{\partial Q} B_r(q)}}=\mu_q([-\pi/2, \pi/2])>0
\end{align*}holds for $\Leb_{\partial Q}$-a.e. $q \in \partial Q$. Hence, \begin{gather*}
    n^{-1}\approx \mu_{\Delta_m}(\pi_{\Delta_m}\pi^{-1}S_r)\approx_{m}\mu_{\mathcal{M}}(S_r)\approx_{m}\mu_{\mathcal{M}}\big(B_r(q) \times [-\pi/2, \pi/2]\big)\approx_{m,q} r,\\
    p=n^{1-\epsilon}\approx_{\epsilon} \mu_{\Delta_m}(\pi_{\Delta_m}\pi^{-1}S_r)^{-1+\epsilon}\approx_{m,\epsilon}\mu_{\mathcal{M}}(S_r)^{-1+\epsilon}\approx_{m,\epsilon}r^{-(1-\epsilon)}.
\end{gather*}

By using the estimates above, Lemmas \ref{3} and \ref{4}, and choosing $k=p/4$, we obtain that \begin{align*}
    &n\cdot \sup_{h\in[0,1]} \Big| \mathbb{E}\Big[\mathbbm{1}_{\pmb{X}_0=\pmb{a}}  h(\pmb{X}_p,\cdots, \pmb{X}_{n-l})\Big]-\mathbb{E}\mathbbm{1}_{\pmb{X}_0=\pmb{a}}  \mathbb{E}h(\pmb{X}_p,\cdots, \pmb{X}_{n-l})\Big|\\
    &\precsim_m n\cdot [1+n\mu_{\Delta_{e,m}'}(\pmb{\overline{B}}\ge 1)]\mu_{\Delta_{e,m}'} (\pmb{\partial{B}}\ge 1)+n\beta_m^{p-2k} \mu_{\Delta_{e,m}'}(\pmb{\overline{B}}\ge 1)\\
    & \precsim_{m,\epsilon,q}r^{-1}[1+r^{-1}(r+\beta^{\frac{N_{e,m}\alpha }{2m}\frac{r^{-(1-\epsilon)}}{4}})]\beta^{\frac{N_{e,m}\alpha }{2m}\frac{r^{-(1-\epsilon)}}{4}}+r^{-1}(r+\beta^{\frac{N_{e,m}\alpha }{2m}\frac{r^{-(1-\epsilon)}}{4}})\beta_m^{\frac{r^{-(1-\epsilon)}}{2}}
\end{align*}converges to zero, when $r\to 0$. It concludes a proof.
\end{proof}

\begin{proof}[Proof of Theorem \ref{thm}]
We finished the estimates of (\ref{shortreturn}) and (\ref{correlationinpoissonlaw}) in the Lemma \ref{poissonappro}. Thus, in view of the Remark \ref{remarkbefore5}, Theorem \ref{thm} holds.
\end{proof}
\section{Applications}\label{sectionapp}
\subsection{A practical scheme for applications of the obtained results to concrete systems} Here we present some criteria to verify the cone conditions in Assumption \ref{assumption} for two-dimensional billiards. The following notations will be used throughout this section: $x=(q, \phi) \in \mathcal{M}$, $\pi_{\partial_{\partial Q} Q}x=q$, $dx=(dq, d\phi)\in \mathcal{T}_x\mathcal{M}$, $\mathcal{K}=\mathcal{K}(q)$ is the curvature of the boundary at a point $q \in \partial Q$.  Let $f: \mathcal{M}\to \mathcal{M}$ be a billiard map which maps points of the phase space at reflection times to their images at the next reflection time. It preserves an invariant measure $d\mu_{\mathcal{M}}:=(2\Leb_{\partial Q} \partial Q)^{-1}\cos \phi d\phi dq$. Let $x_n=(q_n,\phi_n)=f^n(x_0)$ for all $n \in \mathbb{Z}$, where $x_0 \in \mathcal{M}$. A wave front is a smooth curve in $Q$ equipped with a continuous family of unit normal vectors. Denote by $B^{+}(x_n)$ (resp. $B^-(x_n)$) a curvature of a wave front right after (resp. before) the collision with the boundary $x_n\in \mathcal{M}$. We list now several basic formulas  (see e.g. \cite{CMbook}) for two-dimensional billiards \begin{gather}
    v:=d\phi/dq=B^-\cos \phi +\mathcal{K}=B^+\cos \phi -\mathcal{K},\nonumber\\
    1/B^{-}(x_{n+1})=\tau_{n}+1/B^{+}(x_{n}),\nonumber\\
    Df(x_n)=\frac{-1}{\cos \phi_{n+1}}\begin{bmatrix}
    \tau_n\mathcal{K}(q_n)+\cos \phi_n    & \tau_n \\
    \tau_n\mathcal{K}(q_n)\mathcal{K}(q_{n+1})+\mathcal{K}(q_n)\cos\phi_{n+1} + \mathcal{K}(q_{n+1})\cos \phi_n \quad & \tau_n \mathcal{K}(q_{n+1})+\cos \phi_{n+1}\label{derivativeofbilliardmap}
\end{bmatrix},\\
\frac{||dx_{n+1}||_p}{||dx_n||_p}=\frac{||Df(dx_n)||_p}{||dx_n||_p}=|1+\tau_nB^{+}(x_n)|\label{pesudoderivative},\\
||dx_n||=\frac{||dx_n||_p}{\cos \phi_n} \sqrt{1+(\frac{d\phi_n}{dq_n})^2}\label{rnderivative},
\end{gather}where $\tau_n$ is the length of the free path from $x_{n}$ to $x_{n+1}$, $||dx_n||_p:=\cos \phi_n |dq_n|$, $||dx_n||:=\sqrt{(d\phi_n)^2+(dq_n)^2}$.

In what follows we always assume that all boundary components of a billiard table are at least $C^3$.  We call boundary components with zero curvature flat components, dispersing components are convex inwards, and focusing components are convex outwards.

\begin{lemma}\label{practicalcone}
Suppose that regular components of the boundary $\partial Q$ are either flat, or dispersing, or focusing, and
\begin{itemize}
    
\item[] each focusing component $\Gamma_i$ is an arc of a circle, but not a full circle, and this circle does not intersect any other components of the boundary $\partial Q$. In what follows we will call this condition the simplest focusing chaos (SFC) condition.
\end{itemize}
Then there are cone fields $C^u,C^s \subseteq \mathcal{T}(\mathcal{M})$, which satisfy the conditions in Assumption \ref{assumption}.
\end{lemma}
\begin{proof}
We construct $C^u, C^s$ fiber-wisely as follows. For any $x=(q, \phi)$\begin{gather*}
    C^u_x:=\{(dq,d\phi)\in \mathcal{T}_x\mathcal{M}: \mathcal{K}\le d\phi/dq \le \infty\},\\ \interior{C^u_x}:=\{(dq,d\phi)\in \mathcal{T}_x\mathcal{M}\setminus \{0\}: \mathcal{K}< d\phi/dq < \infty\}
\end{gather*}for dispersing and flat boundary components, and \begin{gather*}
    C^u_x:=\{(dq,d\phi)\in \mathcal{T}_x\mathcal{M}: \mathcal{K}\le d\phi/dq \le 0\},\\ \interior{C^u_x}:=\{(dq,d\phi)\in \mathcal{T}_x\mathcal{M}\setminus \{0\}: \mathcal{K}< d\phi/dq < 0\}
\end{gather*}for focusing arcs.
\begin{gather*}
    C^s_x:=\{(dq,d\phi)\in \mathcal{T}_x\mathcal{M}: -\infty \le d\phi/dq \le -\mathcal{K}\},\\
    \interior{C^s_x}:=\{(dq,d\phi)\in \mathcal{T}_x\mathcal{M}\setminus \{0\}: -\infty < d\phi/dq < -\mathcal{K}\}
\end{gather*}for dispersing and flat boundary components, and \begin{gather*}
    C^s_x:=\{(dq,d\phi)\in \mathcal{T}_x\mathcal{M}: 0\le d\phi/dq \le -\mathcal{K}\}\\
    \interior{C^s_x}:=\{(dq,d\phi)\in \mathcal{T}_x\mathcal{M}\setminus \{0\}: 0< d\phi/dq < -\mathcal{K}\}
\end{gather*}for focusing arcs.

Define  $C^u:=\bigcup_{x\in \mathcal{M}}C^u_x$,  $\interior{C^u}:=\bigcup_{x\in \mathcal{M}}\interior{C^u_x}$, $C^s:=\bigcup_{x\in \mathcal{M}}C^s_x$ and $\interior{C^s}:=\bigcup_{x\in \mathcal{M}}\interior{C^s_x}$ in $\mathcal{T}\mathcal{M}$. It was shown in Theorem 8.9 of \cite{CMbook} that $Df(C^u_x)\subseteq C^u_{f(x)}$, $Df^{-1}(C^s_x)\subseteq C^s_{f^{-1}(x)}$. Clearly $\interior{C^u}\bigcap \interior{C^s}=\emptyset$, $\dim(\interior{C^u}\bigcap \interior{C^s})=0<1$.

Now let $x_0=(q_0, \phi_0)\in \{q\}\times [-\pi/2,\pi/2]$. Clearly, $\dim \big(\mathcal{T}(\{q\}\times [-\pi/2,\pi/2])\bigcap \interior{C^u}\big)=0<1$ and $\dim \big(\mathcal{T}(\{q\}\times [-\pi/2,\pi/2])\bigcap \interior{C^s}\big)=0<1$. Now we will prove that the following claims hold.

\textbf{Claim:} $(Df)\mathcal{T}(\{q\}\times [-\pi/2,\pi/2]) \subseteq C^u$, $(Df)\mathcal{T}(\{q\}\times [-\pi/2,\pi/2]) \subseteq \interior{C^u}$ if $\phi_1 \neq \pm \pi/2$. 

Suppose that $dx_0 \in \mathcal{T}(\{q\}\times [-\pi/2,\pi/2])$. Then $B^+(x_0)\cos \phi_0-\mathcal{K}(q_0)=d\phi_0/dq_0=\infty$, i.e.,  $B^+(x_0)=\infty$. It implies that \begin{align*}
    \frac{d\phi_1}{dq_1}=\mathcal{K}(q_1)+\frac{\cos \phi_1}{\tau_0+1/B^+(x_0)}=\mathcal{K}(q_1)+\frac{\cos \phi_1}{\tau_0}\in [\mathcal{K}(q_1), \infty).
\end{align*}

If $q_1$ belongs to a focusing arc, then by the SFC-condition $\tau_0 \ge -2\mathcal{K}(q_1)^{-1} \cos \phi_1 $, which implies that \begin{align*}
    \frac{d\phi_1}{dq_1}=\mathcal{K}(q_1)+\frac{\cos \phi_1}{\tau_0}\le \mathcal{K}(q_1)+\frac{\cos \phi_1}{-2\mathcal{K}(q_1)^{-1} \cos \phi_1}=\mathcal{K}(q_1)/2<0.
\end{align*}

Then $dx_1 \in C^u_{x_1}$ for any $q \in \partial Q$. Particularly, if 
$\phi_1 \neq \pm \pi/2$, then $\frac{d\phi_1}{dq_1}>\mathcal{K}(q_1)$. Thus $dx_1 \in \interior{C^u_{x_1}}$, and the claim holds.

\textbf{Claim:} $(Df)^{-1}\mathcal{T}(\{q\}\times [-\pi/2,\pi/2]) \subseteq C^s$, $(Df)^{-1}\mathcal{T}(\{q\}\times [-\pi/2,\pi/2]) \subseteq \interior{C^s}$ if $\phi_{-1} \neq \pm \pi/2$. 

Suppose that $dx_0 \in \mathcal{T}(\{q\}\times [-\pi/2,\pi/2])$. Then $B^-(x_0)\cos \phi_0+\mathcal{K}(q_0)=d\phi_0/dq_0=\infty$, i.e.,  $B^-(x_0)=\infty$. Therefore \begin{gather*}
    0=1/B^{-}(x_0)=\tau_{-1}+1/B^{+}(x_{-1}),\\
    \frac{d\phi_{-1}}{dq_{-1}}=-\mathcal{K}(q_{-1})+B^{+}(x_{-1})\cos \phi_{-1}=-\mathcal{K}(q_{-1})-\frac{\cos \phi_{-1}}{\tau_{-1}}\in (-\infty, -\mathcal{K}(q_{-1})].
\end{gather*}

If $q_{-1}$ belongs to a focusing arc, then by the SFC-condition we have $\tau_{-1} \ge -2\mathcal{K}(q_{-1})^{-1} \cos \phi_{-1} $, which implies that
\begin{align*}
    \frac{d\phi_{-1}}{dq_{-1}}=-\mathcal{K}(q_{-1})-\frac{\cos \phi_{-1}}{\tau_{-1}}\ge  -\mathcal{K}(q_{-1})-\frac{\cos \phi_{-1}}{-2\mathcal{K}(q_{-1})^{-1} \cos \phi_{-1}}=-\mathcal{K}(q_{-1})/2>0.
\end{align*}

Hence, $dx_{-1} \in C^s_{x_{-1}}$ for any $q \in \partial Q$. Particularly, if 
$\phi_{-1} \neq \pm \pi/2$, then $\frac{d\phi_{-1}}{dq_{-1}}<-\mathcal{K}(q_{-1})$. Thus $dx_{-1} \in \interior{C^s_{x_{-1}}}$, which proves the claim.

\textbf{Claim:} For the set $\Phi:=(\bigcup_{i\in \mathbb{Z}}f^{-i}\{\phi=\pm \pi/2\})^c\subseteq \mathcal{M}$ we have $\mu_{\mathcal{M}}(\Phi)=1$.

This claim follows from the facts that $\mu_{\mathcal{M}}$ is $f$-invariant, and $\mu\{\phi=\pm \pi/2\}=0$. 

\textbf{Claim:} For any $x_0\in \Phi$ we have $(Df)\interior{C^u_{x_0}} \subseteq \interior{C^u_{f(x_0)}}$ and $(Df)^{-1}\interior{C^s_{x_0}} \subseteq \interior{C^s_{f^{-1}(x_0)}}$.

In view of the involution property of a billiard map, we just need to show that $(Df)\interior{C^u_{x_0}} \subseteq \interior{C^u_{f(x_0)}}$. Let $dx_0 \in \interior{C^u_{x_0}}$. Then \begin{gather*}
    \frac{d\phi_0}{dq_0}=-\mathcal{K}(q_0)+ B^+(x_0)\cos \phi_0 \in (\mathcal{K}(q_0), \infty) \text{ if } \mathcal{K}(q_0)\ge 0,\\
    \frac{d\phi_0}{dq_0}=-\mathcal{K}(q_0)+ B^+(x_0)\cos \phi_0 \in (\mathcal{K}(q_0), 0) \text{ if }\mathcal{K}(q_0)< 0.
\end{gather*} 

In order to prove the relation $dx_1 \in \interior{C^u_{x_1}}$, we will show that \begin{align*}
    \frac{d\phi_1}{dq_1}=\mathcal{K}(q_1)+\frac{\cos \phi_1}{\tau_0+1/B^+(x_0)}\in (\mathcal{K}(q_1), \infty) \text{ if } \mathcal{K}(q_1)\ge 0,\\
    \frac{d\phi_1}{dq_1}=\mathcal{K}(q_1)+\frac{\cos \phi_1}{\tau_0+1/B^+(x_0)}\in (\mathcal{K}(q_1),0) \text{ if } \mathcal{K}(q_1)< 0,
\end{align*} for case by case, depending on the positions of $q_0, q_1$, where $x_0 \in \Phi$.

If $\mathcal{K}(q_0)\ge0$ and $\mathcal{K}(q_1)\ge 0$, then $B^{+}(x_0)>0$, $d\phi_1/dq_1 \in (0, \infty)$.

If now $\mathcal{K}(q_0)\ge 0$ and $\mathcal{K}(q_1)< 0$, then $B^{+}(x_0)>0$, $d\phi_1/dq_1>\mathcal{K}(q_1)$, and by the SFC-condition we get \begin{align*}
    \frac{d\phi_1}{dq_1}< \mathcal{K}(q_1)+\frac{\cos \phi_1}{\tau_0}\le \mathcal{K}(q_1)+\frac{\cos \phi_1}{-2\mathcal{K}(q_1)^{-1}\cos \phi_1}\le \mathcal{K}(q_1)/2<0.
\end{align*}

If $\mathcal{K}(q_0)<0$ and $\mathcal{K}(q_1)\ge 0$, then $B^{+}(x_0)\in \big(\frac{2\mathcal{K}(q_0)}{\cos \phi_0},\frac{\mathcal{K}(q_0)}{\cos \phi_0} \big)\subseteq (-\infty, 0)$, and by the SFC-condition \begin{gather*}
   \tau_0+1/B^{+}(x_0)>\tau_0+\mathcal{K}(q_0)^{-1}\cos \phi_0
   \ge -2\mathcal{K}(q_0)^{-1}\cos \phi_0+\mathcal{K}(q_0)^{-1}\cos \phi_0>0\\
   \implies \frac{d\phi_1}{dq_1}= \mathcal{K}(q_1)+\frac{\cos \phi_1}{\tau_0+1/B^{+}(x_0)} \in (\mathcal{K}(q_1), \infty).
\end{gather*} 

If $\mathcal{K}(q_0)<0$ and $\mathcal{K}(q_1)< 0$, then $B^{+}(x_0)\in \big(\frac{2\mathcal{K}(q_0)}{\cos \phi_0}, \frac{\mathcal{K}(q_0)}{\cos \phi_0}\big)$, $\tau_0 +1/B^{+}(x_0)>0$ and \begin{gather*}
    \frac{d\phi_1}{dq_1}= \mathcal{K}(q_1)+\frac{\cos \phi_1}{\tau_0+1/B^{+}(x_0)}>\mathcal{K}(q_1).
\end{gather*} 

According to the SFC-condition  (i.e., $\tau_0\ge -\mathcal{K}(q_0)^{-1}\cos\phi_0-\mathcal{K}(q_1)^{-1}\cos\phi_1$),
\begin{gather*}
    \frac{d\phi_1}{dq_1}< \mathcal{K}(q_1)+\frac{\cos \phi_1}{\tau_0+ \mathcal{K}(q_0)^{-1}\cos\phi_0}\le \mathcal{K}(q_1)+\frac{\cos \phi_1}{-\mathcal{K}(q_1)^{-1}\cos \phi_1}=0.
\end{gather*}

Therefore $dx_1 \in \interior{C^u_{x_1}}$, and the claim holds.

Combining all the claims above, we obtain that for any $x_0=(q, \phi_0)\in \Phi $, $n \ge 1$, \begin{gather*}
    (Df)^n\mathcal{T}(\{q\}\times [-\pi/2,\pi/2]) \subseteq \interior{C^u},\quad (Df)^{-n}\mathcal{T}(\{q\}\times [-\pi/2,\pi/2]) \subseteq \interior{C^s}.
\end{gather*}

\textbf{Claim:} For a.e. $q\in \partial Q$, we have \begin{gather*}
    (Df)^n\mathcal{T}(\{q\}\times [-\pi/2,\pi/2]) \subseteq \interior{C^u},\quad (Df)^{-n}\mathcal{T}(\{q\}\times [-\pi/2,\pi/2]) \subseteq \interior{C^s}.
\end{gather*}

If it is not the case, then there exists a subset $O\subseteq \partial Q$ with $\Leb_{\partial Q} O>0$, so that for any $q\in O$, \begin{gather*}
    (Df)^n\mathcal{T}(\{q\}\times [-\pi/2,\pi/2]) \nsubseteq \interior{C^u} \text{ or } (Df)^{-n}\mathcal{T}(\{q\}\times [-\pi/2,\pi/2]) \nsubseteq \interior{C^s}.
\end{gather*}

Then $\mu_{\mathcal{M}}(O \times [-\pi/2,\pi/2])>0$, $(O \times [-\pi/2,\pi/2])\bigcap \Phi\neq \emptyset$, and for any $x_0=(q,\phi_0)\in (O \times [-\pi/2,\pi/2])\bigcap \Phi $,\begin{gather*}
    (Df)^n\mathcal{T}(\{q\}\times [-\pi/2,\pi/2]) \nsubseteq \interior{C^u} \text{ or } (Df)^{-n}\mathcal{T}(\{q\}\times [-\pi/2,\pi/2]) \nsubseteq \interior{C^s}.
\end{gather*}

Hence, we came to a contradiction, and the claim holds, which concludes a proof of this lemma.\end{proof}

All billiard systems, which will be considered below, satisfy the conditions of Lemma \ref{practicalcone}. Therefore the condition on existence of the cone fields in Assumption \ref{assumption} holds automatically.

\subsection{Sinai and Diamond billiards}
Pictures of billiard tables of Sinai billiards and of diamond billiards  are presented in Figures \ref{F1} and \ref{F2}. Choose the first return time $R=1$. The Assumption \ref{assumption} in this case holds automatically. 

\begin{figure}[!htb]
   \begin{minipage}{0.6\textwidth}
     \includegraphics[width=.7\linewidth]{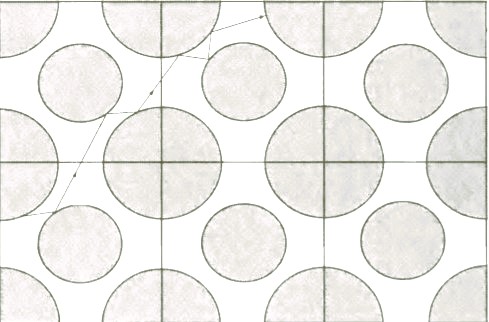}
     \caption{Sinai billiard}
     \label{F1}
   \end{minipage}\hfill
   \begin{minipage}{0.5\textwidth}     \includegraphics[width=.6\linewidth]{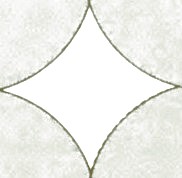}
     \caption{Diamond billiard}\label{F2}
   \end{minipage}
\end{figure}

\begin{corollary}
Theorem \ref{thm} holds for Sinai billiards with a bounded or unbounded horizon (see e.g. \cite{CMbook}), and for diamond billiards (see e.g. \cite{peneijm}).
\end{corollary}
\begin{remark}
In fact, if a billiard map of a two-dimensional billiard has a CMZ structure, i.e., $R=1$,  and if the boundary of a billiard table satisfies the conditions of Lemma \ref{practicalcone}, then Theorem \ref{thm} holds. Moreover, such billiards have exponentially mixing rates (or exponential decay of correlations), i.e., of order $O(\rho^n)$ for some $\rho\in (0,1)$.
\end{remark}

In the following subsections, we consider two-dimensional slowly mixing billiards, which were studied in \cite{CZcmp,CZnon}. The rates of mixing (decay of correlations) for these billiards are either of order $O(n^{-1})$, or $O(n^{-2})$.
\subsection{Squashes or Stadium-type billiards}
A billiard table $Q$ of a squash billiard is a convex domain bounded by two circular arcs and two straight (flat) segments tangent to the arcs at their common endpoints. A squash billiard is called a stadium if flat sides are parallel, see Figure \ref{F3}. Initially called squash billiards, they were later sometimes called ``skewed" stadia, drive-belt billiards, etc, see Figure \ref{F4}. Note that squashes contain a boundary arc, which is longer than a half circle. We will verify now the Assumption \ref{assumption} for this class of billiards.

\begin{figure}[!htb]
   \begin{minipage}{0.6\textwidth}
     \includegraphics[width=.6\linewidth]{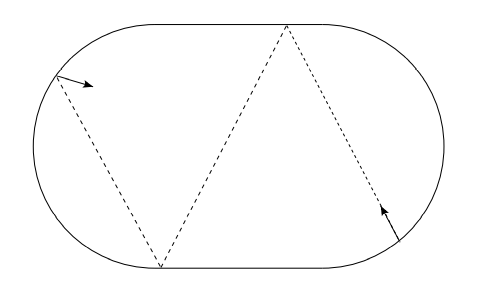}
     \caption{Stadium billiard}
     \label{F3}
   \end{minipage}\hfill
   \begin{minipage}{0.6\textwidth}
     \includegraphics[width=.6\linewidth]{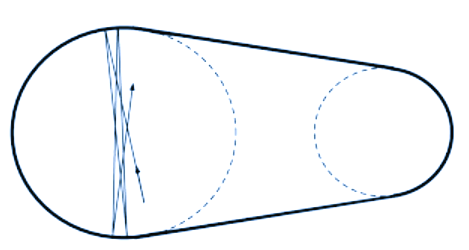}
     \caption{Squash billiard}\label{F4}
   \end{minipage}
\end{figure}

Let for a (“straight”) stadium $X\subseteq \mathcal{M}$ be the region where the first collisions (in a series of consecutive collisions with one and the same circular arc) with circular arcs occur. Denote by $R$ the first return time to $X$ for the billiard map $f$. Let $l$ be the length of each straight segment in the boundary of a billiard table. Without loss of generality, we may assume that the radius of circular arcs equals $1$.
\begin{enumerate}
\item $\bigcup_{i\ge 1}\partial \{x\in X:R(x)=i\}\subseteq \mathbb{S}$. Note that, if $R$ varies between $i$ and $i+1$, then a collision must occur at the endpoints of two circular arcs. It follows from (\ref{derivativeofbilliardmap}) that $\mathcal{K}$ has a jump after such collisions. The singularities of $f^R$ appear only because of this type of collisions with the boundary.
\item $B_r(q)\times [-\pi/2,\pi/2]$ is a quasi-section. Consider the map $\pi_X: \pi^{-1}\big(B_r(q)\times [-\pi/2,\pi/2]\big) \to X$. Let it be non-injective. Then there are $x_1,x_2\in \pi^{-1}\big(B_r(q)\times [-\pi/2,\pi/2]\big)$, such that $\pi_X x_1=\pi_X x_2=x\in X$. This implies that there are $0\le n_1<n_2 <R(x)$, such that $f^{n_1}x=x_1$, $f^{n_2}x=x_2$, and $d(q,\pi_{\partial Q}x_1)<r$, $d(q,\pi_{\partial Q}x_2)<r$, where $x$ is a point of the first collision (in a series) with one of the circular arcs.

Suppose that $q$ belongs to a circular arc, and $r$ is sufficiently small. Then $\pi_{\partial}x_1$ and $\pi_{\partial}x_2$ belong to a $r$-neighborhood of $q$.  So the orbits of $x_1$ and of $x_2$ are sliding along the same circular arc, and $f^{n_2-n_1}x_1=x_2$. Hence the angle $|\phi|$ for $x_1$ (and $x_2$) is greater or equal to $\pi/2-2r/(2n_2-2n_1)\ge \pi/2-r$. Therefore, a ``non-injective" configuration is $\{(q',\phi)\in \mathcal{M}: d(q,q')<r, |\phi|>\pi/2-r\}$. Clearly, it has a measure of order $O(r^3)$. 

Suppose now that $q$ belongs to the flat part of the boundary, and that $r$ is sufficiently small. Since the collisions at $x_1, x_2$ occur after the first collision at $x$ on a circular arc and there is no reflection off another circular arc yet, then $x_1, x_2$ must be bouncing on the boundary component, which contains $q$. Since the radius of the circular arc is $1$, and at these points  the angle of reflection $|\phi|$ is the same, then $|\phi|$ does not exceed $\arctan [2r/(2n_2-2n_1)]\le \arctan(r)$. Therefore, the set of ``non-injective" configurations in this case is $\{(q',\phi)\in \mathcal{M}: d(q,q')<r, |\phi|<\arctan(r)\}$, which has a measure of order $O(r^2)$. 

Summarizing the arguments above, we have that the set of ``non-injective" configurations has a measure of order $O(r^2)$. Therefore $B_r(q)\times [-\pi/2,\pi/2]$ is a quasi-section.
     
\item H\"older continuity along small (un)stable manifolds. The reason here  for using  ``small" manifolds is that in Definition \ref{unstablemfd} (un)stable manifolds are supposed to be maximal, and may not be H\"older continuous. Therefore, besides the singularities $\mathbb{S}_1$ of $f^R$, other points in $X$ have to be added into $\mathbb{S}$ (see Remark \ref{singular}). We define $\mathbb{S}$ by $(f^R)^{-1}\mathbb{S}_1\bigcup \mathbb{S}_1 \bigcup f^R\mathbb{S}_1$, and define  ``small" (un)stable manifolds in the same way as that in Definition \ref{unstablemfd}. 

First we consider a stable manifold with the stable cone \begin{align*}
    C^s(x):=\{dx=(dq, d\phi)\in \mathcal{T}_{x}\mathcal{M}:   & B^{+}(x)\in [-\frac{1}{\cos \phi},0] \text{ if } \mathcal{K}(q)=-1,\\
    & B^{+}(x)\in [-\infty, 0] \text{ if } \mathcal{K}(q)=0\}.
\end{align*} 

Suppose that the first collision on a circular arc is at a point $x_0 \in X$, and its stable manifold is $\gamma^s(x_0)$. Several cases must be considered.
\begin{enumerate}
    \item Sliding along a circular arc, i.e., the points $x_0, x_1, \cdots, x_k$ $(k\ge 0)$ belong to the same circular arc, and $x_{k+1}$ does not. Then $dx_k \in C^{s}(x_k) $, i.e., 
    \begin{align*}
        &B^+(x_k)=\frac{-2}{\cos \phi_k}+\frac{1}{\tau_{k-1}+\frac{1}{B^+(x_{k-1})}}\in [-\frac{1}{\cos \phi_{k}},0]\\
        & \implies B^+(x_{k-1})\in [\frac{-1}{\tau_{k-1}-\cos \phi_{k}}, \frac{-1}{\tau_{k-1}-\frac{\cos \phi_k}{2}}]=[\frac{-1}{\cos \phi_{0}}, \frac{-1}{\frac{3\cos \phi_{0}}{2}}].
    \end{align*}

Inductively, we have for any $i\in [0,k]$,\begin{gather*}
        B^{+}(x_i)\in \big[\frac{-1}{\cos \phi_{0}}, \frac{-2(k-i)}{[2(k-i)+1]\cos \phi_{0}}\big],
        \end{gather*} and for any $i\in [0,k)$,
        \begin{gather*}
        |1+\tau_i B^+(x_i)|\le 1.
    \end{gather*}
    
    By (\ref{pesudoderivative}) and (\ref{rnderivative}), for any $i\in [0,k]$, \begin{gather*}
        \frac{||dx_i||_p}{||dx_0||_p}=\prod_{j\le i-1}|1+\tau_j B^+(x_j)|\le 1,\\
        \frac{||dx_i||}{||dx_0||}=\frac{||dx_i||_p}{||dx_0||_p} \frac{\cos \phi_0}{\cos \phi_i} \frac{\sqrt{1+(\frac{d\phi_{i}}{dq_{i}})^2}}{\sqrt{1+(\frac{d\phi_{0}}{dq_{0}})^2}}\precsim 1,
    \end{gather*}where the last $``\precsim"$ holds thanks to the fact that on a circular arc $\cos \phi_0=\cos \phi_i$, and a slope of the stable manifold is uniformly bounded. Therefore, $f^i|_{\gamma^s(x_0)}$ is Lipschitz in this case. 
    \item Suppose that $x_{k+1}, \cdots, x_{k+n}$, $(n>1)$ are on the flat sides, and $x_{k+n+1}$ is on another circular arc. Then \begin{align*}
        &B^+(x_{k+n+1})=\frac{-2}{\cos \phi_{k+n+1}}+\frac{1}{\tau_{k+n}+\frac{1}{B^+(x_{k+n})}}\in [-\frac{1}{\cos \phi_{k+n+1}},0]\\
        & \implies B^+(x_{n+k})\in [\frac{-1}{\tau_{n+k}-\cos \phi_{n+k+1}}, \frac{-1}{\tau_{n+k}-\frac{\cos \phi_{n+k+1}}{2}}].
    \end{align*} Observe now that $\mathcal{K}(q_{n+k})=0$. Then \begin{align*}
        &B^+(x_{k+n})=\frac{0}{\cos \phi_{k+n}}+\frac{1}{\tau_{k+n-1}+\frac{1}{B^+(x_{k+n-1})}}\in [\frac{-1}{\tau_{n+k}-\cos \phi_{n+k+1}}, \frac{-1}{\tau_{n+k}-\frac{\cos \phi_{n+k+1}}{2}}]\\
        & \implies B^+(x_{n+k-1})\in [\frac{-1}{\tau_{n+k}+\tau_{n+k-1}-\cos \phi_{n+k+1}}, \frac{-1}{\tau_{n+k}+\tau_{n+k-1}-\frac{\cos \phi_{n+k+1}}{2}}].
    \end{align*}
    
    Inductively, we have for any $0\le i< n$\begin{gather*}
        B^+(x_{n+k-i})\in [\frac{-1}{\sum_{j\le i}\tau_{n+k-j}-\cos \phi_{n+k+1}}, \frac{-1}{\sum_{j\le i}\tau_{n+k-j}-\frac{\cos \phi_{n+k+1}}{2}}],\\
     |1+\tau_{n+k-i}B^+(x_{n+k-i})|\le 1,
    \end{gather*}where the last inequality is due to the SFC-condition, i.e., $\tau_{n+k}\ge 2\cos \phi_{n+k+1}$.
    
    By (\ref{pesudoderivative}), for any $i\in [k+1,k+n]$, \begin{gather*}
        \frac{||dx_i||_p}{||dx_0||_p}=\prod_{j\le i-1}|1+\tau_j B^+(x_j)|\le 1.
        \end{gather*}
    
    Since $n>1$, then $|\phi_i|$ is uniformly bounded away from $\pi/2$. Thus
        \begin{gather*}
        \frac{||dx_i||}{||dx_0||}=\frac{||dx_i||_p}{||dx_0||_p} \frac{\cos \phi_0}{\cos \phi_i} \frac{\sqrt{1+(\frac{d\phi_{i}}{dq_{i}})^2}}{\sqrt{1+(\frac{d\phi_{0}}{dq_{0}})^2}}\precsim 1,
    \end{gather*}where the last $``\precsim"$ holds due to the fact that a slope of a stable manifold is uniformly bounded. Therefore, $f^i|_{\gamma^s(x_0)}$ in this case is Lipschitz too.
    \item Consider now transitions between circular arcs through a flat side, i.e., $x_k$ is the last in a series collision on a circular arc, and $x_{k+1}$ is on a flat side, while $x_{k+2}$ is on another circular arc. We have \begin{gather*}
        \frac{||dx_{k+1}||_p}{||dx_0||_p}=\prod_{j\le k}|1+\tau_j B^+(x_j)|\le 1,\\
        \frac{||dx_{k+1}||}{||dx_0||}=\frac{||dx_{k+1}||_p}{||dx_0||_p} \frac{\cos \phi_0}{\cos \phi_{k+1}} \frac{\sqrt{1+(\frac{d\phi_{k+1}}{dq_{k+1}})^2}}{\sqrt{1+(\frac{d\phi_{0}}{dq_{0}})^2}}\precsim \frac{\cos \phi_0}{\cos \phi_{k+1}}= \frac{\cos \phi_k}{\cos \phi_{k+1}}.
        \end{gather*}
        
        For each $\phi_k$ a possible minimum of $\cos \phi_{k+1}$ (i.e., the maximum of $\phi_{k+1}$) satisfies the relation \begin{gather*}
            \cos^2 \phi_k= (1-l^2) \cos^2 \max \phi_{k+1}+2l\cos \max \phi_{k+1} \sin \max \phi_{k+1},
        \end{gather*}which holds if $q_{k+1}$ is at the end point of this flat side (or at the end point of another circular arc). Thus \begin{gather*}
            \frac{||dx_{k+1}||}{||dx_0||}\precsim_l \frac{1}{\cos \phi_k} \precsim k.
        \end{gather*}
        
        Since $\diam \gamma^s(x_0)=O(1/k^2)$, we have \begin{gather*}
        \diam f^{k+1}\big(\gamma^s(x_0)\big) \precsim k \diam \gamma^s(x_0) \precsim [\diam \gamma^s(x_0)]^{1/2}.\end{gather*}
\end{enumerate}

The arguments above show that there exists $C>0$, such that for any $\gamma^s\subseteq \big(\bigcup_{i \ge -1}(f^R)^{-i}\mathbb{S}_1\big)^c$, \begin{align*}
        d_{f^j\gamma^s}(f^jx, f^jy)\le C d_{\gamma^s}(x,y)^{1/2} \text{ for all }j<R(x).
    \end{align*}
    
Next we turn to unstable manifolds, and to the  unstable cone field \begin{align*}
    C^u(x):=\{dx=(dq, d\phi)\in \mathcal{T}_{x}\mathcal{M}:   & B^{+}(x)\in [-\frac{2}{\cos \phi},-\frac{1}{\cos \phi}] \text{ if } \mathcal{K}(q)=-1,\\
    & B^{+}(x)\in [0, \infty] \text{ if } \mathcal{K}(q)=0\}.
\end{align*} 

Suppose that the first collision in a series on a circular arc is  at $x_0 \in X$. The unstable manifold at this point is $\gamma^u(x_0)$. Again we will consider several cases.
\begin{enumerate}
    \item Sliding along a circular arc, i.e., $x_0, x_1, \cdots, x_k$ $(k\ge 0)$ belong to one and the same circular arc, while $x_{k+1}$ does not. Then $dx_0 \in C^{u}(x_0) $, i.e., 
    \begin{gather*}
        B^+(x_1)=\frac{-2}{\cos \phi_1}+\frac{1}{\tau_{0}+\frac{1}{B^+(x_{0})}},\quad  B^+(x_0)\in [-\frac{2}{\cos \phi_0},-\frac{1}{\cos \phi_0}],\\
         \implies B^+(x_{1})\in [-\frac{4}{3\cos \phi_{1}}, -\frac{1}{\cos \phi_{1}}].
    \end{gather*}
    
    Inductively, we have for any $i\in [0,k]$,\begin{gather*}
        B^{+}(x_i)\in \big[-\frac{2i+2}{(2i+1)\cos \phi_i},-\frac{1}{\cos \phi_{i}}\big],
        \end{gather*} and for any $i\in [0,k)$,
        \begin{gather*}
        |1+\tau_i B^+(x_i)|\le \frac{2i+3}{2i+1}.
    \end{gather*}
    
    By (\ref{pesudoderivative}) and (\ref{rnderivative}), for any $i\in [0,k)$, \begin{gather*}
        \frac{||dx_i||_p}{||dx_0||_p}=\prod_{j\le i-1}|1+\tau_j B^+(x_j)| \precsim i \le k,\\
        \frac{||dx_i||}{||dx_0||}=\frac{||dx_i||_p}{||dx_0||_p} \frac{\cos \phi_0}{\cos \phi_i} \frac{\sqrt{1+(\frac{d\phi_{i}}{dq_{i}})^2}}{\sqrt{1+(\frac{d\phi_{0}}{dq_{0}})^2}}\precsim  k,
    \end{gather*}where the last $``\precsim"$ holds because $\cos \phi_0=\cos \phi_i$ on a circular arc, and the slope of a stable manifold is uniformly bounded.
    
    Our small unstable manifold $\gamma^u(x_0)$ is contained in a connected component of the set $\big(\bigcup_{i\ge  -1}(f^R)^{i}\mathbb{S}_1\big)^c$. Thus the length of $\gamma^u(x_0)$ is $O(k^{-2})$. Hence, \begin{gather*}
       \diam f^i \big(\gamma^u(x_0)\big) \precsim k \diam \gamma^u(x_0) \precsim [\diam \gamma^u(x_0)]^{1/2}.
    \end{gather*}
    Therefore, $f^i|_{\gamma^s(x_0)}$ is H\"older continuous.
    \item Bouncing on flat sides. Here the points $x_{k+1}, \cdots, x_{k+n}$, $(n>1)$ are on flat sides, and $x_{k+n+1}$ belongs to another circular arc. \begin{gather*}
        B^+(x_{k+1})=\frac{0}{\cos \phi_{k+1}}+\frac{1}{\tau_{k}+\frac{1}{B^+(x_{k})}},\quad  B^+(x_{k})\in [-\frac{2k+2}{(2k+1)\cos \phi_{k}},-\frac{1}{\cos \phi_{k}}],\\
         B^+(x_{k+2})=\frac{0}{\cos \phi_{k+2}}+\frac{1}{\tau_{k+1}+\frac{1}{B^+(x_{k+1})}}=\frac{1}{\tau_{k+1}+\tau_k+\frac{1}{B^+(x_{k})}}.
    \end{gather*} 
    Inductively, we have for any $0\le i \le n$\begin{gather*}
          B^+(x_{k+i})=\frac{1}{\sum_{0\le j \le i-1}\tau_{k+j}+\frac{1}{B^+(x_{k})}},\\
      |1+\tau_{k+i}B^+(x_{k+i})|=\big|\frac{\sum_{0\le j \le i}\tau_{k+j}+\frac{1}{B^+(x_{k})}}{\sum_{0\le j \le i-1}\tau_{k+j}+\frac{1}{B^+(x_{k})}}\big|.
    \end{gather*}
    
    Since $n>1$, then $k$ may assume only a finite number of values. Then, it is bounded by a constant depending on $Q$ only, while $|\phi_k|, \cdots, |\phi_{k+n}|$ are bounded away from $\pi/2$ by a positive constant, which also depends only on $Q$. By (\ref{pesudoderivative}) for any $i\in [k+1,k+n]$, \begin{gather*}
        \frac{||dx_i||_p}{||dx_0||_p}=\prod_{j\le i-1}|1+\tau_j B^+(x_j)|\precsim \big|\frac{\sum_{0\le j \le i}\tau_{k+j}+\frac{1}{B^+(x_{k})}}{\frac{1}{B^+(x_k)}}\big| \precsim n,\\
        \frac{||dx_i||}{||dx_0||}=\frac{||dx_i||_p}{||dx_0||_p} \frac{\cos \phi_0}{\cos \phi_i} \frac{\sqrt{1+(\frac{d\phi_{i}}{dq_{i}})^2}}{\sqrt{1+(\frac{d\phi_{0}}{dq_{0}})^2}}\precsim n,
    \end{gather*}where the last $``\precsim"$ holds thanks to the fact that the slope of a stable manifold is uniformly bounded. 
    
    A small unstable manifold $\gamma^u(x_0)$ is contained in a connected component of $\big(\bigcup_{i\ge  -1}(f^R)^{i}\mathbb{S}_1\big)^c$. Thus the length of $\gamma^u(x_0)$ is $O(n^{-2})$. Hence \begin{gather*}
       \diam f^i \big(\gamma^u(x_0)\big) \precsim n \diam \gamma^u(x_0) \precsim [\diam \gamma^u(x_0)]^{1/2}.
    \end{gather*}
    Therefore $f^i|_{\gamma^s(x_0)}$ is H\"older continuous.
    \item Sliding on a flat side, i.e, $x_k$ corresponds to a point of the last collision with a circular arc, $x_{k+1}$ is on a flat side, and $x_{k+2}$ is on another circular arc. We have\begin{gather*}
     B^+(x_{k})\in [-\frac{2k+2}{(2k+1)\cos \phi_{k}},-\frac{1}{\cos \phi_{k}}],\\
     \implies \frac{||dx_{k+2}||_p}{||dx_{k+1}||_p}=|1+\tau_{k+1}B^+(x_{k+1})|=\Big|1+\frac{\tau_{k+1}}{\tau_k+\frac{1}{B^+(x_k)}}\Big|\ge \frac{\tau_k+\tau_{k+1}-\frac{2k+1}{2k+2}\cos \phi_k}{\tau_k-\frac{2k+1}{2k+2}\cos \phi_k}\ge 1,
  \end{gather*}where the last $``\ge"$ is due to the SFC-condition. Therefore \begin{align*}
      \frac{||dx_{k+2}||}{||dx_{k+1}||}\approx \frac{||dx_{k+2}||_p}{||dx_{k+1}||_p} \frac{\cos \phi_{k+1}}{\cos \phi_{k+2}}\succsim \frac{\cos \phi_{k+1}}{\cos \phi_{k+2}}\succsim \cos \phi_{k+2},
  \end{align*}where the argument for the last $``\succsim"$ is the same as that for the case of ``sliding on the flat sides" for stable manifolds. Since $\gamma^u(x_0)\subseteq  \big(\bigcup_{i \ge -1}(f^R)^{i}\mathbb{S}_1\big)^c$, then $f^{k+2}\gamma^u(x_0)$ belongs to a $k'$-sliding cell, $\diam f^{k+2}\big( \gamma^u(x_0)\big)=O(1/k'^2)$ and $\cos \phi_{k+2}\approx 1/k'$, so $
      \frac{||dx_{k+2}||}{||dx_{k+1}||}\succsim \cos \phi_{k+2} \approx 1/k'$. From Lemma 8.45 of \cite{CMbook} we have \begin{align*}
            \diam  f^{k+1}\big(\gamma^u(x_0) \big) \precsim k'\diam  f^{k+2}\big(\gamma^u(x_0) \big) \precsim [\diam f^{k+2}\big(\gamma^u(x_0) \big)]^{1/2}\precsim [\diam \gamma^u(x_0)]^{1/8}.
        \end{align*}
\end{enumerate}

The arguments above show that there is $C>0$ such that for any $\gamma^u\subseteq \big(\bigcup_{i \ge -1}(f^R)^{i}\mathbb{S}_1\big)^c$ \begin{align*}
        d_{f^j\gamma^u}(f^jx, f^jy)\le C d_{\gamma^u}(x,y)^{1/8} \text{ for all }j<R(x).
    \end{align*}
\end{enumerate}

Now we turn to a squash (or a ``skewed” stadium). Since two flat sides are not parallel, we can suppose that the angle between them is $\gamma>0$. Following \cite{CZnon}, define $X$ to be the same set as that for ``straight" stadiums. Verification of Assumption \ref{assumption} is the same as for a ``straight" stadium. Therefore we will skip the details and outline only the differences.
\begin{enumerate}
    \item $B_r(q)\times [-\pi/2,\pi/2]$ is a quasi-section. If $q$ is on the flat sides, then the angle of reflection increases or decreases by $\gamma$ between two consecutive bounces on flat sides. Then $B_r(q)\times [-\pi/2,\pi/2]$ is a section, provided that $r$ is much smaller than $\gamma$. 
    
    Let now a point $q$ belongs to a circular arc. The analysis in the case for orbits, sliding on these arcs, is the same as the one for a straight stadium. Another case is a bouncing on a circular arc. It is enough to consider this case for the bigger arc only. Let us estimate a measure of the set of ``non-injective" configurations. Suppose that $(p_1,\phi_1)=f^n(x)$ and $(p_2,\phi_2)=f^m(x)$, where $n<m$, $d(q,p_1)<r$, $d(q,p_2)<r$, $x=(p_0,\theta)\in X$, $n,m<R(x)$, and $\theta$ is sufficiently small. Assume that $r$ is so small that all points $q, p_0, p_1, p_2$ are on the bigger arc. From the elementary geometry we have $\phi_1=\phi_2=\theta$, $d(p_0, p_1)=4n\theta$, $d(p_0, p_2)=4m\theta$, and $d(p_1,p_2)=4(m-n)\theta$. Therefore, $d(p_1,p_2)=4(m-n)\theta<2r$, which implies that $\phi_1=\phi_2=\theta<r/2$. Thus the set of ``non-injective" configurations for bouncing on the circular arc is contained in $\{(p,\phi)\in \mathcal{M}: p\in B_r(q), |\phi|<r/2\}$, which has a measure of order $O(r^2)$. 
    \item Another difference is in verification of H\"older continuity. Besides the three cases studied for the (straight) stadiums, we also need to consider here a bouncing on the bigger circular arc. The argument there is exactly the same as that for ``sliding on a circular arc".
\end{enumerate}

We can conclude now this subsection by stating the following
\begin{corollary}
Theorem \ref{thm} holds for squash billiards (stadiums with unequal or equal focusing arcs).
\end{corollary}
\subsection{Another class of billiards with focusing components} In this section we consider billiard tables $Q$ for which each smooth component $\Gamma_i \subseteq \partial Q$ of the boundary is either dispersing, i.e., convex inwards, or focusing, i.e., convex
outwards. A curvature of every dispersing component is bounded
away from zero and infinity. We assume that every focusing component is an arc of a circle, and that there are no points of $\partial Q$ on that circle or inside it, other than the arc itself (that is, the SFC-condition). We assume also that two dispersing components intersect (if they do) transversally (i.e., there are no cusps) and, besides, every focusing arc is not longer than a half of the corresponding circle, e.g. see the Figure \ref{F5}.
Denote the union of dispersing components by $\partial Q^{+}$, and the union of focusing components by $\partial Q^{-}$. Let $X\subseteq \mathcal{M}$ be \begin{gather*}
    X:=(\partial Q^{+}\times [-\pi/2,\pi/2]) \bigcup \{x \in \mathcal{M}: \pi_{\partial Q}x \in \partial Q^{-}, \pi_{\partial Q}x \text{ and }\pi_{\partial Q}(f^{-1}x)\text{ belong to  different }\Gamma_i \},
\end{gather*}  i.e., only the first collisions on circular arcs and any collisions with the dispersing components are included in $X$. Therefore, the case with $R>1$ may occur only in the series of reflections off a circular arc.

\begin{figure}[!htb]
   \begin{minipage}{0.6\textwidth}
\includegraphics[width=.6\linewidth]{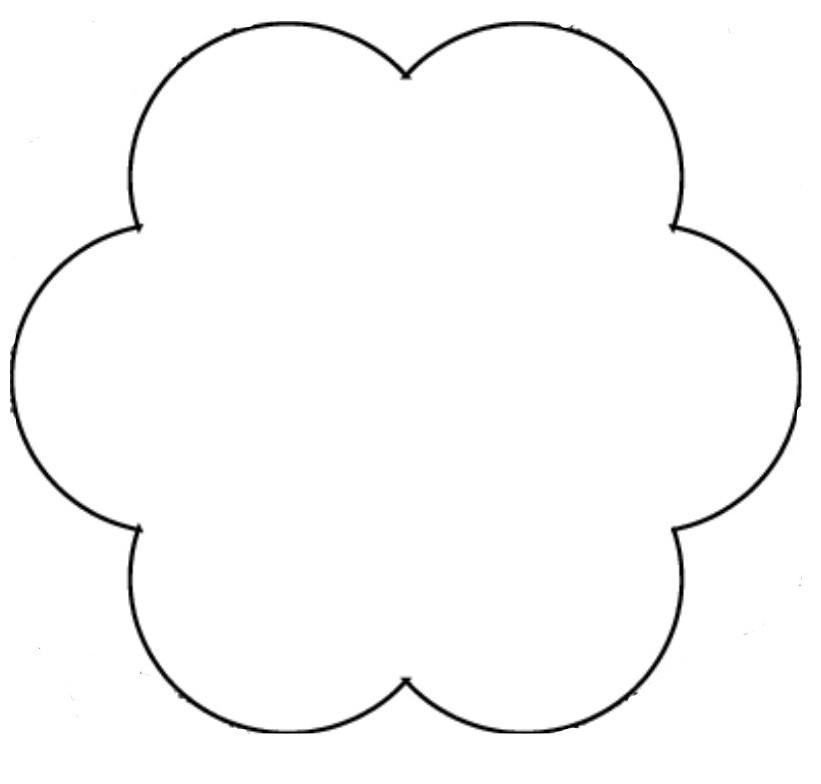}
     \caption{Flower billiard}
     \label{F5}
   \end{minipage}\hfill
   \begin{minipage}{0.48\textwidth}
     \includegraphics[width=.6\linewidth]{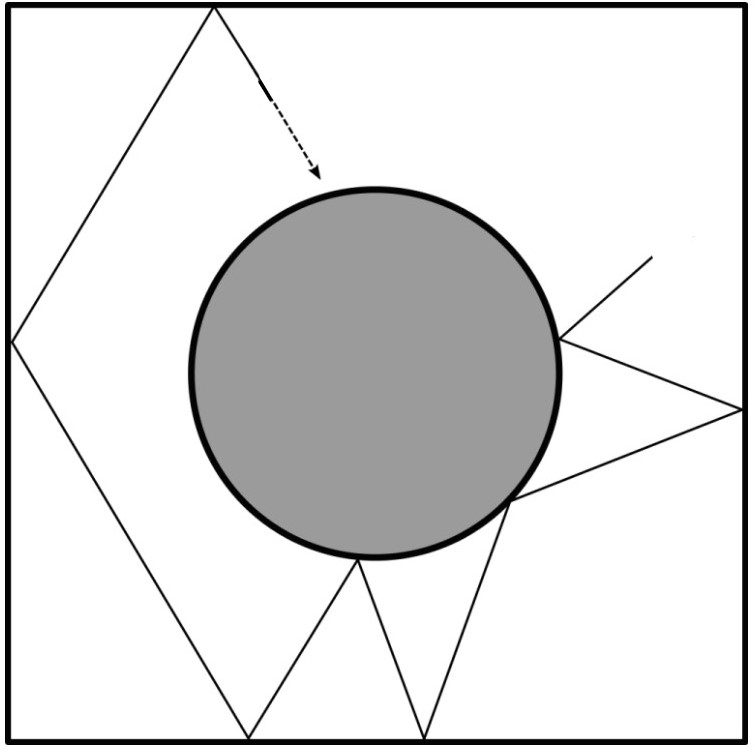}
     \caption{Semi-dispersing billiard}\label{F6}
   \end{minipage}
\end{figure}

The verification of Assumption \ref{assumption} is similar (actually it is easier) to the one for stadium billiards. Therefore we just outline below the differences.
\begin{enumerate}
    \item $B_r(q)\times [-\pi/2,\pi/2]$ is a quasi-section. If a point $q$ belongs to dispersing components, then clearly $B_r(q)\times [-\pi/2,\pi/2]$ is a section.  If $q$ belongs to circular arc, then $B_r(q)\times [-\pi/2,\pi/2]$ is a quasi-section. The argument in this case is exactly the same as for stadium billiards.
    \item H\"older continuity. Note that $R(x)>1$ only occurs if $\pi_{\partial Q}x\in \partial Q^{-}$. In this case, the orbit $\{x, f(x), f^2x \cdots f^{R-1}(x)\}$ is a series of consecutive reflections off a circular arc. Thus the argument for H\"older continuity is exactly the same as that for stadium-type billiards.
\end{enumerate}

Now we can conclude this subsection by stating the following
\begin{corollary}
Theorem \ref{thm} holds for the class of  billiards analyzed in this subsection.
\end{corollary}
\subsection{Semi-dispersing billiards} In this subsection we consider billiard tables of the following type. Let $R_0 \subseteq \mathbb{R}^2$ be a rectangle, and scatterers $B_1, \cdots, B_r \subseteq \interior R_0$ are open strictly convex sub-domains with smooth, (at least $C^3)$, or piece-wise smooth boundaries whose curvature is bounded away from zero,
and such that $B_i \bigcap B_j=\emptyset$ for $i\neq j$. The boundary of a billiard table $Q = R_0\setminus \bigcup_iB_i$ is partially dispersing (convex inwards) and partially neutral (flat), e.g. see Figure \ref{F6}. The flat part is $\partial R_0$.

Denote  by $\partial Q^{+}$ the union of dispersing components, and the union of four flat sides by $\partial R_0$. Let $X:=\{x\in \mathcal{M}: \pi_{\partial Q}x\in \partial Q^+\}$, where $R$ is the first return time to $X$. If $\sup R<\infty$, then this billiard system has exponential decay of correlations, i.e., it is not slowly mixing. So we assume that $\sup R=\infty$. Clearly $f^R$ is a billiard map of a Sinai billiard with an infinite horizon. Verification of the Assumption \ref{assumption} for such billiards has a few differences with the one for stadium billiards. Namely

\begin{enumerate}
    \item clearly, $B_r(q)\times [-\pi/2,\pi/2]$ is a section if $q \in \partial Q^+$. We will show that, if $q \in \partial R_0$, then $B_r(q)\times [-\pi/2,\pi/2]$ is a quasi-section. Without loss of generality, we may assume that $R_0=[0,1]^2$, $q=(0, c) \in \partial R_0$, where $c \in (0,1)$. Unfold now the bounded billiard table $Q$ to $\mathbb{R}^2$ by mirror reflections after collisions with the flat boundary. Then a billiard orbit, which is a broken line in $Q$, is lifted to become a straight line in $\mathbb{R}^2$, and scatterers $B_1, \cdots, B_r$ are lifted to generate a periodic configuration of scatterers in $\mathbb{R}^2$. (Note that this trick is the same as the one used for Sinai billiards with infinite horizon. Namely, the analysis of a billiard flow between two reflections off scatterers in $R_0$ reduces to consideration of straight segments between consecutive reflections off scatterers in $\mathbb{R}^2$). Suppose that a point $q$ is lifted in this way to the points $\{(p, c+k): p, k \in \mathbb{Z}\}$. 
    
    In order to prove that $B_r(q) \times [-\pi/2,\pi/2]$ is a quasi-section for sufficiently small $r>0$, we will study the ``non-injective" part, i.e., a measure of the configuration $(q_1, \phi_1) \in  B_r(q) \times [-\pi/2,\pi/2]$, which satisfies $f^n(q_1, \phi_1)=(q_2, \phi_2) $ for some $n\ge 1, (q_2, \phi_2) \in B_r(q) \times [-\pi/2,\pi/2]$. A billiard orbit, which is moving along the following set of points \begin{equation*}
        (q_1, \phi_1), f(q_1, \phi_1), \cdots, f^{n-1}(q_1, \phi_1), f^n(q_1, \phi_1)=(q_2, \phi_2),
    \end{equation*} does not intersect $\bigcup_{i} B_i$. The lifting to $\mathbb{R}^2$  has the following property. There are $p, k \in \mathbb{Z}$ (which depend on $\phi_1, n$) such that the line between $q_1\in B_r(q)$ and $q_2+(p,k)\in B_r(q+(p,k))$ does not intersect the periodic configuration of the scatterers, and the slope of this line is 
    \begin{equation*}
        \tan \phi_1=k/p+O(r)=\frac{k/\gcd(k,p)}{ p/\gcd(k,p)}+O(r).
    \end{equation*}
    
    Since this line intersects $B_r(q)$ and $B_r(q+(p,k))$, then it intersects  $B_r(q+(p/i,k/i))$ for any $i$, such that $i|p$ and $i|k$, in particular, $i=\gcd(p,k)$. Therefore every ``non-injective" configuration $(q_1, \phi_1)\in B_r(q) \times [-\pi/2,\pi/2]$ corresponds to unique direction vector $(p/\gcd(k,p), k/\gcd(k,p))$. Without loss of generality, we assume that $\gcd(k,p)=1$. Here $p$ means that the line between $q_1$ and $q_2+(p,k)$ intersects the vertical boundary $\approx p$ times, which implies that the billiard flow in $Q$ divides $R_0$ into several pieces with measures of order $O(1/p)$. Let $k$ mean that the line between $q_1$ and $q_2+(p,k)$ intersects  the horizontal boundary $\approx k$ times, which implies that the billiard flow in $Q$ divides $R_0$ into several pieces with measure of order $O(1/k)$. If $p$ or $k$ is larger than $M>0$ (which depends on the size of $\bigcup_iB_i$), then the billiard orbit, which is passing consecutively through the points \begin{equation*}
        (q_1, \phi_1), f(q_1, \phi_1), \cdots, f^{n-1}(q_1, \phi_1), f^n(q_1, \phi_1)=(q_2, \phi_2)
    \end{equation*}  must intersect $\bigcup_{i} B_i$. Therefore $\max\{p,k\}$ must not exceed $M$ for such billiard orbit, in order for it not to intersect $\bigcup_{i} B_i$. It implies that there are finitely many pairs $(p,k)$ with $\gcd(p,k)=1$ such that $\tan \phi_1=k/p+O(r)$. Then $\phi_1$ has a measure of order $O(r)$. So, a measure of ``non-injective" configuration $(q_1,\phi_1) \in B_r(q)\times [-\pi/2,\pi/2]$ is of order $O(r^2)$. Hence $ B_r(q)\times [-\pi/2,\pi/2]$ is a quasi-section.
    
    \item H\"older continuity.  Consider, at first, unstable manifolds in $X$ (which are, actually, homogeneous unstable manifolds, defined, e.g., in section 5.4 of \cite{CMbook}, see also Remark \ref{singular}). Suppose that $x_0 \in X$, and that $n:=R(x_0)$ is sufficiently large. Then the orbit of $x_0$ hits $\partial R_0$ many times before getting  back to $\partial Q^+$. Clearly for billiards of this type all possible angles $\phi_1$ are bounded by $\pi/2$. Consider now the worse case when $\phi_1\approx 0$, i.e., $|\phi_1|< \phi_Q$ for some small $\phi_Q$, (which depends only on exact shape of a billiard table under consideration), such that $\phi_1=\phi_2=\cdots=\phi_{j-1}=\phi_{j+1}=\cdots =\phi_{n-1}$, and $\phi_j=\pi/2-\phi_1 \approx \pi/2$ for some $2\le j< n$.  Then for any $i \neq j$ we have $\cos \phi_j\approx 1/n$, and $\cos \phi_i\approx 1$. Assume now that $f^n\big(\gamma^u(x_0)\big)$ belongs to a homogeneity strip (see its definition on page 536 in \cite{Chernovjsp}), for instance to $\cos \phi_n \approx 1/k^2$. Observe also that  $k \succsim n^{1/4}$ (see page 544 of \cite{Chernovjsp}).  It is known, see e.g. \cite{CMbook}, that $B^+(x_0)\in [\frac{2K(q_0)}{\cos \phi_0}, \frac{2K(q_0)}{\cos \phi_0}+O(1)]$. Thus we have\begin{gather*}
       \frac{1}{B^{+}(x_j)}=\frac{1}{B^{-}(x_i)}=\sum_{0\le i \le j-1 }\tau(f^ix_0) + \frac{1}{B^{+}(x_0)}\\
       \implies \frac{||dx_n||_p}{||dx_j||_p}=|1+\sum_{j \le i <n}\tau(f^ix)B^+(x_j)|=\frac{\sum_{0\le i \le n-1}\tau(f^ix_0) + \frac{1}{B^{+}(x_0)}}{\sum_{0\le i \le j-1 }\tau(f^ix_0) + \frac{1}{B^{+}(x_0)}} \ge 1\\
       \implies \frac{||dx_n||}{||dx_j||}\succsim \frac{\cos \phi_j}{\cos \phi_n}\frac{||dx_n||_p}{||dx_j||_p} \succsim \frac{\cos \phi_j}{\cos \phi_n} \approx \frac{1/n}{1/k^2} \succsim \frac{1}{k^2}\\
       \implies \diam f^j\big(\gamma^u(x_0)\big) \precsim k^2 \diam f^n\big(\gamma^u(x_0)\big) \precsim [\diam f^n\big(\gamma^u(x_0)\big)]^{1/3},
         \end{gather*}where the last $``\precsim"$ is due to $\diam f^n\big[\gamma^u(x_0)\big]=O(1/k^3)$. We also have that
       \begin{gather*}
       \frac{||dx_n||_p}{||dx_0||_p}=|1+\sum_{j <n}\tau(f^jx)B^+(x_0)|\approx \frac{n}{\cos \phi_0}\\
       \implies \frac{||dx_n||}{||dx_0||}\approx \frac{\cos \phi_0}{\cos \phi_n} |1+\sum_{j <n}\tau(f^jx)B^+(x_0)|\approx \frac{n}{\cos \phi_n}.
    \end{gather*} 
    
    Let now $s=\cos \phi_n$. Then $|ds|\approx |d\phi_n|\approx |dx_n|$. Hence, we get\begin{gather*}
        \diam \gamma^u(x_0)=\int_{f^n \gamma^u(x_0)} \frac{s}{n}|dx_n|\approx \int_{f^n \gamma^u(x_0)} \frac{s}{n}|ds| \succsim \frac{1}{n} (\int_{f^n \gamma^u(x_0)}|ds|)^2 \approx \frac{[\diam f^n \gamma^u(x_0)]^2}{n}.
    \end{gather*} 
    
    Since $\diam \gamma^u(x_0) =O(1/n^2)$, then \begin{gather*}
        \diam f^n \gamma^u(x_0) \precsim [n  \diam \gamma^u(x_0)]^{1/2} \precsim  [\diam \gamma^u(x_0)]^{1/4}.
    \end{gather*}
    
    By combining now all the arguments above we obtain \begin{gather*}
         \diam f^j \gamma^u(x_0) \precsim  [\diam f^n \gamma^u(x_0)]^{1/3}  \precsim  [\diam \gamma^u(x_0)]^{1/12}.
    \end{gather*}
    
    If $i \in [0,n) \setminus \{j\}$, then, by making use of the relation $\diam \gamma^u(x_0) =O(1/n^2)$, we get \begin{gather*}
        \frac{||dx_i||_p}{||dx_0||_p}=|1+\sum_{j <i}\tau(f^jx)B^+(x_0)|\precsim \frac{n}{\cos \phi_0}\\
        \implies \frac{||dx_i||}{||dx_0||}\approx \frac{\cos \phi_0}{\cos \phi_i} |1+\sum_{j <n}\tau(f^jx)B^+(x_0)|\precsim \frac{n}{\cos \phi_i}\precsim n\\
        \implies  \diam f^i \gamma^u(x_0) \precsim n  \diam \gamma^u(x_0) \precsim  [\diam \gamma^u(x_0)]^{1/2}.
        \end{gather*}
        
        So far we proved H\"older continuity for the case when $\phi_1<\phi_Q$. Actually this argument is analogous to the one for a series of reflections off the flat sides in stadium billiards. If $\phi_1> \phi_Q$, then all $\phi_i$ $i\in [0, R(x_0))$ are uniformly bounded away from $0$ and $\pi/2$.  Then the argument is the same as for bouncing on the flat sides in stadium billiards. Actually, this case is much easier, and we skip a proof. Therefore we obtain  H\"older continuity of unstable manifolds. For stable manifolds the argument is similar and, basically, the same as  the one for stadium billiards. Thus we do not repeat it here. 
\end{enumerate}

Therefore, we obtain the following result.
\begin{corollary}
Theorem \ref{thm} holds for the class of semi-dispersing billiards considered in this subsection.
\end{corollary}
\begin{remark}
From the proof of this corollary, it could be seen that the singularities for this class of semi-dispersing billiards have, in a sense, a similar structure as the singularities in the stadium-type billiards. This is a reason why these two classes of billiards have the same rate of decay of correlations (see \cite{CZcmp, CZnon}).
\end{remark}

\begin{acknowledgements}
We thank the both anonymous referees for numerous comments and suggestions, which allowed to essentially improve readability of the paper. L.B. was partially supported by the NSF grant DMS-2054659.
\end{acknowledgements}

%
%

\medskip

\bibliography{bibtext}

\end{document}